\documentclass{amsart}
\usepackage{bc}

\begin{document}
\title[Van Kampen's Theorem for bisets]{Algorithmic aspects of branched coverings I.\\ Van Kampen's Theorem for bisets}
\author{Laurent Bartholdi}
\email{laurent.bartholdi@gmail.com}
\author{Dzmitry Dudko}
\email{dzmitry.dudko@gmail.com}
\address{\'Ecole Normale Sup\'erieure, Paris \emph{and} Mathematisches Institut, Georg-August Universit\"at zu G\"ottingen}
\thanks{Partially supported by ANR grant ANR-14-ACHN-0018-01 and DFG grant BA4197/6-1}
\date{December 28, 2015}
\begin{abstract}
  We develop a general theory of \emph{bisets}: sets with two
  commuting group actions. They naturally encode topological
  correspondences.

  Just as van Kampen's theorem decomposes into a graph of groups the
  fundamental group of a space given with a cover, we prove
  analogously that the biset of a correspondence decomposes into a
  \emph{graph of bisets}: a graph with bisets at its vertices, given
  with some natural maps. The \emph{fundamental biset} of the graph of
  bisets recovers the original biset.

  We apply these results to decompose the biset of a Thurston map (a
  branched self-covering of the sphere whose critical points have
  finite orbits) into a graph of bisets. This graph closely parallels
  the theory of Hubbard trees.

  This is the first part of a series of five articles, whose main goal
  is to prove algorithmic decidability of combinatorial equivalence of
  Thurston maps.
\end{abstract}
\maketitle

\section{Introduction}
This is the first of a series of five articles, and develops the theory of
decompositions of bisets. For an overview of the series,
see~\cite{bartholdi-dudko:bc0}.

\subsection{Bisets}
Bisets are algebraic objects used to describe continuous maps up to
isotopy. Classically, a map $f\colon(Y,\dagger)\to(X,*)$ between pointed
spaces induces a homomorphism between the fundamental groups
$f_*\colon\pi_1(Y,\dagger)\to\pi_1(X,*)$. If, however, $f$ does not
preserve any natural base points, it is much more convenient to consider a
weaker kind of relation between $\pi_1(Y,\dagger)$ and $\pi_1(X,*)$: this
is precisely a \emph{biset}, namely a set $B(f)$ with two commuting group
actions, of $\pi_1(Y,\dagger)$ on the left and of $\pi_1(X,*)$ on the
right. Bisets may naturally be multiplied, and the product of bisets
corresponds to composition of the maps.

Another advantage of bisets is that only bijective maps may be
inverted, while every biset $B$ has a \emph{contragredient} $B^\vee$,
obtained by dualizing the actions. Thus it is almost as easy to handle
correspondences $Y\overset{i}\leftarrow Z\overset{f}\rightarrow X$ as
genuine maps $Y\to X$, see~\S\ref{ss:corr}: the biset of the
correspondence is $B(i)^\vee\otimes B(f)$.

Van Kampen's theorem is a cornerstone in algebraic topology: given a space
$X$ covered by (say) open subsets $(U_\alpha)$, the theorem expresses the
fundamental group of a space in terms of the fundamental groups of the
pieces $U_\alpha$ and of their intersections. This is best expressed in
terms of a \emph{graph of groups}, namely a simplicial graph with groups
attached to its vertices and edges. The \emph{fundamental group} of a graph
of groups is algebraically constructed in terms of the graph data, and
recovers the fundamental group of $X$.

The main construction in this article is a \emph{graph of bisets}
expressing the decomposition of a continuous map, or more generally a
topological correspondence, between spaces given with compatible covers,
see Definition~\ref{defn:graphofbisets}. We single out the subclass of
\emph{fibrant} graphs of bisets. Graphs of bisets can be multiplied, and
the product of fibrant graphs of bisets is again fibrant.

We construct in Definition~\ref{defn:p1biset} the fundamental biset of a
graph of bisets, and show in Corollary~\ref{cor:NormalForm} how they can
be conveniently computed for fibrant graphs of bisets.

Our central result, Theorem~\ref{thm:vankampenbis}, is an analogue of van
Kampen's theorem in the language of bisets. It expresses the biset of a
correspondence as the fundamental biset of a graph of bisets constructed
from restrictions of the correspondence to elements of the cover.

\subsection{Applications}
We then apply these results, in~\S\ref{ss:dynamics}, to complex
dynamics. Following Thurston, we consider branched coverings of the
sphere, namely self-maps $f\colon S^2\righttoleftarrow$ that are
locally modeled on $z\mapsto z^d$ in complex charts. The subset of
$S^2$ at which the local model is $z^d$ with $d>1$ is called the
\emph{critical set} of $f$, and the \emph{post-critical set} $P(f)$ is
the strict forward orbit of the critical set. A \emph{Thurston map} is
a branched covering of the sphere for which $P(f)$ is finite. Examples
are rational maps such as $f(z)=z^2-1$.

The simplest example of all is the map $f(z)=z^d$ itself, with
$P(f)=\{0,\infty\}$. The map $f$ is a covering on
$\hC\setminus\{0,\infty\}$, and
$\pi_1(\hC\setminus\{0,\infty\},1)=\Z$. In this case, we can give the
biset of $f$ quite explicitly: it is $B(f)=\Z$, with right and left
actions given by $m\cdot b\cdot n=d m+b+n$. We call such bisets
\emph{regular cyclic bisets}, and we define a \emph{cyclic biset} as a
transitive biset over cyclic groups.

The next simplest maps are complex polynomials $f(z)$ with finite
post-critical set. They have been intensively studied, in particular
via their \emph{Hubbard tree},
see~\cites{douady-h:edpc1,douady-h:edpc2,poirier:trees}. This is a
dynamically-defined $f$-invariant tree containing $P(f)$ and embedded
in $\C$. We apply the van Kampen Theorem~\ref{thm:vankampenbis} to the
Hubbard tree and obtain in this manner a decomposition of the biset of
$f$ as a graph of cyclic bisets, see Theorem~\ref{thm:GrBsOutOfHT}.

One of the advantages in working with branched coverings rather than
rational maps is that surgery operations are possible. For example,
given a Thurston map $f$ with a fixed point $z$ mapping locally to
itself by degree $d>1$, and given a topological polynomial $g$ of
degree $d$, a small neighbourhood of $f$ may be removed to be replaced
by a sphere, punctured at $\infty$, on which $g$ acts. One calls the
resulting map the \emph{tuning} of $f$ by $g$.

This operation has a transparent interpretation in terms of graphs of
bisets: it amounts to replacing a cyclic biset, at a vertex of the
graph of bisets of $f$, with the biset $B(g)$; see
Theorem~\ref{thm:tuning}.

In case $f$ itself also has degree $d$, one calls the resulting map
$h$ the \emph{mating} of $f$ and $g$.  The sphere on which $h$ acts is
naturally covered by the spheres on which $f$, respectively $g$ act,
and the ``equator'' on which they overlap. Thus $h$ is naturally
expressed by a graph of bisets with two vertices corresponding to $f$
and $g$ respectively, and an edge corresponding to the equator; see
Theorem~\ref{thm:mating}.

Laminations allow Julia sets of polynomials to be obtained out of the
Julia set of $z^d$, namely a circle, by pinching. Similarly, van
Kampen's Theorem~\ref{thm:vankampenbis} applied to the lamination of
$f$ decomposes the biset $B(f)$ into a graph of bisets made of trivial
bisets and one regular cyclic biset (the biset of $z^d$). We compute
it explicitly for the map $f(z)=z^2-1$, and describe in this manner,
in~\S\ref{ss:laminations}, the mating of $z^2-1$ with an arbitrary
quadratic polynomial.

\subsection{Notation}
We introduce some convenient notations, which we will follow
throughout the series of articles. For a set $S$, we write $S\perm$
for its permutation group. This has the mnemonic advantage that
$\#(S\perm)=(\#S)!$. We always write $\one$ for the identity map from
a set to itself. The restriction of a map $f\colon Y\to X$ to a subset
$Z\subset Y$ is written $f\restrict Z$.  Self-maps are written
$f\colon X\selfmap$ in preference to $f\colon X \to X$.

Paths are continuous maps $\gamma\colon[0,1]\to X$; the path starts at
$\gamma(0)$ and ends at $\gamma(1)$. We write $\gamma\#\delta$ for
concatenation of paths, following first $\gamma$ and then $\delta$;
this is defined only when $\gamma(1)=\delta(0)$. We write
$\gamma^{-1}$ for the inverse of a path.

We write $\approx$ for isotopy of paths, maps etc, $\sim$ for
conjugacy or combinatorial equivalence, and $\cong$ for isomorphism of
algebraic objects.

A graphs of groups is a graph $\gf$ with groups associated with $\gf$'s
vertices and edges. We \emph{always} write $G_x$ for the group associated
with $x\in\gf$; thus if $\mathfrak Y$ is another graph of groups, we will
write $G_y$ for the group associated with $y\in\mathfrak Y$, and no
relationship should be assumed between $G_x$ and $G_y$. Similarly, in a
graph of bisets $\mathfrak B$ there are bisets $B_z$ associated with
$z\in\mathfrak B$, and different graphs of bisets will all have their bisets
written in this manner.

\section{Bisets}
We show, in this section, how topological data can be conveniently
converted to group theory. We shall extend, along the way, the
classical dictionary between topology and group theory.

Consider a continuous map $f\colon Y\to X$ between path connected topological
spaces, and basepoints $\dagger\in Y$ and $*\in X$. Denote by
$H=\pi_1(Y,\dagger)$ and $G=\pi_1(X,*)$ their fundamental groups. If
$f(\dagger)=*$, then $f$ induces a homomorphism $f_*\colon H\to G$; however, no
such natural map exists if $f$ does not preserve the basepoints.

A solution would be to express $f$ in the fundamental groupoid
$\pi_1(X)$, whose objects are $X$ and whose morphisms from $x$ to $y$
consist of all paths from $x$ to $y$ in $X$ up to homotopy rel their
endpoints. However, for computational purposes, a much more practical
solution exists: one chooses basepoints in $X$ and $Y$, and expresses
$f$ as an \emph{$H$-$G$-biset}.

\begin{defn}[Bisets]
  Let $H,G$ be two groups. An \emph{$H$-$G$-biset} is a set $B$ equipped
  with a left $H$-action and a right $G$-action that commute; namely,
  a set $B$ and maps $H\times B\to B$ and $B\times G\to B$, both
  written $\cdot$, such that
  \[h\cdot h'\cdot (b\cdot g g')=h h'\cdot(b\cdot g g')=(h h'\cdot b)\cdot g g'=(h h'\cdot b)\cdot g\cdot g',
  \]
  so that no parentheses are needed to write any product of $h$'s, $b$,
  and $g$'s. We will also omit the $\cdot$, and write the actions as
  multiplication.

  A \emph{$G$-biset} is a $G$-$G$-biset.

  An $H$-$G$-set $B$ is \emph{left-free} if, qua left $H$-set, it is
  isomorphic to $H\times S$ for a set $S$ (where, implicitly, the left
  action of $H$ is by multiplication on the first co\"ordinate); it is
  \emph{right-free} if, qua right $G$-set, $B$ is isomorphic to
  $T\times G$. It is \emph{left-principal}, respectively
  \emph{right-principal}, if furthermore $S$, respectively $T$ may be
  chosen a singleton, so that the respective action is simply
  transitive; more generally, it is left, respectively right
  \emph{free of rank $r$} if $\#S=r$, respectively $\#T=r$.
\end{defn}

Bisets should be thought of as generalizations of
homomorphisms. Indeed, if $\phi\colon H\to G$ is a group homomorphism, one
associates with it the $H$-$G$-set $B_\phi$, which, qua right $G$-set,
is plainly $G$; the left $H$-action is by
\[h\cdot b=h^\phi b.
\]
All bisets of the form $B_\phi$ are right-principal; and they are
left-free if and only if $\phi$ is injective. Note that we write the
evaluation of maps as `$f(x)$', unless the map is a group
homomorphism, in which case we write `$x^f$'.

Let us return to our continuous map $f\colon Y\to X$, but drop the
assumption $f(\dagger)=*$. The $H$-$G$-biset of $f$ is defined as
homotopy classes of paths rel their endpoints:
\begin{equation}\label{eq:DfnBisOfMap}
  B(f)=B(f,\dagger,*)=\{\gamma\colon[0,1]\to X\mid \gamma(0)=f(\dagger),\;\gamma(1)=*\}/{\sim}.
\end{equation}
For paths $\gamma,\delta\colon[0,1]\to X$ with $\gamma(1)=\delta(0)$, we
denote by $\gamma\#\delta$ their \emph{concatenation}, defined by
\[(\gamma\#\delta)(t)=\begin{cases}
  \gamma(2t) & \text{ if }0\le t\le\frac12,\\
  \delta(2t-1) & \text{ if }\frac12\le t\le1;
\end{cases}
\]
and, for a path $\gamma\colon[0,1]\to X$, its \emph{reverse} $\gamma^{-1}$
is defined by
\[\gamma^{-1}(t)=\gamma(1-t).\]
The left action of $H$ on $B(f)$ is, for a loop $\lambda$ in $Y$ based
at $\dagger$,
\[[\lambda]\cdot[\gamma]=[f\circ\lambda\#\gamma],\]
and the right action of $G$ is, for a loop $\mu$ in $X$ based at $*$,
\[[\gamma]\cdot[\mu]=[\gamma\#\mu].\]

It is then clear that $B(f)$ is a right-principal biset. If
$f(\dagger)=*$, then $B(f)$ is naturally isomorphic to $B_{f_*}$.

As we will see in~\S\ref{ss:corr}, bisets actually encode topological
correspondences, as generalizations of continuous maps.

\subsection{Morphisms}
\label{ss:BisMorph}
We consider three different kinds of maps between bisets.
\begin{defn}[Biset morphisms]
  Let ${}_H B_G$ and ${}_H{B'}_G$ be two $H$-$G$-bisets. A
  \emph{morphism} between them is a map $\beta\colon B\to B'$ such that
  \[h b^\beta g=(h b g)^\beta\quad\text{ for all }h\in H,b\in B,g\in G.
  \]
  The definitions of endomorphism, isomorphism, automorphism,
  monomorphism, and epimorphism are standard.

  Let now ${}_H B_G$ and ${}_{H'}{B'}_{G'}$ be two bisets. A
  \emph{congruence} is a triple $(\psi,\beta,\phi)$ of maps, with
  $\psi\colon H\to H'$ and $\phi\colon G\to G'$ group homomorphisms, and
  $\beta\colon B\to B'$ a map, such that
  \[h^\psi b^\beta g^\phi=(h b g)^\beta\quad\text{ for all }h\in H,b\in B,g\in G.
  \]
  The map $\beta$ itself is called a $(\psi,\phi)$-congruence; therefore a
  morphism is the same thing as a $(\one,\one)$-congruence. A congruence in
  injective if all $\psi,\beta,\phi$ are injective, and similarly for
  surjective etc.

  Consider finally a $G$-biset ${}_GB_G$ and a $G'$-biset
  ${}_{G'}{B'}_{G'}$. Apart from morphisms and congruences, a third
  (intermediate) notion relates them: a \emph{semiconjugacy} is a pair
  $(\phi,\beta)$ of maps, with $\phi\colon G\to G'$ and $\beta\colon B\to B'$,
  such that
  \[h^\phi b^\beta g^\phi=(h b g)^\beta\quad\text{ for all }g,h\in G,b\in B.
  \]
  In other words, $\beta$ is a $(\phi,\phi)$-congruence.  Note that we
  do not require, as is sometimes customary, that $\beta$ be
  surjective. A \emph{conjugacy} is an invertible semiconjugacy, i.e.\
  a semiconjugacy in which both $\phi$ and $\beta$ are invertible.
\end{defn}

Note now the following important, if easy, fact: the isomorphism class
of a biset $B_\phi$ remembers precisely the homomorphism $\phi$ up to
inner automorphisms. More precisely,
\begin{lem}
  Let $\phi,\psi\colon H\to G$ be two homomorphisms. Then the bisets
  $B_\phi$ and $B_\psi$ are congruent if and only if there exists an
  automorphism $\eta$ of $G$ such that $\psi=\phi\eta$; and they are
  isomorphic if and only if $\eta$ is inner.
\end{lem}
\begin{proof}
  We only prove the second assertion (``isomorphic bisets if and only
  if $\eta$ is inner'').

  Assume first $\psi=\phi\eta$, and let $\eta$ be conjugation by
  $g$. Then an isomorphism between the bisets $B_\phi$ and $B_\psi$ is
  given by $\beta\colon b\mapsto g^{-1}b$. Indeed,
  \[(h\cdot b)^\beta=(h^\phi b)^\beta=g^{-1}h^\phi b=h^{\phi\eta}g^{-1}b=h^\psi b^\beta=h\cdot b^\beta.\]

  Conversely, if $B_\phi$ and $B_\psi$ are isomorphic, let $\beta$ be
  such an isomorphism, and set $g=(1^\beta)^{-1}$. Because $\beta$
  commutes with the right $G$-action, which is principal, the map
  $\beta$ must have the form $b^\beta=1^\beta b=g^{-1}b$, and the
  same computation as above shows that $\psi=\phi\eta$, with $\eta$
  conjugation by $g$.
\end{proof}

\subsection{Products}
The \emph{product} of the $H$-$G$-biset $B$ with the $G$-$F$-biset
$C$ is the $H$-$F$-biset $B\otimes_G C$, defined as
\[B\otimes_G C=B\times C\big/\{(b,g\cdot c)=(b\cdot g,c)\text{ for all }b\in
B,g\in G,c\in C\};
\]
it is naturally an $H$-$F$-biset for the left $H$-action on $B$
and the right $F$-action on $C$. We have the easy
\begin{lem}\pushQED{\qed}
  Let $\phi\colon H\to G$ and $\psi\colon G\to F$ be group homomorphisms. Then
  \[B_{\phi\psi} \cong B_{\phi}\otimes_G B_\psi.\qedhere\]
\end{lem}
Note that we stick to the usual topological convention that $g\circ f$
is first $f$, then $g$; but we use the algebraic order on composition
of homomorphisms and bisets, so that $f g$ means `first $f$, then $g$'
for composable algebraic objects $f$ and $g$. In other words, we may
write $g\circ f=f g$.

The \emph{contragredient} of the $H$-$G$-biset $B$ is the $G$-$H$-biset
$B^\vee$, which is $B$ as a set (but with elements written $\check
b$), and actions
\[g\cdot \check b\cdot h=(h^{-1}\cdot b\cdot g^{-1})^{\displaystyle\check{}}.
\]
If $\phi\colon H\to G$ is invertible, then we have
$(B_\phi)^\vee=B_{\phi^{-1}}$.  In all cases, we have a canonical
isomorphism $(B\otimes C)^\vee=C^\vee\otimes B^\vee$.
\begin{rem}
  Answering a question of Adam Epstein, we may define
  $\Hom(B,C)=B^\vee\otimes C$, and then note that `$\otimes$' is the
  adjoint to this internal Hom functor, namely $\Hom(C\otimes
  B,D)=\Hom(B,\Hom(C,D))$ is natural.
\end{rem}

\noindent Products allow us to write congruences in terms of morphisms:
\begin{lem}\label{lem:congruencesmorphisms}
  Let $\psi\colon H\to H'$ and $\phi\colon G\to G'$ be homomorphisms,
  let $B$ be an $H$-$G$-biset, and let $B'$ be an
  $H'$-$G'$-biset. Then there is an equivalence between congruences
  $(\psi,\beta,\phi)\colon B\to B'$ and morphisms $\gamma\colon
  B_\psi^\vee\otimes B\otimes B_\phi\to B'$, in the following sense:
  there is a natural map $\theta\colon B\to B_\psi^\vee\otimes
  B\otimes B_\phi$ defined by $b\mapsto 1^\vee\otimes b\otimes 1$, and
  the congruence $(\psi,\beta,\phi)$ factors as
  $(\psi,\beta,\phi)=\theta\gamma$.\qed
\end{lem}
In particular, the bisets ${}_GB_G$ and ${}_{G'}B'_{G'}$ are conjugate
if and only if there exists an isomorphism $\phi\colon G\to G'$ with
$B'\cong B_\phi^\vee\otimes B\otimes B_\phi=B_{\phi^{-1}}\otimes
B\otimes B_{\phi}$.

\subsection{Transitive bisets}
We call an $H$-$G$-biset \emph{transitive} if it consists of a single
$H\times G$-orbit. If desired, transitive bisets may be viewed as quotients
of right-principal bisets (namely, bisets coming from homomorphisms), as
follows:
\begin{lem}\label{lem:biset=fibre}
  Let $B$ be a transitive $H$-$G$-biset. Then there exist a group $K$ and
  homomorphisms $\phi\colon K\to G$, $\psi\colon K\to H$ such that
  $B=B_\psi^\vee\otimes_K B_\phi$. Furthermore, there exists a unique
  minimal such $K$, in the sense that if $(K',\phi',\psi')$ also satisfy
  $B=B_{\psi'}^\vee\otimes_{K'}B_{\phi'}$ then there exists a homomorphism
  $\tau\colon K'\to K$ with $\tau\phi=\phi'$ and $\tau\psi=\psi'$:
  \[\begin{tikzpicture}[description/.style={fill=white,inner sep=2pt}]
    \matrix (m) [matrix of math nodes, row sep=3em,
    column sep=2.5em, text height=1.5ex, text depth=0.25ex]
    { & {K'}\\
      & {K}\\
      {H} && {G}\\};
    \path[->,font=\scriptsize]
    (m-1-2) edge[dotted] node[description] {$\tau$} (m-2-2)
    (m-2-2) edge node[description] {$\psi$} (m-3-1)
    (m-2-2) edge node[description] {$\phi$} (m-3-3)
    (m-1-2) edge[dotted] node[description] {$\psi'$} (m-3-1)
    (m-1-2) edge[dotted] node[description] {$\phi'$} (m-3-3);
  \end{tikzpicture}\]
\end{lem}
\begin{proof}
  Choose a basepoint $b\in B$, and define
  \[K=\{(h,g)\in H\times G\mid b g=h b\}.
  \]
  The homomorphisms $\psi,\phi$ are given by projection on the first,
  respectively second co\"ordinate. To construct an isomorphism
  $\beta\colon B_\psi^\vee\otimes B_\phi\to B$, set $(h^\vee\otimes
  g)^\beta\coloneqq h^{-1}b g$; note that this is well-defined because
  $h^\vee\otimes g=(h')^\vee\otimes g'$ if and only if there exists
  $(h'',g'')\in K$ with $g'=g''g$ and $h'=h''h$; and then
  $h^{-1}b g=(h')^{-1}b g'$.

  To prove the unicity of $K$, consider a group $K'$ with homomorphisms
  $\phi'\colon K'\to G$ and $\psi'\colon K'\to H$, and an isomorphism
  $\beta\colon B_{\psi'}^\vee\otimes_{K'}B_{\phi'}\to B$. Define then the
  map $\tau\colon K'\to K$ as follows. Write $\beta^{-1}(b)=(b_1,b_2)$.
  For $k'\in K'$, define $h\in H,g\in G$ by
  \[\beta(b_1,k' b_2)=h b=b g\text{ and then }\tau(k')=(h,g).\qedhere\]
\end{proof}

Note that the product of two bisets, when expressed as groups $K,L$ with
homomorphisms as in the lemma, is nothing but the fibre product of the
corresponding groups $K,L$.

If furthermore $B$ is left-free, then the construction can be made
even more explicit: the biset $B_\psi^\vee$ is a subbiset of $B$, and
$B_\phi$ is a left- and right-free. We summarize this in the
\begin{prop}\label{prop:DecompOfBisets}
  Let $B$ be a left-free transitive $H$-$G$-biset. Choose $b\in B$.
  Define
  \begin{itemize}
  \item $G_b\coloneqq\{g\in G\mid b g\in H b\}$, the right stabilizer of $b$;
  \item $D\coloneqq{}_{G_b}G_G$, the natural $G_b$-$G$ biset;
  \item $C^\vee\coloneqq H b$, an $H$-$G_b$ subbiset of ${}_H B_{G_b}$.
  \end{itemize}
  Then $C$ is a right-principal biset and $B\cong C^\vee\otimes_{G_b}
  D$.\qed
\end{prop}

\subsection{Combinatorial equivalence of bisets}\label{ss:combequiv}
Let ${}_GB_G$ be a left-free biset. For every $n\ge 0$ we have a right
$G$-action on $1\otimes_{G}B^{\otimes n}_G$, which gives a well defined
action on the set
\[T\coloneqq \bigsqcup_{n\ge0} 1\otimes_{G}B^{\otimes n}_G.
\]
The $G$-set $T$ naturally has the structure of a rooted tree, by putting an
edge from $1\otimes b_1\otimes\cdots\otimes b_n$ to $1\otimes
b_1\otimes\cdots\otimes b_n\otimes b_{n+1}$ for all $b_i\in B$. The action
of $G$ on $T$ is self-similar in an appropriate labeling of $T$, see
\cite{nekrashevych:ssg}*{Chapter~2}.

Set $G'\coloneqq G/\ker(\text{action})$; thus $G'$ is the
quotient of $G$ so that the induced action on $T$ is faithful. As in
\cite{nekrashevych:combinatorialmodels} we say that ${}_GB_G$ is
\emph{combinatorially equivalent} to
\[_{G'}B'_{G'} \coloneqq G' \otimes_{G} B\otimes_G G'.
\]
(The motivation is dynamical: if ${}_GB_G$ is contracting, then so is
${}_{G'}B'_{G'}$ and, moreover, the limit dynamical systems associated
with ${}_GB_G$ and ${}_{G'}B'_{G'}$ are topologically conjugate, see
e.g.~\cite{nekrashevych:ssg}*{Corollary~3.6.7}.)

We say that two left-free ${}_GB_G$ and ${}_H C_H$ are \emph{combinatorially
  equivalent} if ${}_{G'}B'_{G'}$ and ${}_{H'}C'_{H'}$ are conjugate.

It is easy to check that any surjective semi-conjugacy
$(\phi,\beta)\colon {}_GB_G\to {}_H C_H$ between left-free bisets
descents into a semi-conjugacy
$ (\phi',\beta')\colon {}_{G'}B'_{G'}\to {}_{H'}C'_{H'}$. We say that
a surjective semi-conjugacy $(\phi,\beta)\colon {}_GB_G\to {}_H C_H$
between left-free bisets \emph{respects combinatorics} if
$ (\phi',\beta')\colon {}_{G'}B'_{G'}\to {}_{H'}C'_{H'}$ is a
conjugacy (so that ${}_GB_G$ and $ {}_H C_H$ are combinatorially
equivalent).

\subsection{Biset presentations}\label{ss:wreath}
Let $B$ be a left-free $H$-$G$-biset. Choose a \emph{basis} of $B$,
namely a subset $S$ of $B$ such that $B$, qua left $H$-set, is
isomorphic to $H\times S$. In other words, $S$ contains precisely one
element from each $H$-orbit of $B$.

The structure of the biset is then determined by the right action in
that description. For $g\in G$ and $s\in S$, there are unique
$(h,t)\in H\times S$ such that $s g=ht$.  The choice of basis therefore
leads to a map $S\times G\to H\times S$, or, which is the same, a map
$\Phi\colon G\to(H\times S)^S$.

Associativity of the biset operations yields, as is easy to see, that
$\Phi$ is a group homomorphism $G\to H\wr S\perm$. This last group, the
\emph{wreath product} of $H$ with the symmetric group $S\perm$ on $S$, is
by definition the semidirect product of $H^S$ with $S\perm$, in which
$S\perm$ acts by permutations of the co\"ordinates in $H^S$. We call $\Phi$
a \emph{wreath map} of the biset $B$.

Wreath products are best thought of in terms of \emph{decorated
  permutations}: one writes permutations of $S$ in the standard arrow
diagram notation, but adds a label belonging to $H$ on each arrow in
the permutation. Permutations are multiplied by concatenating arrow
diagrams and multiplying the labels.

A \emph{presentation} of $B$ is a choice of generating sets $\Gamma,\Delta$
for $G,H$ respectively, and for each $g\in\Gamma$ an expression of the form
\[g=\pair{h_1,\dots,h_d}(i_1,\dots),\] describing $\Phi(g)$; the $h_i$ are
words in $\Delta$, and $(i_1,\dots)$ is a permutation of
$S\cong\{1,\dots,d\}$ in disjoint cycle format.

\begin{lem}\label{lem:conjwr}
  If $\Phi,\Psi\colon G\to H\wr d\perm$ are two wreath maps of the
  same biset ${}_H B_G$, then there exists $w\in H\wr d\perm$ such that
  $\Psi=\Phi\cdot(h\mapsto h^w)$.
\end{lem}
\begin{proof}
  Let $S,T$ be the bases of $B$ in which $\Phi,\Psi$ are computed, and
  identify $S,T$ with $\{1,\dots,d\}$ by writing $S=\{s_1,\dots,s_d\}$
  and $T=\{t_1,\dots,t_d\}$. In $B$, write $t_{i^\pi}=w_i\cdot s_i$
  for $i=1,\dots,d$, defining thus $w=\pair{w_1,\dots,w_d}\pi\in H\wr
  d\perm$.

  Consider an arbitrary $g\in G$, and write
  \[\Phi(g)=\pair{g_1,\dots,g_d}\sigma,\quad\Psi(g)=\pair{h_1,\dots,h_d}\tau.
  \]
  We therefore have $s_i g=g_is_{i^\sigma}$ and $t_i g=h_it_{i^\tau}$ for all $i$. Now
  \[h_{i^\pi}t_{i^{\pi\tau}}=t_{i^\pi}g=w_i^{-1}s_i g=w_i^{-1}g_is_{i^\sigma}=w_i^{-1}g_i w_{i^{\sigma\pi}}t_{i^{\sigma\pi}},\]
  so $\Phi(g)^w=\Psi(g)$.
\end{proof}

It may help to introduce an example here. Consider the map $f(z)=z^d$ from
the cylinder $X=\C\setminus \{0\}$ to itself. Choose $*=\dagger=1$ as
basepoints, and identify $\pi_1(X,*)$ with $\Z$ by choosing as generator
the loop $\tau(t)=\exp(2i\pi t)$.

The biset $B(f)$ is the set of homotopy classes of loops at $1$, and so may
be naturally identified with $\Z$. The biset structure is given by $n\cdot
b\cdot m=d n+b+m$ for $n,m\in\Z\cong\pi_1(X,*)$ and $b\in\Z\cong B(f)$. Thus
$B(f)$ is left-free of rank $d$, and a basis is a complete set of
congruence representatives modulo $d$, for example $\{0,1,\dots,d-1\}$.  In
that basis, the wreath map reads
\[\Phi(\tau)=\pair{1,\dots,1,\tau}(1,2,\dots,d),\]
or in diagram notation
\[\begin{tikzpicture}
  \node at (-1,0.5) {$\tau=$};
  \draw[->] (0,1) -- (1,0);
  \draw[->,dashed] (1,1) -- (2,0);
  \draw[->,dashed] (2,1) -- (3,0);
  \draw[->] (3,1) -- (4,0);
  \draw[->] (4,1) -- node[above]{$\tau$} (0,0);
\end{tikzpicture}\]

\subsection{Conjugacy classes in groups}\label{ss:ccgroups}
We will, in later articles, consider conjugacy classes so as to represent
unbased loops in topological spaces. There are various applications; for
example, a fundamental construction by Goldman~\cite{goldman:pi1lie}
defines a Lie bracket structure on the vector space spanned by conjugacy
classes in $\pi_1(X,*)$ for a closed surface $X$. We consider, here, the
general case of groups and bisets.

Let $G^G$ denote the set of conjugacy classes of the group $G$,
and consider a left-free $H$-$G$-biset $B$. Choose a basis $S$ of $B$,
whence a wreath map $\Phi\colon G\to H\wr S\perm$.  Consider
$g^G\in G^G$, write $\Phi(g)=\pair{h_1,\dots,h_d}\pi$, and let
$S_1,\dots,S_\ell$ be the orbits of $\pi$ on $S$.  For each
$j=1,\dots,\ell$, let $k_j$ be the product of the $h_i$'s along the
orbit $S_j$; namely, if $S_j=\{s_1,\dots,s_{d_j}\}$ with
$s_i^\pi=s_{i+1}$, indices being computed modulo $d_j$, then
$k_j=h_{s_1}h_{s_2}\cdots h_{s_{d_j}}$.  The multiset
$\{(d_j,k_j^H)\mid i=j,\dots,\ell\}$ consisting of degrees and
conjugacy classes in $H$ is called the \emph{lift} of $g^G$.
\begin{lem}
  The lift of $g^G$ is independent of all the choices made: of $g$ in
  its conjugacy class, of the basis $S$, and of the cyclic ordering of
  the orbit $S_j$.
\end{lem}
\begin{proof}
  Different choices of bases give conjugate wreath maps, by
  Lemma~\ref{lem:conjwr}.
\end{proof}

Let $\Q G^G$ denote the vector space spanned by the conjugacy classes
$G^G$, and consider again a left-free $H$-$G$-biset $B$. Then the lift
operation gives rise to a linear map
$B^*\colon\Q G^G\to\Q H^H$, defined on the basis by
\begin{equation}\label{eq:conjclasses}
  B^*(g^G)\coloneqq\sum_{(d_i,h_i^H)\in\text{lift of }g^G}\frac{k_j^H}{d_j},
\end{equation}
called the \emph{Thurston endomorphism} of $B$.

\section{Graphs of bisets}
We define here the fundamental notion in this article: viewing groups as
fundamental groups of spaces, bisets describe continuous maps between
spaces. Graphs of groups describe decompositions of spaces, and graphs of
bisets describe continuous maps that are compatible with the
decompositions.

\subsection{Graphs of groups}
\label{ss:GrOfGr}
We begin by precising the notion of graph we will use; it is close to
Serre's~\cite{serre:trees}, but allows more maps between graphs. In
particular, we allow graph morphisms that send edges to vertices.
\begin{defn}[Graphs]
\label{defi:Graph}
  A \emph{graph} $\gf$ is a set $\gf=V\sqcup E$,
  decomposed in two subsets called \emph{vertices} and \emph{edges}
  respectively, equipped with two self-maps $x\mapsto x^-$ and
  $x\mapsto\overline x$, and subject to axioms
  \begin{equation}\label{eq:gf}
    \forall x\in\gf:\qquad\overline{\overline x}=x,\quad x^-\in V,\quad
    \text{and}\quad x=x^-\Leftrightarrow x=\overline x\Leftrightarrow x\in V.
  \end{equation}
  The object $\overline x$ is called the \emph{reverse} of
  $x$. Setting $x^+=(\overline x)^-$, the vertices $x^-,x^+$ are
  respectively the \emph{origin} and \emph{terminus} of $x$.
\end{defn}

We write $V(\gf)$ and $E(\gf)$ for the sets of vertices and edges
respectively in a graph $\gf$. A \emph{path} is a sequence
$(e_1,\dots,e_n)$ of edges with $e_i^+=e_{i+1}^-$ for all
$i=1,\dots,n-1$. A graph is \emph{connected} if there exists a path joining
any two objects. A \emph{circuit} is a sequence $(e_1,\dots,e_n)$ of edges
with $e_i^+=e_{i+1}^-$ for all $i$, indices taken modulo $n$. A \emph{tree}
is a graph with no non-trivial circuits; that is in every circuit
$(e_1,\dots,e_n)$ one has $e_{i+1}=\overline{e_i}$ for some $i$.

A \emph{graph morphism} is a map $\theta\colon\mathfrak Y\to\gf$ satisfying
$\theta(\overline y)=\overline{\theta(y)}$ and
$\theta(y^-)=\theta(y)^-$ for all $y\in\mathfrak Y$. Note that $\theta$
maps the vertices of $\mathfrak Y$ to those of $\gf$. It is
\emph{simplicial} if furthermore $\theta$ maps the edges of $\mathfrak
Y$ to those of $\gf$.

A \emph{graph of groups} is a connected graph $\gf=V\sqcup E$, with a
group $G_x$ associated with every $x\in\gf$, and homomorphisms
$()^-\colon G_x\to G_{x^-}$ and $\overline{()}\colon G_x\to G_{\overline x}$ for
each $x\in\gf$, satisfying the same axioms as~\eqref{eq:gf}, namely
the composition $G_x\to G_{\overline x}\to G_{\overline{\overline
    x}}=G_x$ is the identity for every $x\in\gf$, and if $x\in V$ then
the homomorphisms $G_x\to G_{x^-}$ and $G_x\to G_{\overline x}$ are
the identity. For $g\in G_e$ we write $g^+={\overline g}^-$. The graph
of groups is still denoted $\gf$.

Let $\gf=V\sqcup E$ be a graph of groups. For $v,w\in V$, consider the
set
\begin{equation}\label{eq:gog_elements}
  \Pi_{v,w}=\{(g_0,x_1,g_1,\dots,x_n,g_n)\mid
x_i\in\gf,x_1^-=v,x_i^+=x_{i+1}^-,x_n^+=w,g_i\in G_{x_i^+}=G_{x_{i+1}^-}\},
\end{equation}
of (group-decorated) paths from $v$ to $w$. As a special case, if $v=w$
then we allow $n=0$ and $g_0\in G_v=G_w$. Say that two paths are
equivalent, written $\sim$, if they differ by a finite sequence of
elementary local transformations of the form
$(g h)\leftrightarrow(g,x,1,\overline x,h)$ for some $x\in\gf$ and $g,h\in
G_{x^-}$, or of the form $(g h^-,x,k)\leftrightarrow(g,x,h^+k)$ for some
$x\in\gf$ and $g\in G_{x^-},h\in G_x,k\in G_{x^+}$, or of the form
$(g,x,h)\leftrightarrow(g h)$ for some $x\in V$ and $g,h\in G_x$.

The product of two paths $(g_0,x_1,\dots,g_m)$ and
$(h_0,y_1,\dots,h_n)$ is defined if $x_m^+=y_1^-$, and equals
$(g_0,x_1,\dots,g_m h_0,y_1,\dots,h_n)$; this product
$\Pi_{u,v}\times\Pi_{v,w}\to\Pi_{u,w}$ is compatible with the
equivalence relation. The \emph{fundamental groupoid} $\pi_1(\gf)$ of
$\gf$ is a groupoid with object set $V$, and with morphisms between
$v$ and $w$ the set $\pi_1(\gf,v,w)=\Pi_{v,w}/{\sim}$ of equivalence
classes of paths from $v$ to $w$.

In particular, for $v\in V$, the \emph{fundamental group} $\pi_1(\gf,v)$ is
$\Pi_{v,v}/{\sim}$. If $\gf$ is connected, then $\pi_1(\gf,v)$ is up to
isomorphism independent of the choice of $v\in V$.

In fact, the elements $x_i$ in a path may be assumed to all belong to
$E$, and if $x_{i+1}=\overline{x_i}$ then the element $g_{i+1}$ may be
assumed not to belong to $(G_{x_{i+1}})^-$. Such paths are called
\emph{reduced}, and $\Pi_{v,w}/{\sim}$ may be identified with the set
of reduced paths from $v$ to $w$.

In a more algebraic language, the fundamental groupoid $\pi_1(\gf)$ is the
universal groupoid with object set $V$, and whose morphism set is generated
by $\gf\sqcup\bigsqcup_{v\in V}G_v$; the source and target of $x\in\gf$ are
$x^-$ and $x^+$ respectively; the source and target of $g\in G_v$ are $v$;
the relations are those of the $G_v$ as well as $v=1\in G_v$ for all $v\in
V$, and $x\overline x=1\in G_{x^-}$ and $g^-x=x g^+$ for all $x\in\gf$ and
$g\in G_x$.  The path $(g_0,x_1,\dots,g_n)$ is identified with $g_0x_1\dots
g_n$. The following property follows easily from the definitions:
\begin{lem}[e.g.\ Serre~\cite{serre:trees}*{\S5.2}]\label{lem:InjGraphGroups}
   Let $\mathfrak X$ be a graph of groups. If all morphisms $G_x\to
   G_{x^-}$ are injective, then all natural maps
   $G_x\to\pi_1(\mathfrak X,*)$ are injective.\qed
\end{lem}

\subsection{Decompositions and van Kampen's theorem}\label{ss:vk}
Graphs of groups, and their fundamental group, generalize decompositions of
spaces and their fundamental group. Here is a useful example of graphs of
groups:
\begin{exple}
  If $X$ be a path connected surface with at least one puncture, consider a
  graph $\gf$ drawn on $X$ which contains precisely one puncture in each
  face. Then $X$ deformation retracts to $\gf$, and the groups $\pi_1(X,*)$
  and $\pi_1(\gf,*)$ are isomorphic. Here we consider $\gf$ as a graph of
  groups in which all groups $G_x$ are trivial.
\end{exple}

\begin{defn}[Finite $1$-dimensional covers]\label{defn:1dimcovers}
  Consider a path connected space $X$, covered by a finite collection
  of path connected subspaces $(X_v)_{v\in V}$. It is a \emph{finite
    $1$-dimensional cover} of $X$ if
  \begin{itemize}
  \item for every $v,w\in V$ and for every path connected component $X'$ of
    $X_{u}\cap X_{v}$ there are an open neighborhood $\widetilde X'\supset
    X'$ and an $X_{w}\subset X'$ such that $X_{w}\hookrightarrow \widetilde
    X'$ is a homotopy equivalence;
  \item if $X_{u}\subseteq X_{v}\subseteq X_{w}$ then $u=v$ or $v=w$.
  \end{itemize}
  We order $V$ by writing $u<v$ if $X_{u}\varsubsetneqq X_{v}$.
\end{defn}

\begin{rem}\label{rem:OnFinitnOfDecomp}
  In a classical formulation of van Kampen's theorem it is assumed that all
  $X_v$ are open subsets of $X$. Then every curve $\gamma\subset X$ is a
  concatenation of finitely many curves $\gamma_i$ where each $\gamma_i$
  lies entirely in some $X_{v(i)}$.

  For our dynamical purposes it is more convenient to allow the $X_v$ to be
  closed subsets of $X$. Thus we add the condition that $X_w$ is homotopy
  equivalent to an open neighborhood of $X_w$. Under this assumption every
  curve $\gamma\subset X$ is, up to homotopy rel endpoints, a concatenation of
  finitely many curves $\gamma_i$ which all lie entirely in some
  $X_{v(i)}$.
\end{rem}

\begin{defn}[Graphs of groups from covers]\label{defn:gog_1dimcovers}
  Consider a path connected space $X$ with a $1$-dimensional cover
  $(X_v)_{v\in V}$. It has an associated graph of groups $\gf$, defined as
  follows. The vertex set of $\gf$ is $V$. For every pair $u<v$ there are
  edges $e$ and $\bar e$ connecting $u=e^-=\bar e^+$ and $v=e^+=\bar e^-$,
  and we let $E$ be the set of these edges. Set $\gf=V\sqcup E$.

  Choose basepoints $*_v\in X_v$ for all $v\in V$, and set
  $G_v\coloneqq\pi_1(X_v,*_v)$. For every edge $e$ with $e^-<e^+$
  choose a path $\gamma_e\colon[0,1]\to X_{e^+}$ from $*_{e^-}$ to
  $*_{e^+}$, set $G_e\coloneqq G_{e^-}$ and define
  $G_e\to G_{e^-}\coloneqq\one$ and $G_e\to G_{e^+}$ by
  $\gamma\mapsto\gamma_e^{-1}\#\gamma\#\gamma_e$. Finally for edges
  $e$ with $e^->e^+$ set $G_e\coloneqq G_{\bar e}$ and
  $\gamma_{\bar e}=\gamma_e^{-1}$ and define morphisms
  $G_e\to G_{e^-}$ and $G_e\to G_{e^+}$ as
  $G_{\bar e}\to G_{\bar e^+}$ and $G_{\bar e}\to G_{\bar e^-}$
  respectively.
\end{defn}

The graph of groups $\gf$ depends on the choice of basepoints $*_v$ and
paths $\gamma_e$, but mildly: we will show in Lemma~\ref{lem:gfCongClass}
that the congruence class of $\gf$ is independent of the choice of $*_v$
and $\gamma_e$.

Here is a simple example: on the left, the subspaces $X_v$ are the
simple loops and the two triple intersection points; on the right, the
corresponding graph of groups, with trivial or infinite cyclic groups.
\begin{center}\begin{tikzpicture}
  \draw (0,0) .. controls (0,1.5) and (-1.5,0) .. (0,0)
  .. controls (-1.5,0) and (0,-1.5) .. (0,0)
  .. controls (1,0.6) and (2,0.6) .. (3,0)
  .. controls (3,1.5) and (4.5,0) .. (3,0)
  .. controls (3,-1.5) and (4.5,0) .. (3,0)
  .. controls (2,-0.6) and (1,-0.6) .. (0,0);
  \begin{scope}[xshift=6cm,scale=0.6]
    \draw (-1,1) node {$\bullet$} node[left] {$\Z$} -- (0,0) node
    {$\bullet$} node[above right] {$1$} -- (-1,-1) node {$\bullet$}
    node[left] {$\Z$}; \draw (0,0) -- (2,0) node {$\bullet$}
    node[above=2pt] {$\Z$} -- (4,0) node {$\bullet$} node[above left]
    {$1$} -- (5,1) node {$\bullet$} node[right] {$\Z$} (4,0) -- (5,-1)
    node {$\bullet$} node[right] {$\Z$};
  \end{scope}
\end{tikzpicture}
\end{center}

\begin{thm}[van Kampen's Theorem]\label{thm:vankampen}
  Let $X$ be a path connected space with finite $1$-dimensional cover
  $(X_v)$ and a choice of basepoints $*_v\in X_v$ as well as paths
  connecting $*_{u}$ and $*_{v}$ if $X_{u}\varsubsetneqq X_{v}$. Let
  $\gf$ be the associated graph of groups.

  Then, for every $v$, the groups $\pi_1(X,*_v)$ and $\pi_1(\gf,v)$
  are isomorphic.
\end{thm}
\begin{proof}[Sketch of proof]
  The isomorphism $\theta\colon\pi_1(\gf,v)\to\pi_1(X,*_v)$ is defined by
  \[\theta(g_0,e_1,g_1,\dots,e_n,g_n)=g_0\#\gamma_{e_1}\#\dots \gamma_{e_n}\#g_n.\]
  Every loop $\gamma\colon[0,1]\to X$ with $\gamma(0)=\gamma(1)=*_v$ is
  homotopic to a concatenation of the above type, see
  Remark~\ref{rem:OnFinitnOfDecomp}; thus $\theta$ is surjective.  By the
  classical van Kampen argument (see~\cite{massey:at}*{Chapter~IV}), the
  path $\theta(g)$ is homotopic to $\theta(g')$ in $X$ if and only if $g$
  and $g'$ are equivalent in $\pi_1(\gf,v)$; thus $\theta$ is injective.
\end{proof}

In case $\gf$ is a tree, $\pi_1(\gf,*)$ is an iterated free product with
amalgamation of the $G_v$'s, amalgamated over the $G_e$'s.  It may be
constructed explicitly as follows: consider the set of finite sequences
over the alphabet $\bigcup_{v\in V}G_v$; identify two sequences if they
differ on subsequences respectively of the form $\{(1),()\}$; or
$\{(g,h),(g h)\}$ for $g,h$ in the same $G_v$; or $\{(g),(g)\}$ where the
first, respectively second, `$g$' denote the image of $g\in G_e$ in
$G_{e^-}$, respectively $G_{e^+}$; and quotient by the equivalence relation
generated by these identifications.

Note the following three operations on a graph of groups $\gf$, which
does not change its fundamental group:
\begin{description}
\item[(1) split an edge] choose an edge $e\in E$. Add a new vertex $v$
  to $V$, and replace $e,\overline e$ by new edges
  $e_0,\overline{e_0},e_1,\overline{e_1}$ with $e_0^+=e_1^-=v$ and
  $e_0^-=e^-$ and $e_1^+=e^+$. Define the new groups by
  $G_v=G_{e_0}=G_{e_1}=G_e$, with the obvious maps between them;
\item[(2) add an edge] choose a vertex $v\in V$, and a subgroup
  $H\subseteq G_v$. Add a new vertex $w$ to $V$, and add new edges
  $e,\overline e$ with $e^-=v$ and $e^+=w$. Define the new groups by
  $G_e=G_v=H$, with the obvious maps between them;
\item[(3) barycentric subdivision] construct a new graph $\gf'=V'\sqcup E'$
  with $V'=\gf/\{x=\overline x\}$ and $E'=E\times\{+,-\}$; for $e\in E$ and
  $\varepsilon\in\{\pm1\}$, set $(e,\varepsilon)^\varepsilon=e^\varepsilon$
  and $(e,\varepsilon)^{-\varepsilon}=[e]$ and
  $\overline{(e,\varepsilon)}=(\overline e,-\varepsilon)$. Define finally a
  graph of groups structure $(G'_{[x]})$ on $\gf'$ by setting $G'_{[x]}$ to
  be $G_x$ for all $[x]\in V'$ (recalling that $G_x$ and $G_{\bar x}$ are
  isomorphic), and $G'_{(e,\varepsilon)}=G_e$, with the obvious maps
  between these groups.
\end{description}

Recall that graph morphisms may send edges to vertices, but not vertices to
(midpoints of) edges. If such a map between the topological realizations of
the graphs is needed, it may be expressed as a map between their
barycentric subdivisions.

If $\gf$ be a graph of groups, let $\gf'$ denote the underlying graph
of $\gf$, but with trivial groups. Then, on the one hand,
$\pi_1(\gf')$ is the usual fundamental groupoid of $\gf'$, and
$\pi_1(\gf',v)$ is a free group; on the other hand, there exists a
natural morphism of groupoids $\pi_1(\gf)\to\pi_1(\gf')$.

If $\gf,\mathfrak Y$ be graphs of groups, $\theta\colon\mathfrak Y\to\gf$ be a
graph morphism, and $\theta_y$ be, for every $y\in\mathfrak Y$, a
homomorphism $G_y\to G_{\theta(y)}$ such that
$\theta_y(g)^-=\theta_{y^-}(g^-)$ and
$\overline{\theta_y(g)}=\theta_{\overline y}(\overline g)$ for all $g\in
G_y$, then there is an induced morphism of groupoids $\pi_1(\mathfrak
Y)\to\pi_1(\gf)$.  We do not give a name to these kinds of maps, because
they are too restrictive; see the discussion in the introduction
of~\cite{bass:covering}.  We will give later, in~\S\ref{ss:examplesgfgps},
more examples of graphs of groups; we now describe in more detail the
precise notion of morphisms between graphs of groups that we will use.

\subsection{Graphs of bisets}
Just as bisets generalize appropriately to our setting homomorphisms
between groups, so do ``graphs of bisets'' also generalize the notion
of homomorphisms between graphs of groups. In fact, our notion also
extends the definition of morphisms between graphs of groups given
in~\cite{bass:covering}, in that a ``right-principal graph of bisets''
is the same thing as a morphism between graphs of groups in the sense
of Bass.

\begin{defn}[Graph of bisets]\label{defn:graphofbisets}
  Let $\gf,\mathfrak Y$ be two graphs of groups. A \emph{graph of
    bisets} ${}_{\mathfrak Y}\mathfrak B_\gf$ between them is the
  following data:
  \begin{itemize}
  \item a graph $\mathfrak B$, not necessarily connected;
  \item graph morphisms $\lambda\colon\mathfrak B\to\mathfrak Y$ and
    $\rho\colon\mathfrak B\to\gf$;
  \item for every $z\in\mathfrak B$, a
    $G_{\lambda(z)}$-$G_{\rho(z)}$-biset $B_z$, a congruence
    $()^-\colon B_z\to B_{z^-}$ with respect to the homomorphisms
    $G_{\lambda(z)}\to G_{\lambda(z)^-}$ and $G_{\rho(z)}\to
    G_{\rho(z)^-}$, and a congruence $\overline{()}\colon B_z\to B_{\overline
      z}$ with respect to the homomorphisms $G_{\lambda(z)}\to
    G_{\overline{\lambda(z)}}$ and $G_{\rho(z)}\to
    G_{\overline{\rho(z)}}$.

    These homomorphisms must satisfy the same axioms as~\eqref{eq:gf},
    namely: the composition $B_z\to B_{\overline z}\to
    B_{\overline{\overline z}}=B_z$ is the identity for every
    $z\in\gf$, and if $z\in V$, then the homomorphisms $B_z\to B_{z^-}$
    and $B_z\to B_{\overline z}$ are the identity. For $b\in B_z$ we
    write $b^+={\overline b}^-$.
  \end{itemize}
  We call $\mathfrak B$ a \emph{$\mathfrak Y$-$\gf$-biset}.
\end{defn}

\begin{exple}\label{ex:map=biset}
  A graph morphism $\theta\colon\mathfrak Y\to\gf$, given with homomorphisms
  $\theta_y\colon G_y\to G_{\theta(y)}$ for all $y\in\mathfrak Y$ such that $()^{-} \circ \theta_y= \theta_{y^-}\circ ()^{-}$, naturally
  gives rise to a $\mathfrak Y$-$\gf$-biset $\mathfrak B_\theta$. Its
  underlying graph is $\mathfrak B=\mathfrak Y$; the maps are
  $\lambda=\one$ and $\rho=\theta$. For every $z\in\mathfrak B$, the biset
  $B_z$ is $G_{\theta(z)}$ with its natural right $G_{\theta(z)}$-action,
  and its left $G_z$-action is $g\cdot b=\theta_z(g)b$.
\end{exple}

\begin{defn}[Product of graphs of bisets]
  Let $\mathfrak B$ be a $\mathfrak Y$-$\gf$-biset, and let $\mathfrak C$
  be a $\gf$-$\mathfrak W$-biset. The \emph{product} of $\mathfrak B$ and
  $\mathfrak C$ is the following graph of bisets ${}_{\mathfrak Y}\mathfrak
  D_{\mathfrak W}=\mathfrak B\otimes_\gf\mathfrak C$: its underlying graph
  is the fibre product
  \[\mathfrak D=\{(b,c)\in\mathfrak B\times\mathfrak C\mid \rho(b)=\lambda(c)\},
  \]
  with $(b,c)^-=(b^-,c^-)$ and $\overline{(b,c)}=(\overline
  b,\overline c)$. The map $\lambda\colon\mathfrak D\to\mathfrak Y$ is
  $\lambda(b,c)=\lambda(b)$, and the map $\rho\colon\mathfrak D\to\mathfrak
  W$ is $\rho(b,c)=\rho(c)$. The biset $B_{(b,c)}$ is
  $B_b\otimes_{G_{\rho(b)}}B_c$.
\end{defn}

\noindent Properties of the product will be investigated
in~\S\ref{ss:PropOfTensProd}. For now we content ourselves with some
illustrative examples.

\begin{lem}[See Example~\ref{ex:map=biset}]
  Let $\theta\colon\mathfrak Y\to\gf$ and $\kappa\colon\gf\to\mathfrak W$ be
  graph morphisms, given with homomorphisms $\theta_y\colon G_y\to
  G_{\theta(y)}$ and $\kappa_x\colon G_x\to G_{\kappa(x)}$. Then their
  bisets satisfy $\mathfrak B_\theta\otimes_\gf \mathfrak
  B_\kappa=\mathfrak B_{\theta\kappa}$.
\end{lem}
\begin{proof}
  We have natural identifications $G_y\otimes_{G_x}G_x=G_y$ between
  the bisets of $\mathfrak B_\theta\otimes_\gf \mathfrak B_\kappa$ and
  those of $\mathfrak B_{\theta\kappa}$.
\end{proof}

The contragredient $\mathfrak B^\vee$ of the graph of bisets
$\mathfrak B$ is defined by exchanging $\rho$ and $\lambda$, and
taking the contragredients of all bisets $B_z$. If $\theta\colon\mathfrak
Y\to\gf$ is bijective and all $\theta_y$ are isomorphisms, then the
contragredient of $\mathfrak B_\theta$ is $\mathfrak B_{\theta^{-1}}$.

The identity biset for products is the following biset
${}_\gf\mathfrak I_\gf$: its underlying graph is $\gf$, with maps
$\lambda=\rho=\one$, and bisets $B_z=G_z$.

\begin{exple}[See Lemma~\ref{lem:biset=fibre}; a partial converse to
  Example~\ref{ex:map=biset}]
  Let ${}_{\mathfrak Y}\mathfrak B_\gf$ be a graph of bisets such that the
  graph $\mathfrak B$ is connected and $B_z$ is transitive for every $z\in
  \mathfrak B$. Suppose that for every $z\in\mathfrak B$ we are given an
  element $b_z\in B_z$ such that $\overline{b_z}=b_{\overline z}$ and
  $(b_z)^-=b_{z^-}$ for all $z\in\mathfrak B$. Then there exist a graph of
  groups $\mathfrak Z$ and morphisms of graphs of groups
  $\lambda\colon\mathfrak Z\to\mathfrak Y$ and $\rho\colon\mathfrak
  Z\to\gf$ as in Example~\ref{ex:map=biset} such that $\mathfrak B$ is isomorphic to $\mathfrak
  B_\lambda^\vee\otimes\mathfrak B_\rho$.

  Note, however, that it is not always possible to make coherent choices of
  basepoints $b_z\in B_z$. For example, it may happen that $(B_e)^{-}\cap
  (B_f)^-=\emptyset$ for two edges $e,f$ with $e^-=f^-$.
 
  Here is a dynamical situation in which this happens. Let $f$ be a
  rational map whose Julia set is a Sierpi\'nski carpet, and assume that
  its Fatou set has two invariant components $\mathcal U,\mathcal V$,
  perforce with disjoint closures, on which $f$ acts as $z\mapsto z^m,z^n$
  respectively. Remove these components, and attach on each of them a
  sphere with a self-map that has an invariant Fatou component mapped to
  itself by the same degree. To these data correspond a graph of groups
  $\gf$ with three vertices (the main sphere and the two grafted ones), and
  a graph of bisets ${}_\gf\mathfrak B_\gf$ describing the self-map. The
  basepoint on the main sphere cannot be chosen to coincide with the
  basepoints of the grafted spheres, because the Fatou components of the
  Sierpi\'nski map have disjoint closures.
\end{exple}

\subsection{The fundamental biset}
\label{ss:FundBis}
\begin{defn}[Fundamental biset of graph of bisets]\label{defn:p1biset}
  Let $\mathfrak B$ be a $\mathfrak Y$-$\gf$-biset; choose $*\in\gf$
  and $\dagger\in\mathfrak Y$. Write $G=\pi_1(\gf,*)$ and
  $H=\pi_1(\mathfrak Y,\dagger)$. The \emph{fundamental biset} of
  $\mathfrak B$ is an $H$-$G$-biset $B=\pi_1(\mathfrak B,\dagger,*)$,
  constructed as follows.
  \begin{equation}\label{eq:freebiset}
    B=\frac{\bigsqcup_{z\in V(\mathfrak B)}\pi_1(\mathfrak Y,\dagger,\lambda(z))\otimes_{G_{\lambda(z)}}B_z\otimes_{G_{\rho(z)}}\pi_1(\gf,\rho(z),*)}{\left\{q b^-p = q\lambda(z)b^+\overline{\rho(z)}p\quad\forall\begin{array}{c}q\in\pi_1(\mathfrak Y,\dagger,\lambda(z)^-),b\in B_z,\\p\in\pi_1(\gf,\rho(z)^-,*),z\in E(\mathfrak B)\end{array}\right\}}.
  \end{equation}
  In other words, elements of $B$ are sequences
  $(h_0,y_1,h_1,\dots,y_n,b,x_1,\dots,g_{m-1},x_n,g_n)$ subject to the
  equivalence relations used previously to define $\pi_1(\gf)$, as
  well as
  $(y_n,h b^-g,x_1)\leftrightarrow(y_n,h,\lambda(z),b^+,\overline{\rho(z)},g,x_1)$
  for all $z\in\mathfrak B$, $b\in B_z$, $h\in G_{\lambda(z)^-}$,
  $g\in G_{\rho(z)^-}$.
\end{defn}

In particular, if $\mathfrak Y=\gf$ are graphs with one vertex and no edges
and associated groups $H,G$ respectively, the definition simplifies a
little: then $B$ is the quotient of $\bigsqcup_{z\in V(\mathfrak B)}B_z$ by
the relation identifying $b^-$ with $b^+$, for all edges $z\in\mathfrak B$
and all $b\in B_z$, with respective images $b^\pm$ in $B_{z^\pm}$.

If $f\colon Y\to X$ be a continuous map and the spaces $X,Y$ admit
decompositions giving a splitting of their fundamental group as a graph of
groups, as in Theorem~\ref{thm:vankampen}, and if $f$ is compatible with
the decompositions of $X$ and $Y$, then $B(f)$ will decompose into a graph
of bisets, as we will see in Theorem~\ref{thm:vankampenbis}.

Let $\gf$ be a graph of groups. Choose $*\in\gf$, and consider the graph
$\{*\}$ with a single vertex and no edge. We treat it as a graph of groups,
with group $\pi_1(\gf,*)$ attached to $*$, and denote it by
$\gf_0$. Consider now the following biset ${}_\gf(\gf,*)_{\gf_0}$: its
underlying graph is $\gf$; the maps are $\lambda=\one$ and $\rho(z)=*$ for
all $z\in\gf$; and the bisets are $B_z=\pi_1(\gf,z^-,*)$; for vertices,
this is just $\pi_1(\gf,z,*)$, while for edges this is a left $G_z$-set via
the embedding $G_z\to G_{z^-}$. The embedding $B_z\to B_{z^-}$ is the
identity, while the map $B_z\to B_{\overline z}$ is $b\mapsto \overline
z b$.
\begin{lem}\pushQED{\qed}
  Write $G=\pi_1(\gf,*)$. Then we have an isomorphism of
  $G$-$G$-bisets
  \[\pi_1((\gf,*),*,*)\cong {}_G G_G.\qedhere\]
\end{lem}

The bisets $(\gf,*)$ and $(\mathfrak Y,\dagger)$ enable us,
using~\eqref{eq:freebiset}, to write $\pi_1(\mathfrak B,\dagger,*)$ more
simply. We will prove Lemma~\ref{lem:CharOfFindBiset}
in~\S\ref{ss:PropOfTensProd}:
\begin{lem}\label{lem:CharOfFindBiset}
  Let ${}_{\mathfrak Y}\mathfrak B_\gf$ be a graph of bisets, and
  choose $*\in\gf,\dagger\in\mathfrak Y$. Then $\pi_1(\mathfrak
  B,\dagger,*)=\pi_1((\mathfrak Y,\dagger)^\vee\otimes\mathfrak B\otimes(\gf,*),\dagger,*)$.
\end{lem}
\noindent Note, in particular, the expression
$\pi_1((\gf,\dagger)^\vee\otimes_\gf(\gf,*))=\pi_1(\gf,\dagger,*)$ of the
set of paths from $\dagger$ to $*$ in $\gf$ in terms of graphs of bisets.

\subsection{Fibrant and covering bisets}
We now define a class of graphs of bisets for which the fundamental biset
admits a convenient description.

\begin{defn}[Fibrant and covering graph of bisets]\label{dfn:FibrationGrBis}
  Let $\mathfrak B$ be a $\mathfrak Y$-$\gf$-graph of bisets. Then
  $\mathfrak B$ is a \emph{left-fibrant} graph of bisets if $\rho\colon \mathfrak
  B\to\gf$ is a simplicial graph-map, and for every vertex $v\in \mathfrak
  B$ and every edge $f\in \gf$ with $f^-=\rho(v)$ the map
  \begin{equation}\label{eq:dfn:FibrationGrBis}
    \bigsqcup_{\substack{e\in \rho^{-1}(f)\\ e^-=v}} G_{\lambda(v)}
    \otimes_{G_{\lambda(e)}} B_e\rightarrow B_v,\qquad g\otimes b\mapsto g b^-
  \end{equation}
  is a $G_{\lambda(v)}$-$G_f$ biset isomorphism, for the action of $G_f$ on
  $B_v$ given via $()^-\colon G_f\to G_{\rho(v)}$.

  If in addition $B_z$ is a left-free biset for every object $z\in \gf$,
  then $\mathfrak B$ is a \emph{covering}, or \emph{left-free}, graph
  of bisets.
  
  If furthermore $B_z$ is left-principal for every $z\in \mathfrak B$
  and $\rho \colon \mathfrak B\to \mathfrak Y$ is a graph isomorphism,
  then $\mathfrak B$ is \emph{left-principal}.  Right-fibrant, -free,
  and -principal graphs of bisets are defined similarly. For example,
  the biset associated with a morphism of groups as in
  Example~\ref{ex:map=biset} is right-principal.
\end{defn}  
We think about~\eqref{eq:dfn:FibrationGrBis} as a unique lifting
property: we may always rewrite
$p b q\in \pi_1(\mathfrak B,\dagger,*)$ as
$p' b'\in\pi_1(\mathfrak,\dagger,*)$ in an essentially unique way, as
we will see in Corollary~\ref{cor:NormalForm}. We will show in
Lemma~\ref{lem:DefnPrincipGrBis} that for left-principal graphs of
bisets~\eqref{eq:dfn:FibrationGrBis} follows automatically.
  
For every object $y\in \mathfrak Y$ let us denote by $\widehat G_y$
the image of $G_y$ in the fundamental groupoid $\pi_1(\mathfrak Y)$;
we write $g\mapsto \widehat g$ the associated natural quotient map
$G_y\to\widehat G_y$. We define $\widehat G_x$ and $g\to\widehat g$
similarly for $x\in\mathfrak \gf$.  Finally, for $z\in \mathfrak B$ we
set
\[\widehat B_z\coloneqq \widehat G_{\lambda(z)}\otimes_{G_{\lambda(z)}} B_z \otimes_{G_{\rho(z)}}  \widehat G_{\rho(z)}\]
and denote by $b\mapsto\widehat b$ the natural quotient map $B_z\to
\widehat B_z$.

Let $\widehat{\mathfrak B}$ be the graph of bisets obtained from $\mathfrak
B$ by replacing all $G_x, G_y, B_z$ by $\widehat G_x,\widehat G_y,\widehat
B_z$ respectively.  It is immediate that $\pi_1( { \mathfrak B}, \dagger ,
* )$ and $\pi_1(\widehat{\mathfrak B}, \dagger ,*)$ are naturally
isomorphic via $y b x\mapsto \widehat y\widehat b\widehat x$.
\begin{thm}\label{thm:FaithfGrOfBis}
  Suppose that $\mathfrak B$ is a left-fibrant biset. Then
  \begin{enumerate}
  \item \label{en0:thm:FaithfGrOfBis} $\widehat{\mathfrak B}$ is left-fibrant;
  \item \label{en1:thm:FaithfGrOfBis} $\widehat B_z \cong \widehat
    G_{\lambda(z)} \otimes_{ G_{\lambda(z)}} B_z$ via the natural map
    $1\otimes b\otimes 1 \leftarrow 1 \otimes b$;
  \item \label{en2:thm:FaithfGrOfBis} $()^-\colon\widehat B_e\to\widehat
    B_{e^-}$ are monomorphisms for all $e\in\mathfrak B$; and
  \item \label{en3:thm:FaithfGrOfBis} if $\mathfrak B$ is left-free,
    respectively left-principal, then so is $\widehat{\mathfrak B}$.
\end{enumerate}
\end{thm} 
\begin{proof}
  For $z\in\mathfrak B$ let us write $B'_z\coloneqq\widehat G_{\lambda(z)}\otimes
  B_z$, viewed as a right $G_{\rho(z)}$-set. For $x\in \gf$ let us write
  \[C_x\coloneqq\bigsqcup_{z\in\rho^{-1}(x)}B'_z,
  \]
  viewed as a right $G_x$-set. In particular, if $f\in\gf$ is an edge,
  then~\eqref{eq:dfn:FibrationGrBis} gives a bijection
  \begin{equation}\label{eq:RedOf:dfn:FibrationGrBis}
    \bigsqcup_{e\in\rho^{-1}(f)}\big(\widehat G_{\lambda(e^-)}/\widehat
    G_{\lambda(e)}\big)\times B'_e\to C_{f^-}
  \end{equation}
  of right $G_f$-sets. Finally, for $x\in\gf$ we denote by $\Aut(C_x)$ the
  set of \emph{pure automorphisms} of $C_x$:
  \[\Aut(C_x)\coloneqq \{\phi=(\phi_z)_{z\in\rho^{-1}(x)}\mid
  \phi_z\in(B'_z)\perm,\;(g b)^{\phi_z}=g\,b^{\phi_z}\;\forall b\in B'_z,
  g\in \widehat G_{\lambda(z)}\}.
  \]
  The right action of $G_x$ on $C_x$ induces a homomorphism
  \[\theta_x\colon G_{x}\to\Aut(C_x).\]

  \begin{lem}\label{lem:Aut(C_x)}
    Suppose $\mathfrak B$ is a left-fibrant biset. Then for every edge
    $f\in \gf$ there is a natural embedding
    \begin{equation}\label{eq:lem:Aut(C_x)}
      ()^-\colon\begin{cases}\Aut(C_f) \hookrightarrow \Aut(C_{f^-})\\
        (\phi_e)_{e\in\rho^{-1}(f)} \mapsto  \left[(g,b)\in\big(\widehat G_{\lambda(e^-)}/\widehat
    G_{\lambda(e)}\big)\times B'_e\mapsto(g,b^{\phi_e})\right]
      \end{cases}
    \end{equation}
    under the identification of $C_{f^-}$ given
    by~\eqref{eq:RedOf:dfn:FibrationGrBis}. Moreover, $()^-$ embeds
    $\theta_f(G_f)$ into $\theta_{f^-}(G_{f^-})$.
  \end{lem}
  \begin{proof}
    The contents of~\eqref{eq:RedOf:dfn:FibrationGrBis} that we use is that
    every $G_f$-set occurring in $C_f$ occurs as a summand of $C_{f^-}$,
    with multiplicity $\ge1$. Therefore $()^-$ is injective. Since the map
    $()^-\colon G_f\to G_{f^-}$ is equivariant with respect to the actions
    on $C_f$ and $C_{f^-}$, we get $\theta_{f^-}(G_{f^-})\cong (\theta_f(G_f))^-\le
    \theta_{f^-}(G_{f^-})$.
  \end{proof}
 
  Let $\gf'$ be the graph of groups obtained from $\gf$ by replacing each
  $G_x$ with $\theta_{x}(G_x)$, with maps $()^-$ given
  by~\eqref{eq:lem:Aut(C_x)}. By Lemma~\ref{lem:InjGraphGroups}, every
  $\theta_{x}(G_x)$ embeds into the fundamental groupoid $\pi_1(\gf')$. We
  also have an epimorphism $\theta\colon\pi_1(\gf)\to\pi_1(\gf')$ induced
  by the epimorphisms $\theta_x\colon G_x\to\theta_x(G_x)$.  Therefore,
  $\theta_x\colon G_x\to\theta_x(G_x)$ descends to a map $\widehat
  \theta_x\colon \widehat G_x\to\theta_x(G_x)$. The action of $\widehat
  G_{\rho(z)}$ on $\widehat B_z$ therefore lifts to $B'_z$, and this
  completes the proof of Claim~\eqref{en1:thm:FaithfGrOfBis}.

  Multiplying~\eqref{eq:dfn:FibrationGrBis} on the left by $\widehat
  G_{\lambda(v)}$ gives then Claim~\eqref{en0:thm:FaithfGrOfBis} as well as
  Claim~\eqref{en2:thm:FaithfGrOfBis}.
 
  Finally, if $B_z$ is left-free then so is $\widehat
  G_{\lambda(z)}\otimes_{G_{\lambda(z)}} B_z=B'_z\cong \widehat B_z$, and
  similarly if $B_z$ is left-principal then so is $\widehat B_z$; this
  finishes the proof of Theorem~\ref{thm:FaithfGrOfBis}.
\end{proof}

\subsection{Canonical form for fibrant bisets}
We now derive a canonical expression for element in a left-fibrant
biset. Let us set
\begin{multline*}
  S\coloneqq\bigsqcup_{z\in V(\mathfrak B)}\pi_1(\mathfrak Y,\dagger,\lambda(z)) \otimes_{\widehat G_{\lambda(z)}} \widehat B_z \otimes_{\widehat G_{\rho(z)}} \pi_1(\gf,\rho(z),*)
  \\ \cong\bigsqcup_{z\in V(\mathfrak B)}\pi_1(\mathfrak Y,\dagger,\lambda(z))\otimes_{G_{\lambda(z)}}B_z\otimes_{G_{\rho(z)}}\pi_1(\gf,\rho(z),*).
\end{multline*}
Then $\pi_1(\mathfrak B, \dagger,*)\cong S/_{\sim}$, where $\sim$ is the
equivalence relation from~\eqref{eq:freebiset}. For any path $p\in
\pi_1(\gf, x, *)$ starting at some vertex $x\in\gf$, we consider the
following subset of $S$:
\[S(p)\coloneqq\bigcup_{z\in\rho^{-1}(x)}\pi_1(\mathfrak Y,\dagger,\lambda(z))\otimes_{\widehat G_{\lambda(z)}}\widehat B_z\otimes_{\widehat G_{\rho(z)}} p\]
viewed as a left $\pi_1(\mathfrak Y,\dagger)$-set.
  
\begin{prop}\label{prop:S(x)}
  Suppose $\mathfrak B$ is a left-fibrant graph of bisets. Then for every $p_1,p_2\in
  \pi_1(\gf, -, *)$ the relation $\sim$ on $S(p_1)\times S(p_2)$ is, in
  fact, a $\pi_1(\mathfrak Y,\dagger)$-set isomorphism between $S(p_1)$ and
  $S(p_2)$.
\end{prop}
\begin{proof}
  If $p\in\pi_1(\gf,x,*)$ then, clearly, $S(p)=S(g p)$ for every $g\in
  G_x$. We now show that for every edge $f\in\gf$ with $f^+=x$ the
  equivalence $q b^-p = q\lambda(z)b^+\overline{\rho(z)}p$ in
  \eqref{eq:freebiset} induces an isomorphism between $S(f p)$ and
  $S(p)$. Since $\widehat B^+_e$ is isomorphic to $\widehat B^-_e$ by
  Theorem~\ref{thm:FaithfGrOfBis}\eqref{en2:thm:FaithfGrOfBis} we have the
  required isomorphism:
\begin{align*}
  S(f p) &= \bigsqcup_{z\in \rho^{-1}(f^-)}\pi_1(\widehat{\mathfrak
    Y},\dagger,\lambda(z))\otimes_{\widehat G_{\lambda(z)}}\widehat B_z\otimes f p\\
  &= \bigsqcup_{z\in \rho^{-1}(f^-)}\pi_1(\widehat{\mathfrak Y},\dagger,\lambda(z))
  \otimes_{\widehat G_{\lambda(z)}}\bigg(\bigsqcup_{\substack{e\in \rho^{-1}(f)\\e^-=z}}
  \widehat G_{\lambda(e^-)} \otimes \widehat B^-_e \bigg) \otimes f p\text{
    using }\eqref{eq:dfn:FibrationGrBis}\\
  &= \bigsqcup_{e\in \rho^{-1}(f)}\pi_1(\widehat{\mathfrak Y},\dagger,\lambda(e^-))
  \otimes_{\widehat G_{\lambda(e^-)}}\widehat B^-_e\otimes f p\\
  &\cong \bigsqcup_{e\in \rho^{-1}(f)}\pi_1(\widehat{\mathfrak
    Y},\dagger,\lambda(e^-))\otimes  (\lambda(e) \widehat B^+_e
  \overline{\rho(e)})\otimes f p\text{ using }q b^-p \sim q\lambda(z)b^+\overline{\rho(z)}p\\
  &=\bigsqcup_{e\in \rho^{-1}(f)}\pi_1(\widehat{\mathfrak Y},\dagger,\lambda(e^+))\otimes \widehat B^+_e\otimes p =S(p).\qedhere
\end{align*}  
\end{proof}

\begin{cor}[Canonical form of $\pi_1(\mathfrak B, \dagger, *)$]\label{cor:NormalForm}
  Let ${}_{\mathfrak Y}\mathfrak B_\gf$ be a left-fibrant graph of
  bisets. Then its fundamental graph of bisets has the following
  description
  \begin{equation}
  \label{eq:cor:NormalForm}
    \pi_1(\mathfrak B,\dagger,*)=\bigsqcup_{z\in \rho^{-1}(*)}\pi_1(\mathfrak Y,\dagger,\lambda(z))\otimes_{G_{\lambda(z)}}B_z,
\end{equation}
 with right action given by lifting of paths in $\pi_1(\gf,*)$. 
\end{cor}
\begin{proof}
  Follows immediately from Proposition~\ref{prop:S(x)}, since the
  right-hand side of~\eqref{eq:cor:NormalForm} is $S(1)$.
\end{proof}

Since $\pi_1(\mathfrak B, \dagger, *)\cong \pi_1(\widehat{ \mathfrak B}, \dagger, *)$, by Theorem~\ref{thm:FaithfGrOfBis} we may rewrite~\eqref{eq:cor:NormalForm} as 
  \begin{equation}
  \label{eq:cor:NormalForm2}
    \pi_1(\mathfrak B,\dagger,*)\cong \bigsqcup_{z\in \rho^{-1}(*)}\pi_1(\widehat{\mathfrak Y},\dagger,\lambda(z))\otimes_{\widehat G_{\lambda(z)}}\widehat B_z.
\end{equation}

\begin{cor}\label{cor:HatBisLeftFree}
  Let ${}_{\mathfrak Y}\mathfrak B_\gf$ be a left-free graph of
  bisets.  Then $\pi_1(\mathfrak B, \dagger, *)$ is a left-free biset
  of degree equal to that of $\bigsqcup_{z\in\rho^{-1}(*)} B_z$.
\end{cor}
\noindent In particular, the number of left orbits of
$\bigsqcup_{z\in\rho^{-1}(*)}\widehat B_z$ is independent on
$*\in \mathfrak B$.
\begin{proof}
  Follows immediately from~\eqref{eq:cor:NormalForm2} and
  Theorem~\ref{thm:FaithfGrOfBis} Claim~\ref{en3:thm:FaithfGrOfBis}.
\end{proof}

\begin{lem}
  Suppose $\mathfrak B$ is a left-fibrant graph of bisets. Then
  $\pi_1(\mathfrak B, \dagger, *)$ is left-free if and only if
  $\widehat{\mathfrak B}$ is a left-free graph of bisets.
\end{lem}
\begin{proof}
  It follows from~\eqref{eq:cor:NormalForm2} that
  $\pi_1(\mathfrak B, \dagger, *)$ is left-free if and only if every
  $\widetilde {B}_z$ is left-free.
\end{proof}
\noindent The following corollary is an adaptation of
Lemma~\ref{lem:InjGraphGroups} to the context of bisets:
\begin{cor}
  Let $\mathfrak B$ be a left-fibrant graph of bisets. If for $y\in
  \mathfrak Y$ the maps $()^-\colon G_y\to G_{y^-}$ are injective, then the
  natural maps
  \[B_z\to\pi_1(\mathfrak B,\lambda(z),\rho(z))\qquad b\mapsto 1\otimes b
  \otimes 1
  \]
  are also injective for all vertices $z\in\mathfrak B$.
\end{cor}
\begin{proof}
  By Lemma~\ref{lem:InjGraphGroups} we have $G_y \cong \widehat G_y$ for
  all $y\in \mathfrak Y$. Therefore, by
  Theorem~\ref{thm:FaithfGrOfBis}\eqref{en1:thm:FaithfGrOfBis} we have
  $B_z\cong \widehat B_z$ for all $z\in \mathfrak B$. In particular, all
  $B_z\to\pi_1(\mathfrak B,\lambda(z),\rho(z))$ are injections by
  Corollary~\ref{cor:NormalForm}.
\end{proof}

Consider a graph of bisets ${}_{\mathfrak Y}\mathfrak B_{\gf}$. As for
graphs of groups, ${}_{\mathfrak Y}\mathfrak B_{\gf}$ has a barycentric
subdivision ${}_{\mathfrak Y'}\mathfrak B'_{\gf'}$: all graphs in
${}_{\mathfrak Y'}\mathfrak B'_{\gf'}$ are the barycentric subdivisions of
those in ${}_{\mathfrak Y}\mathfrak B_{\gf}$, the vertex groups and bisets
$G_y,G_x, B_z$ with $y\in \mathfrak Y', x\in \gf', z\in \mathfrak B'$ are
the groups and bisets of the associated objects in ${}_{\mathfrak
  Y}\mathfrak B_{\gf}$, the edge groups and bisets in ${}_{\mathfrak
  Y'}\mathfrak B'_{\gf'}$ represent the group morphisms and the biset
congruences of ${}_{\mathfrak Y}\mathfrak B_{\gf}$.

\begin{lem}\label{lmm:BarSubGrBis}
  Let ${}_{\mathfrak Y'}\mathfrak B'_{\gf'}$ be the barycentric subdivision
  of a graph of bisets ${}_{\mathfrak Y}\mathfrak B_{\gf}$. Then for
  $\dagger\in \mathfrak Y$ and $* \in \mathfrak \gf$ we have a natural
  isomorphism $\pi_1(\mathfrak B,\dagger, *)\cong \pi_1(\mathfrak
  B',\dagger, *)$.

  If ${}_{\mathfrak Y}\mathfrak B_{\gf}$ is left-fibrant (left-covering,
  etc.), then so is ${}_{\mathfrak Y'}\mathfrak B'_{\gf'}$.
\end{lem}
\begin{proof}
  Write $\mathfrak B'_{\dagger, *}\coloneqq(\mathfrak
  Y',\dagger)^\vee\otimes\mathfrak B'\otimes(\gf',*)$ and $\mathfrak
  B_{\dagger, *}\coloneqq(\mathfrak Y,\dagger)^\vee\otimes\mathfrak
  B\otimes(\gf,*)$. By Lemma~\ref{lem:CharOfFindBiset}, the fundamental
  bisets of $\mathfrak B'_{\dagger, *}$ and $\mathfrak B_{\dagger, *}$ are
  isomorphic to the fundamental bisets of $\mathfrak B'$ and $\mathfrak B$
  respectively. It is also easy to see that $\mathfrak B'_{\dagger, *}$ is
  (isomorphic to) to the barycentric subdivision of $\mathfrak B_{\dagger,
    *}$ because $\pi_1(\mathfrak Y',\dagger)\cong \pi_1(\mathfrak
  Y,\dagger)$ and $\pi_1(\gf',*)\cong \pi_1(\gf,*)$. This reduces the
  problem to the case when $\mathfrak Y$ and $\gf$ are graphs with one
  vertex. In this case we have a simple description of the fundamental
  biset $B$ of $\mathfrak B$, and similarly of $\mathfrak B'$, see
  \S\ref{ss:FundBis}: $B$ is the quotient of $\bigsqcup_{z\in V(\mathfrak
    B)}B_z$ by the relation identifying $b^-$ with $b^+$, for all edges
  $z\in\mathfrak B$ and all $b\in B_z$, with respective images $b^\pm$ in
  $B_{z^\pm}$. The first claim of the lemma easily follows.

  The second part of the lemma is straightforward: the unique lifting
  property~\eqref{eq:dfn:FibrationGrBis} is respected by passing to the
  barycentric subdivision and~\eqref{eq:dfn:FibrationGrBis} holds for the
  new vertices of $\mathfrak B'$ which are the former edges of $\mathfrak
  B$.
\end{proof}

\subsection{Properties of products}\label{ss:PropOfTensProd}
\begin{lem}\label{lm:MorphTensProd}
  For every graphs of bisets ${}_{\mathfrak Y}\mathfrak B_\gf$ and
  ${}_\gf\mathfrak C_{\mathfrak W}$ there exists a biset morphism
  \begin{equation}\label{eq:lm:MorphTensProd}
    \pi_1(\mathfrak B\otimes_\gf\mathfrak C,\dagger,\ddagger)\to\pi_1(\mathfrak
    B,\dagger,*)\otimes_{\pi_1(\gf,*)} \pi_1(\mathfrak C,*,\ddagger)
  \end{equation}
  defined as follows: consider $(v,w)\in V(\mathfrak B\otimes_\gf\mathfrak
  C)$ and $b\otimes c\in B_{(v,w)}$. Consider paths $q\in\pi_1(\mathfrak
  Y,\dagger,\lambda(v))$ and $r\in\pi_1(\mathfrak
  W,\rho(w),\ddagger)$. Then a typical element of $\pi_1(\mathfrak
  B\otimes_\gf\mathfrak C,\dagger,\ddagger)$ is of the form $q(b\otimes
  c)r$, and its image under~\eqref{eq:lm:MorphTensProd} is defined to be
  $(q b p)\otimes(p^{-1}c r)$ for any choice of path
  $p\in\pi_1(\gf,\rho(v),*)=\pi_1(\gf,\lambda(w),*)$.
\end{lem}
\begin{proof}
  To show that~\eqref{eq:lm:MorphTensProd} is well-defined, we must see
  that it is independent of the choice of $q(b\otimes c)r$ in its
  $\sim$-equivalence class and of the choice of $p$. Clearly the product
  over $\gf$ is independent of the choice of $p$. Changing $b\otimes c$ into
  $b g\otimes g^{-1}c$ is the same as changing $p$ into $g p$. For an edge
  $(e,f)$ in $\mathfrak B\otimes_\gf\mathfrak C$, changing $q(b\otimes
  c)^-r$ into $q\lambda(e)(b\otimes c)^+\overline{\rho(f)}r$ is the same as
  changing $p$ into $\lambda(e)p$. It is immediate
  that~\eqref{eq:lm:MorphTensProd} is a biset morphism.
\end{proof}

In general, the morphism in Lemma~\ref{lm:MorphTensProd} need not be
an isomorphism. For instance, $\mathfrak B\otimes_\gf\mathfrak C$ is
the empty graph of bisets if the images of the graphs $\mathfrak B$
and $\mathfrak C$ do not intersect in $\gf$. However,
\begin{lem}\label{lem:tensor commutes with pi1}
  Let $\mathfrak B$ be a left-fibrant $\mathfrak Y$-$\gf$-biset, and let
  $\mathfrak C$ be a $\gf$-$\mathfrak
  W$-biset. Then~\eqref{eq:lm:MorphTensProd} induces an isomorphism
  \[\pi_1(\mathfrak B,\dagger,*)\otimes_{\pi_1(\gf,*)}
  \pi_1(\mathfrak C,*,\ddagger)\cong\pi_1(\mathfrak
  B\otimes_\gf\mathfrak C,\dagger,\ddagger).\]
\end{lem}
\begin{proof}
  By definition, every element in $\pi_1(\mathfrak
  B,\dagger,*)\otimes_{\pi_1(\gf,*)}\pi_1(\mathfrak C,*,\ddagger)$ is of
  the form $q b p c r$ for paths $q,p,r$ in $\mathfrak Y,\gf,\mathfrak W$
  respectively. By Corollary~\ref{cor:NormalForm} the subexpression $q b p$
  can be rewritten in a unique way as $q' b'$. This defines a biset morphism
  \[\pi_1(\mathfrak B,\dagger,*)\otimes_{\pi_1(\mathfrak X,*)}\pi_1(\mathfrak C,*,\ddagger)\to\pi_1(\mathfrak B\otimes_\gf\mathfrak C,\dagger,\ddagger)
  \]
  by $q b p c r\mapsto q'(b'\otimes c)r$, which is clearly the inverse
  of~\eqref{eq:lm:MorphTensProd}.
\end{proof}

\begin{lem}\label{lem:ProdOfFibrGrBis}
  Let ${}_{\mathfrak Y}\mathfrak B_\gf$ and
  ${}_\gf\mathfrak C_{\mathfrak W}$ be left-fibrant (respectively
  left-free, left-principal) bisets. Then
  $\mathfrak B\otimes_\gf\mathfrak C$ is a left-fibrant (respectively
  left-free, left-principal) graph of bisets.
\end{lem}
\begin{proof}
  Observe first that
  $\rho\colon \mathfrak B\otimes_\gf\mathfrak C\to \mathfrak W$ is
  simplicial. Indeed, if $\rho(v, w)$ is a vertex for
  $(v,w)\in \mathfrak B\otimes_\gf\mathfrak C$, then $w$ is a vertex
  of $\mathfrak C$ because $\rho\colon \mathfrak C\to \mathfrak W$ is
  simplicial, and $\lambda(w)$ is a vertex of $\gf$ because morphisms
  send vertices to vertices, so $v\in \rho^{-1}(\lambda(w))$ is a
  vertex of $\mathfrak B$ because $\rho\colon \mathfrak B\to \gf$ is
  simplicial. Similarly, if $\rho(v, w)$ is an edge, then so is $w\in \mathfrak C$ because $\rho \colon \mathfrak C\to \mathfrak W$ is simplicial; thus $(v,w)$ is an edge in $ \mathfrak B\otimes_\gf\mathfrak C$.

  Consider now a vertex $(v,w)\in \mathfrak B\otimes_\gf\mathfrak C$
  and an edge $f\in \mathfrak W$ with $f^-=\rho(v,w)$. Let us
  verify~\eqref{eq:dfn:FibrationGrBis}. Note first that if $(\ell, e)\in \rho^{-1}(f)$ with $(\ell,e)^-=(v,w)$, then $\ell= v$ if and only if $\lambda(\ell)= \lambda(v)$, because $\rho\colon \mathfrak B\to \gf $ is simplicial.
  We have
  \begin{align*}
    B_{(v,w)} &= B_v\otimes B_w \cong  B_v  \otimes \bigsqcup_{\substack{e\in \rho^{-1}(f)\\ e^-=w}} G_{\lambda(w)}
    \otimes_{G_{\lambda(e)}} B_e \quad\text{using }\eqref{eq:dfn:FibrationGrBis}\text{ for }\mathfrak B_w\\
    &= \bigsqcup_{\substack{e\in \rho^{-1}(f)\\ e^-=w}} (B_v\otimes G_{\lambda(w)})
    \otimes_{G_{\lambda(e)}} B_e =  \bigsqcup_{\substack{e\in \rho^{-1}(f)\\ e^-=w}}B_v  
    \otimes_{G_{\lambda(e)}} B_e \quad\text{because }\rho(v)=\lambda(w) \\
    &= \Bigg(\bigsqcup_{\substack{e\in \rho^{-1}(f)\\ e^-=w\\ \lambda(e)=\rho(v)}} B_v
    \otimes_{G_{\lambda(e)}} B_e  \Bigg) \sqcup
    \Bigg(\bigsqcup_{\substack{e\in \rho^{-1}(f)\\ e^-=w\\ \lambda(e)\not =\rho(v)}} B_v
    \otimes_{G_{\lambda(e)}} B_e  \Bigg)\\
    &\cong \Bigg(\bigsqcup_{\substack{(v,e)\in \rho^{-1}(f)\\ (v,e)^-=(v,w)}} B_{(v,e)}  \Bigg)\sqcup\left(\bigsqcup_{\substack{e\in \rho^{-1}(f)\\ e^-=w\\ \lambda(e)\not =\rho(v)}}\Bigg(\bigsqcup_{\substack{\ell\in \rho^{-1}(\lambda(e))\\ \ell^-=v}} G_{\lambda(v)}\otimes_{G_{\lambda(\ell)}} B_\ell \Bigg) \otimes_{G_{\lambda(e)}} B_e\right) \\
    &\cong \Bigg(\bigsqcup_{\substack{(v,e)\in \rho^{-1}(f)\\ (v,e)^-=(v,w)}} G_{\lambda(v)}\otimes B_{(v,e)}\Bigg) \sqcup\Bigg(\bigsqcup_{\substack{(\ell,e)\in \rho^{-1}(f)\\ (\ell,e)^-=(v,w)\\ \ell\not =v}}
    G_{\lambda(v)}\otimes_{G_{\lambda(\ell)}} B_{(\ell, e)}  \Bigg)\\
    &\cong \bigsqcup_{\substack{(\ell,e)\in \rho^{-1}(f)\\ (\ell,e)^-=(v,w)}}
    G_{\lambda(v)}\otimes_{G_{\lambda(\ell)}} B_{(\ell, e)} = \bigsqcup_{\substack{(\ell,e)\in \rho^{-1}(f)\\ (\ell,e)^-=(v,w)}}
    G_{\lambda(v,w)}\otimes_{G_{\lambda(\ell,e)}} B_{(\ell, e)}.
  \end{align*} 
  This shows that $\mathfrak B\otimes_\gf\mathfrak C$ is a
  left-fibrant graph of bisets. If, moreover,
  ${}_{\mathfrak Y}\mathfrak B_\gf$ and
  ${}_\gf\mathfrak C_{\mathfrak W}$ are left-free (respectively
  left-principal), then all $B_{(v,w)}$ are left-free (respectively
  left-principal); thus $\mathfrak B\otimes_\gf\mathfrak C$ is a
  left-free (respectively left-principal) graph of bisets.
\end{proof}

\noindent We are now ready to prove Lemma~\ref{lem:CharOfFindBiset}:
\begin{proof}[Proof of Lemma~\ref{lem:CharOfFindBiset}]
  Both $(\mathfrak B, \dagger)$ and $(\gf,*)$ are left-principal graphs of
  bisets. The proof now follows from Lemmas~\ref{lem:tensor commutes with
    pi1} and~\ref{lem:ProdOfFibrGrBis}.
\end{proof}

\subsection{Principal and biprincipal graphs of bisets}\label{ss:BiprGrBis}
Let us now simplify the part of Definition~\ref{dfn:FibrationGrBis}
concerning principal graphs of bisets.

\begin{lem}\label{lem:DefnPrincipGrBis}
  A graph of bisets ${}_{\mathfrak Y} \mathfrak B_{\gf}$ is
  left-principal if and only if $\rho \colon \mathfrak B\to\gf $ is a
  graph isomorphism and $B_z$ are left-principal for all
  $z\in \mathfrak B$.
\end{lem}
\begin{proof}
  We verify that the conditions stated in
  Lemma~\ref{lem:DefnPrincipGrBis}
  imply~\eqref{eq:dfn:FibrationGrBis}.

  Since $\rho \colon \mathfrak B\to\gf $ is a graph
  isomorphism,~\eqref{eq:dfn:FibrationGrBis} takes the following form:
  for every $z\in \mathfrak B$ the map
  \begin{equation}\label{eq:lem:DefnPrincipGrBis}
    G_{\lambda(e^-)} 
    \otimes_{G_{\lambda(e)}} B_e\rightarrow B_{e^-},\qquad g\otimes b\mapsto g b^-
  \end{equation} 
  is a $G_{\lambda(e^-)}$-$G_{\rho(e)}$ biset isomorphism (for the
  action of $G_{\rho(e)}$ on $B_{e^-}$ given via
  $()^-\colon G_{\rho(e)}\to G_{{\rho(e)}^-}$). Let us
  verify~\eqref{eq:lem:DefnPrincipGrBis}.

  Consider an edge $e\in \mathfrak B$. We claim that
  ${}_{(G_{\lambda(e)})^{-}}(B_e)^{-}_{G_{\rho(e)}}$ is isomorphic to
  $(G_{\lambda(e)})^{-}\otimes_{G_{\lambda(e)}}
  (B_e)_{G_{\rho(e)}}$. Indeed,
  ${}_{(G_{\lambda(e)})^{-}}(B_e)^{-}_{G_{\rho(e)}}$ is left-free as a
  subbiset of a left-free biset $B_{e^-}$. Furthermore,
  ${}_{(G_{\lambda(e)})^{-}}(B_e)^{-}_{G_{\rho(e)}}$ has a single
  orbit under the action of $(G_{\lambda(e)})^{-}$ because this
  property holds for $B_e$. Therefore,
  ${}_{(G_{\lambda(e)})^{-}}(B_e)^{-}_{G_{\rho(e)}}$ is
  left-principal. As a set it is isomorphic to $(G_{\lambda(e)})^{-}$;
  this proves
  ${}_{(G_{\lambda(e)})^{-}}(B_e)^{-}_{G_{\rho(e)}} \cong
  (G_{\lambda(e)})^{-}\otimes_{G_{\lambda(e)}}
  (B_e)_{G_{\rho(e)}}$. Since $B_{e^-}$ is left-principal we obtain a
  natural isomorphism of $G_{\lambda(e^-)}$-$G_{\rho(e)}$-bisets
  \[B_{e^-}= G_{\lambda(e^-)}(B_e)^{-} \cong
  G_{\lambda(e^-)}\otimes_{G_{\lambda(e)}} B_e \]
  which is \eqref{eq:lem:DefnPrincipGrBis}.
\end{proof}

An example of left-principal graph of bisets is 
${}_{\gf}(\gf,*)_{\gf_0}$ defined in \S\ref{ss:FundBis}.

A biset $_H B_G$ is \emph{biprincipal} if it is left- and
right-principal. Clearly, if $_H B_G$ is biprincipal, then the groups
$G$ and $H$ are isomorphic. Similarly, a graph of bisets
${}_{\mathfrak Y} \mathfrak I_{\gf}$ is \emph{biprincipal} if it is
left- and right-principal. By Lemma~\ref{lem:DefnPrincipGrBis} a graph
of bisets ${}_{\mathfrak Y} \mathfrak I_{\gf}$ is biprincipal if and
only if
\begin{enumerate}
\item $\lambda\colon \mathfrak I\to\mathfrak Y$ and $\rho\colon
  \mathfrak I\to\gf$ are graph isomorphisms; and
\item $B_z$ are biprincipal for all objects $z\in \mathfrak I$.\qedhere
\end{enumerate}

\begin{defn}[Congruence of graphs of groups and bisets]
  Two graphs of groups $\mathfrak Y$, $\gf$ are called \emph{congruent} if
  there is a biprincipal graph of bisets ${}_{\mathfrak Y} \mathfrak
  I_{\gf}$.

  Isomorphism of graphs of bisets is meant in the strongest possible sense:
  isomorphism of the underlying graphs, and isomorphisms of the respective
  bisets.

  Two graphs of bisets ${}_{\mathfrak Y} \mathfrak B_{\gf}$ and
  ${}_{\mathfrak Y'} \mathfrak C_{\gf'}$ are \emph{congruent} if there are
  biprincipal graph of bisets ${}_{\mathfrak Y} \mathfrak I_{\mathfrak Y'}$
  and ${}_{\gf} \mathfrak L_{\gf'}$ such that ${}_{\mathfrak Y} \mathfrak
  B_{\gf}$ is isomorphic to $\mathfrak I \otimes {}_{\mathfrak Y'}
  \mathfrak C_{\gf'} \otimes \mathfrak L^\vee$.

  Finally, two graph of bisets ${}_{\gf} \mathfrak B_{\gf}$ and ${}_{\gf'}
  \mathfrak C_{\gf'}$ are \emph{conjugate} if ${}_{\gf} \mathfrak B_{\gf}$
  is isomorphic to $\mathfrak I \otimes {}_{\gf'} \mathfrak C_{\gf'}
  \otimes \mathfrak I^\vee$ for a biprincipal graph of bisets ${}_{\gf}
  \mathfrak I_{\gf'}$.
\end{defn}

Let ${}_H B_G$ be a biprincipal biset and let $H'$ be a subgroup of $H$. For
every $b\in B$ we consider the subgroup $(H')^b$ of $G$ defined by
\[H' b = b(H')^b.
\]
It is easy to see that the conjugacy class of $(H')^b$ in $G$ is
independent of $b$.

\begin{lem}
  Let $\mathfrak Y$, $\gf$ be graphs of groups and let
  $h\colon \mathfrak Y\to\gf$ be an isomorphism of the underlying
  graphs. Suppose also that a $G_y$-$G_{h(y)}$ biprincipal biset $B_y$
  is given for every object $y\in \mathfrak Y$. Set $\mathfrak I$ to
  be $\mathfrak Y$ as a graph. Then the data
  $\{B_z\}_{z\in \mathfrak I}$ extends (via appropriate embeddings of
  edge bisets into vertex bisets) to a biprincipal graph of bisets
  ${}_{\mathfrak Y}\mathfrak I_{\gf}$ if and only if for every edge
  $e\in \mathfrak I$ the following holds. For any, or equivalently for
  every, $b\in B_e$ the groups $(G_{e}^{-})^b$ and $G_{h(e)}^-$ are
  conjugate as subgroups of $G_{h(e^-)}$.
\end{lem}
\begin{proof}
  Follows immediately from the definition: if $(G_{e}^{-})^b$ and
  $G_{h(e)}^-$ are conjugate, say $((G_{e}^{-})^b)^g=G_{h(e)}^-$, then we
  may define $()^-\colon B_e\to B_{e^-}$ by mapping $B_e$ into $G_{e}^{-}
  (b g) G_{h(e)}^-$. The converse is obvious.
\end{proof}

\subsection{Examples}\label{ss:examplesgfgps}
We will give more examples and applications
in~\S\ref{ss:dynamics}. Nevertheless, we give here an example that
illustrates the main features of Definition~\ref{defn:p1biset}. First let
us introduce some conventions simplifying descriptions of the examples.

Let the graph $\gf=V\sqcup E$ be a graph as in
Definition~\ref{defi:Graph}. By an \emph{undirected graph} we mean
$\gf_0\coloneqq \gf/\{z=\bar z\}$ with vertex set $V_0=V$ and \emph{geometric} edge set
$E_0\coloneqq E/\{e=\bar e\}$. The undirected graph $\gf_0$ is endowed with
an adjacency relation. A map $\theta\colon\gf_0\to\mathfrak Y_0$ is a
\emph{weak morphism} if $\theta$ sends adjacent objects into adjacent or
equal objects. If $\gf'$ and $\mathfrak Y'$ are the barycentric
subdivisions of $\gf$ and $ \mathfrak Y$ as in \S\ref{ss:vk}, then
$\theta\colon\gf_0\to\mathfrak Y_0$ uniquely determines a graph morphism
$\theta\colon\gf'\to\mathfrak Y'$.

\theoremstyle{definition}\newtheorem*{convention}{Convention}
\begin{convention}
  We will often describe a graph of bisets as ${}_{\mathfrak Y_0}{
    \mathfrak B_0 }_{\gf_0}$ with
  \begin{itemize}
  \item undirected graphs ${\mathfrak Y_0},{ \mathfrak B_0 },{\gf_0}$;
  \item weak graph morphisms $\lambda\colon \mathfrak B_0\to\mathfrak Y_0$
    and $\rho \colon \mathfrak B_0\to\gf_0$;
  \item groups $G_y,G_x$ and $G_{\lambda(z)}$-$G_{\rho(z)}$ bisets $B_z$
    for all $y\in \mathfrak Y_0$, $x\in \gf_0$ and $z\in \mathfrak B_0$
    respectively;
  \item congruences from $B_e$ to $B_v$ and $B_w$ for all edges $e\in
    \mathfrak B_0$ adjacent to vertices $v,w\in \mathfrak B_0$; and,
    similarly, group morphisms from $B_e$ to $B_v$ and $B_w$ for all edges
    $e\in \mathfrak Y_0\sqcup \gf_0$ adjacent to vertices $v,w$.
  \end{itemize}
  By passing to a barycentric subdivision of ${}_{\mathfrak Y_0}{ \mathfrak
    B_0 }_{\gf_0}$ we obtain a graph of bisets ${}_{\mathfrak Y'}{
    \mathfrak B' }_{\gf'}$ satisfying
  Definition~\ref{defn:graphofbisets}. Following
  Lemma~\ref{lmm:BarSubGrBis} we set $\pi_1({}_{\mathfrak Y_0}{ \mathfrak
    B_0 }_{\gf_0}, \dagger, *)\coloneqq\pi_1({}_{\mathfrak Y'}{ \mathfrak
    B' }_{\gf'}, \dagger, *)$.
\end{convention}

\begin{exple}\label{exm:Basil}
  Consider the biset corresponding to the ``Basilica map''
  $f(z)=z^2-1$, from $(\C,\{-1,0,1\})$ to $(\C,\{0,-1\})$.

  On the one hand, choose a basepoint $x_0\cong-0.6$ in
  $\C\setminus\{0,-1\}$, and choose a graph $\gf$ with one vertex
  $x_0$ to which $\C\setminus\{0,-1\}$ deformation retracts; let
  $\mathfrak Y$ be the full preimage of $\gf$ under $f$:
  \[\begin{tikzpicture}
    \begin{scope}
      \node at (-1,0) [above] {$y_0'$};
      \node at (1,0) [above] {$y_0''$};
      \filldraw (1,0) circle (2pt) (-1,0) circle (2pt);
      \draw (1,0) .. controls +(1,1) and +(1,-1) .. +(0,0);
      \draw (-1,0) .. controls +(-1,1) and +(-1,-1) .. +(0,0);
      \draw (1,0) .. controls (0,0.5) .. (-1,0);
      \draw (1,0) .. controls (0,-0.5) .. (-1,0);
    \end{scope}
    \draw[->,thick] (2,0) -- node[above] {$f$} +(1,0);
    \begin{scope}[xshift=4cm]
      \node at (0,0) [above] {$x_0$};
      \filldraw (0,0) circle (2pt);
      \draw (0,0) .. controls +(1,1) and +(1,-1) .. +(0,0);
      \draw (0,0) .. controls +(-1,1) and +(-1,-1) .. +(0,0);
    \end{scope}
  \end{tikzpicture}\]
  There, the graphs of groups $\gf$ and $\mathfrak Y$ have trivial
  groups, and the biset $\mathfrak B$ corresponding to $f$ is
  $\mathfrak Y$ as a graph, with $\lambda=1$ and $\rho=f$.

  This biset may be encoded differently, by considering rather an
  orbispace structure on $\C$ with non-trivial groups at $0,\pm1$; the
  complex plane deformation retracts to the intervals $[-1,0]$ and
  $[-1,1]$ respectively:
  \[\begin{tikzpicture}
    \begin{scope}
      \node at (-1,0) [above] {$y_0'$};
      \node at (-1,0) [below] {$\Z$};
      \node at (1,0) [above] {$y_0''$};
      \node at (1,0) [below] {$\Z$};
      \node at (0,0) [above] {$y_1$};
      \node at (0,0) [below] {$\Z$};
      \filldraw (-1,0) circle (2pt) -- (0,0) circle (2pt) -- (1,0) circle (2pt);
    \end{scope}
    \draw[->,thick] (2,0) -- node[above] {$f$} +(1,0);
    \begin{scope}[xshift=5cm]
      \node at (0,0) [above] {$x_0$};
      \node at (0,0) [below] {$\Z$};
      \node at (-1,0) [above] {$x_1$};
      \node at (-1,0) [below] {$\Z$};
      \filldraw (0,0) circle (2pt) -- (-1,0) circle (2pt);
    \end{scope}
  \end{tikzpicture}\]

  The graphs $\gf$ and $\mathfrak Y$ have the group $\Z$
  attached to each vertex, and the edge groups are all trivial. The
  biset $\mathfrak B$ has $\mathfrak Y$ as underlying graph; the map
  $\lambda$ is the retraction to the segment $[y'_0,y_1]=[x_1,x_0]$
  and $\rho$ folds $\mathfrak Y$ around its middle point. The bisets
  $B_{y'_0}$ and $B_{y''_0}$ are ${}_\Z\Z_\Z$, while
  $B_{y_1}={}_{2\Z}\Z_\Z$ is free of rank $2$ on the left. The edge
  bisets embed in $B_{y_1}$ as two elements of opposite parity.

  Note that the underlying graph of $\mathfrak B$ is the ``Hubbard
  tree'' of $f$; that will be explained in more details
  in~\S\ref{ss:hubbardtrees}.
\end{exple}

\begin{prop}
  Let $\phi\colon H\to G$ be a group homomorphism, and let $\gf$ be a graph
  of groups with fundamental group $G$. Then there exists a graph of
  groups decomposition of $H$ and a graph of bisets with limit
  $f$. Furthermore, there exists a unique minimal such decomposition
  for $H$ in the sense that other decompositions are refinements of
  it.
\end{prop}
\begin{proof}
  Let $G$ act on the tree $\widetilde\gf$, the universal cover of
  $\gf$ (also known as the \emph{Bass-Serre tree} of $G$,
  see~\cite{serre:trees}), with stabilizers the groups $G_x$ of
  $\gf$. Then $H$ acts on $\widetilde\gf$ via $\phi$. We define the
  graph of groups $\mathfrak Y$ as $\widetilde\gf/\!/H$, namely as the
  quotient of $\widetilde\gf$ by the action of $H$, remembering the
  stabilizers in $H$ of vertices and edges in the quotient.

  We next define a graph of bisets $\mathfrak B$ as follows. Its
  underlying graph is $\widetilde\gf/H$. The graph morphism
  $\lambda\colon\mathfrak B\to\mathfrak Y$ is the identity, and the graph
  morphism $\rho\colon\mathfrak B\to\gf$ is the natural quotient map
  $x H\mapsto x G$. The biset at the vertex or edge $x H$ is
  $G_{\rho(x H)}$, on which $G_{\rho(x H)}$ acts naturally by left
  multiplication, while $H_{x H}$ acts on the right via $\phi$.
\end{proof}

\begin{rem}
  In a manner similar to Lemma~\ref{lem:biset=fibre}, every graph of
  bisets ${}_{\mathfrak Y}\mathfrak B_\gf$ may be factored as
  $\mathfrak B=\mathfrak A^\vee\otimes_{\mathfrak K}\mathfrak C$ for
  some graph of groups $\mathfrak K$ and some right-principal bisets
  ${}_{\mathfrak K}\mathfrak A_{\mathfrak Y}$ and ${}_{\mathfrak
    K}\mathfrak C_\gf$. Furthermore, there is an essentially unique
  minimal such $\mathfrak K$.
\end{rem}

\begin{lem}
  Let ${}_{\mathfrak Y}\mathfrak B_\gf$ and ${}_{\mathfrak Y}\mathfrak
  C_\gf$ be graphs of bisets over the same graphs, and let
  $\theta\colon\mathfrak B\to\mathfrak C$ be a ``semiconjugacy of graphs of
  bisets'', namely a graph morphism $\theta\colon\mathfrak B\to\mathfrak C$
  and compatible surjective biset morphisms $\theta_z\colon B_z\to
  C_{\theta(z)}$. Then $\theta$ may be factored as
  $\theta=\kappa\circ\iota$ through semiconjugacies $\iota\colon\mathfrak
  B\to\mathfrak C'$ and $\kappa\colon\mathfrak C'\to\mathfrak C$, in such a
  manner that $\iota$ induces an isomorphism between $\pi_1(\mathfrak
  B)$ and $\pi_1(\mathfrak C')$, and the underlying graph of
  $\mathfrak C'$ is the same as that of $\mathfrak C$.
\end{lem}

\section{Correspondences}\label{ss:corr}
Bisets are well adapted to encode more general objects than continuous
maps, \emph{topological correspondences}.
\begin{defn}[Correspondences]
  Let $X,Y$ be topological spaces. A \emph{topological correspondence from
    $Y$ to $X$} is a triple $(Z,f,i)$ consisting of a topological space $Z$
  and continuous maps $f\colon Z\to X$ and $i\colon Z\to Y$. It is often
  abbreviated $(f,i)$, and written $Y\leftarrow Z\to X$.

  In the special case that $f$ is a covering, $(Z,f,i)$ is called a
  \emph{covering correspondence}.  If $X=Y$, then $(Z,f,i)$ is called
  a \emph{self-correspondence}. If $X=Y$ and $f$ is a covering, then
  $(Z,f,i)$ is called a \emph{covering pair}, and is also written
  $f,i\colon Z\rightrightarrows X$.
\end{defn}

Ultimately, we will be interested in dynamical systems that are
partially defined topological coverings $f\colon X\dashrightarrow X$, for
path connected topological spaces $X$; the dashed arrow means that the
map is defined on an open subset $Y\subseteq X$; these can naturally
be viewed as correspondences $(Y,f,i)$ with $i$ the inclusion.

\subsection{The biset of a correspondence}\label{ss:biset of correspondence}
Let $Y\leftarrow Z\rightarrow X$ be a correspondence, and assume that
$X$ and $Y$ are path connected. Assume first that $Z$ is path
connected, and fix a basepoint $\ddagger\in Z$. Then we
have~\eqref{eq:DfnBisOfMap} a $\pi_1(Z,\ddagger)$-$\pi_1(X,*)$-biset
$B(f)$ and a $\pi_1(Z,\ddagger)$-$\pi_1(Y,\dagger)$-biset $B(i)$; we
\emph{define} the biset $B(f,i)$ to be
$B(i)^\vee\otimes_{\pi_1(Z,\ddagger)} B(f)$; this amounts to first
``lifting'' from $Y$ to $Z$, and then projecting $Z$ back to $X$. If
$Z$ is not path connected, then we \emph{define} $B(f,i)$ to be the
disjoint union of the bisets $B(i)^\vee\otimes_{\pi_1(Z,p)} B(f)$ with
one $p$ per path connected component of $Z$.

The following alternate construction of $B(f,i)$ avoids an
unnecessary reference to the basepoint and fundamental group of $Z$:
\begin{lem}\label{lem:cpbiset}
  The biset $B(f,i)$ may be constructed directly as follows, from
  the correspondence:
  \begin{equation}\label{eq:lem:cpbiset}
    B(f,i)=\left\{\left(\begin{array}{c}\delta\colon[0,1]\to Y\\p\in Z\\\gamma\colon[0,1]\to X\end{array}\middle)\right|\begin{array}{cc}\delta(0)=\dagger,&\delta(1)=i(p)\\\gamma(0)=f(p),&\gamma(1)=*\end{array}\right\}\bigg/{\sim}
  \end{equation}
  in which the equivalence relation $\sim$ allows $p$ to move in $Z$
  as long as $\delta,\gamma$ move smoothly with it\footnote{Note the
    similarity with Definition~\ref{defn:p1biset}}.

  The left action is by pre-composition of $\delta$ by loops in $Y$ at
  $\dagger$, and the right action is by post-composition of $\gamma$ by loops in $X$ at $*$:
  \[[\alpha]\cdot([\delta],p,[\gamma])\cdot[\varepsilon]=([\alpha\#\delta],p,[\gamma\#\varepsilon]).
  \]
\end{lem}
\[\begin{tikzpicture}
  \draw (3,3) ellipse (3cm and 1cm) +(2.2,0.4) node {$Z$};
  \draw (0,1) ellipse (2.5cm and 0.9cm) +(-2,-0.2) node {$Y$};
  \draw (6,1) ellipse (2.5cm and 0.9cm) +(2,-0.2) node {$X$};
  \draw (-1.5,1) node {$\bullet$} node[above left] {$\dagger$} .. controls (-0.5,0.6) and (0.5,1.4) .. node[below] {$\delta$} (1.5,1);
  \draw (4.5,1) .. controls (5.5,0.6) and (6.5,1.4) .. node[below] {$\gamma$} (7.5,1) node {$\bullet$} node[above right] {$*$};
  \draw[dashed,->] (3,3) node[above right] {$p$} -- node [pos=0.65,right] {$i$} (1.5,1);
  \draw[dashed,->] (3,3) -- node [pos=0.65,left] {$f$} (4.5,1);
\end{tikzpicture}
\]
\begin{proof}
  By the definition of products of bisets,
  $B(f,i)=\bigsqcup_z\{(\check\delta,\gamma)\}/{\sim}$, with
  $\delta\colon[0,1]\to Y$ ending at $\dagger$ and $\gamma\colon[0,1]\to X$
  ending at $*$ and $\delta(0)=i(z)$ and $\gamma(0)=f(z)$, and with
  one $z$ per path connected component of $Z$. By the homotopy relation, we
  may also allow $z$ to move freely, arriving at the formulation of
  the lemma in which no condition on $z$ is imposed.
\end{proof}

Recall that a continuous map $f\colon Y\to X$ is a \emph{fibration}, or a
\emph{fibrant map} if it has the homotopy lifting property with respect to
arbitrary spaces: for every space $Z$, every homotopy $g\colon
Z\times[0,1]\to X$ and every $h\colon Z\times\{0\}\to Y$ with $f\circ h=g$
on $Z\times\{0\}$ there exists an extension of $h$ to $Z\times[0,1]$ with
$f\circ h=g$.  Clearly, the composition of fibrations is a fibration. If
$X$ is path connected, then in particular all fibres $f^{-1}(x)$ are
homotopy equivalent.

\begin{rem}
  In fact, we may consider at no cost a larger class of maps: let us say
  that a map $f$ is a \emph{$\pi_1$-fibration} if $f$ has the lifting
  property with respect to contractible spaces. This is sufficient for our
  purposes; it implies, for example, that all $f^{-1}(x)$ have isomorphic
  fundamental groups.
\end{rem}

\begin{lem}\label{lmm:FibrCorr}
  Assume that $f\colon Z\to X$ is a fibration. Then the biset of $f$ has
  the following description:
  \begin{equation}\label{eq:Bf:lmm:FibrCorr}
    B(f)=\{\beta\colon[0,1]\to Z\mid\beta(0)=\dagger,f(\beta(1))=*\}/{\sim},
  \end{equation}
  with $\beta\sim \beta'$ if and only if there is a path
  $\varepsilon\colon[0,1]\to f^{-1}(*)$ connecting $\beta(1)$ to
  $\beta'(1)$ such that $\beta\#\varepsilon$ is homotopic to $\beta'$.
    
  Assume that $(f,i)$ is a correspondence with $f$ fibrant. Then
  in~\eqref{eq:lem:cpbiset} we may assume that $\gamma$ is constant, and
  write
  \begin{equation}\label{eq:Bfi:lmm:FibrCorr}
  B(f,i)=\{(\delta\colon[0,1]\to Y,p\in Z)\mid\delta(0)=\dagger,\delta(1)=i(p),f(p)=*\}/{\sim},
  \end{equation}
  with $(\delta,p)\sim(\delta',p')$ if and only if there exists a path
  $\varepsilon\colon[0,1]\to f^{-1}(*)\subseteq Z$ connecting $p$ to $p'$,
  such that $\delta\#i(\varepsilon)$ is homotopic to $\delta'$.
\end{lem}
\begin{proof}
  Recall from~\eqref{eq:DfnBisOfMap} that $B(f)$ is defined as
  $\{\gamma\colon[0,1]\to X\mid \gamma(0)=f(\dagger),\;\gamma(1)=*\}/{\sim}$.
  Since $f$ is fibrant, every $\gamma\in B(f)$ has a lift, say
  $\beta$, with $\beta(0)=\dagger$ and $f\circ\beta = \gamma$. We need
  to show that if $\gamma_0$ is homotopic to $\gamma_1$, say via
  $\gamma_t$, then their lifts $\beta_0$ and $\beta_1$ have the
  property that $\beta_0$ is homotopic to $\beta_1\#\varepsilon$ for
  some $\varepsilon$ in $f^{-1}(*)$. Consider $\gamma_t$ as a map
  $\gamma\colon[0,1]\times [0,1]\to X$, where $t$ is the first variable and
  the second variable parameterizes $\gamma_t$. By construction,
  $\gamma(t,0)=f(\dagger)$ and $\gamma(t,1)=*$. Since $f$ is
  fibrant, there is a lift $\epsilon\colon[0,1]\times[0,1]\to Z$ of
  $\gamma$ such that $\epsilon(t,0)=\dagger$, $\epsilon f=\gamma$,
  $\epsilon(0,\tau)=\beta_0(\tau)$, and
  $\epsilon(1,\tau)=\beta_1(\tau)$. Define
  $\varepsilon(\tau)\coloneqq\epsilon(1,\tau)$. Then
  $\beta^{-1}_0\beta_1\varepsilon$ is
  $\epsilon(\partial([0,1]\times[0,1]))$. Therefore, $\beta_0$ is
  homotopic to $\beta_1\#\varepsilon$.

  Let us now prove the second statement. Using the first part, the
  biset $B(f,i)$ in the sense of~\eqref{eq:lem:cpbiset} is identified
  with
  \[\{(\delta\colon[0,1]\to Y,\beta\colon[0,1]\to
  Z)\mid\delta(0)=\dagger,\delta(1)=i(\beta(0)),f(\beta(1))=*\}/{\sim},
  \]
  where $(\delta_0,\beta_0)\sim (\delta_1,\beta_1)$ if there is a path
  $\varepsilon\colon[0,1]\to f^{-1}(*)$ such that $(\delta_0,\beta_0)$ is
  homotopic to $(\delta_1,\beta_1\varepsilon)$. Projecting
  $\beta_1\varepsilon$ to $Y$ yields the required description.
\end{proof}

\begin{exple}
  Here are some extreme examples of correspondences worth keeping in
  mind.  If $Z=\{\ddagger\}$ is a point, then
  $B(f,i)=\pi_1(Y,\dagger)\times\pi_1(X,*)$. On the other hand, if
  $Z=X=Y$ and $f$ and $i$ are the identity, then
  $B(f,i)\cong\pi_1(Y,\dagger)\cong\pi_1(X,*)$. Finally, one may
  consider $Z=Y\times X$ with $i,f$ the first and second projections
  respectively; then $B(f,i)$ consists of a single element.
\end{exple}

If $f$ is a covering, then the biset $B(f,i)$ is left-free. Indeed,
consider~\eqref{eq:Bfi:lmm:FibrCorr} as a description of
$B(f,i)$. Suppose $f^{-1}(*)=\{p_i\}_{i\in I}$. For every $i$ choose
$\delta_i\colon[0,1]\to Y$ with $\delta_i(0)=\dagger$ and
$\delta_i(1)=i(p_i)$.  Then $(\delta_i,p_i)_{i\in I}$ forms a basis of
$B(f,i)$.

\subsection{Van Kampen's theorem for correspondences}
\label{ss:VanKampCorr}
We describe in this section the decomposition of the biset of a
correspondence into a graph of bisets, using a decomposition of the
underlying spaces. We start by the following straightforward lemma.

\begin{lem}\label{lm:BisMorph}\pushQED{\qed}
  Let
  \begin{itemize}
  \item $f\colon Y\to X$ be a continuous map;
  \item $X'$ and $Y'$ be path connected subsets of $X$ and $Y$ with
    $f(Y')\subseteq X'$;
  \item $\dagger$, $\dagger'$, $*$, and $*'$ be base points in $Y$,
    $Y'$, $X$, and $X'$; and
  \item $\gamma_*\subset X$ and $\gamma_\dagger \subset Y$ be curves
    connecting $*$ to $*'$ and $\dagger$ to $\dagger'$ respectively.
  \end{itemize}
  We allow $X'=X$; in this case we suppose that $\gamma_*$ is the constant
  path and $*'=*$.
 
  Then the maps $\delta\to\gamma_\dagger\delta \gamma^{-1}_\dagger$, $\delta\to
  \gamma_*\delta \gamma^{-1}_*$, and $b\to f(\gamma_\dagger ) b \gamma^{-1}_*$
  define a biset congruence
  \[_{\pi_1(Y',\dagger')}B(f\restrict{Y'},\dagger',*')_{\pi_1(X',*')}\to\ _{\pi_1(Y,\dagger)}B(f,\dagger,*)_{\pi_1(X,*)}.\qedhere\]
\end{lem}

Let $f\colon Z\to X$ be a continuous map between path connected
spaces. Suppose that $(U_{\alpha})$ and $(W_\gamma)$ are finite
$1$-dimensional covers $X$ and $Z$ respectively, as in
Definition~\ref{defn:1dimcovers}. Choose base points
$*_\alpha \in U_{\alpha}$ and $\dagger_\gamma\in W_{\gamma}$ as well
as connections between $*_{\alpha'}$ and $*_{\alpha}$ if
$U_{\alpha'}\varsubsetneqq U_{\alpha}$; and similarly for
$(W_\gamma)$. Consider the associated graphs of groups $\mathfrak X$
and $\mathfrak Z$ according to
Definition~\ref{defn:gog_1dimcovers}. We have
$\pi_1(X,*_\alpha)=\pi_1(\mathfrak X,\alpha)$ and
$\pi_1(Z,\dagger_\gamma)=\pi_1(\mathfrak Z,\gamma)$, by
Theorem~\ref{thm:vankampen}.  Suppose furthermore that $(U_{\alpha})$
and $(W_\gamma)$ are \emph{compatible} with $f$ in the following
sense: for every $\gamma$ there is an $\alpha\eqqcolon\rho(\gamma)$
with $f(W_\gamma)\subset U_\alpha$.

To the above data there is an associated graph of bisets ${}_{\mathfrak
  Z}\mathfrak B(f)_{\mathfrak X}$ is defined as follows. As a graph,
$\mathfrak B$ is $\mathfrak Z$, with $\lambda\colon\mathfrak B\to\mathfrak
Z$ the identity and $\rho\colon\mathfrak B\to\gf$ given above. For every
vertex $\gamma\in\mathfrak B$, the biset $B_\gamma$ is $B(f\colon
W_\gamma\to
U_{\rho(\gamma)},\dagger_{\lambda(\gamma)},*_{\rho(\gamma)})$. For every
edge $e\in \mathfrak B$ representing the embedding
$W_{\gamma'}\varsubsetneqq W_\gamma$ the biset $B_e$ is $B_{\gamma'}$. If
$e$ is oriented so that $e^-=\gamma'$ and $e^+=\gamma$, the congruences
$()^\pm$ are $()^-=\one\colon B_e\to B_{\gamma'}$ and $()^+\colon B_e\to
B_\gamma$ given by Lemma~\ref{lm:BisMorph}.
    
By construction, ${}_{\mathfrak Z}\mathfrak B(f)_{\mathfrak X}$ is a
right-principal graph of bisets.

Consider now a correspondence $(f,i)$, with $f\colon Z\to X$ and $i\colon
Z\to Y$, between path connected spaces $X$ and $Y$.  Suppose $(Z_k)$ are
the path connected components of $Z$, and suppose $(U_{\alpha})$,
$(V_\beta)$, and $(W_\gamma)$ are finite $1$-dimensional covers of $X$,
$Y$, and $Z$ respectively, \emph{compatible with} $f$ and $i$: for every
$\gamma$ there are $\lambda(\gamma)$ and $\rho(\gamma)$ such that
$f(W_\gamma)\subset U_{\rho(\gamma)}$ and $i(W_\gamma)\subset
V_{\lambda(\gamma)}$.

Let $(U_{\alpha})$, $(V_\beta)$, and $(W_\gamma)$ be finite
$1$-dimensional covers of $X$, $Y$, and $Z$ compatible with $f$ and
$i$ as above. We then have graphs of bisets
${}_{\mathfrak Z_k}\mathfrak B(f\restrict{Z_k})_{\mathfrak X}$ and
${}_{\mathfrak Z_k}\mathfrak B(i\restrict{Z_k})_{\mathfrak Y}$. We
define the graph of bisets of the correspondence $(f,i)$ as
\[ {}_{\mathfrak Y}\mathfrak B(f,i)_{\mathfrak X}\coloneqq\bigsqcup_{k} {}_{\mathfrak Y}\mathfrak B(i\restrict{Z_k})^\vee_{\mathfrak Z_k} \otimes_{\mathfrak Z_k} {}_{\mathfrak Z_k}\mathfrak B(f\restrict{Z_k})_{\mathfrak X}.
\]
The following equivalent description of $ {}_{\mathfrak Y}\mathfrak
B_{\mathfrak X}$ follows immediately from the definition; see
Lemma~\ref{lem:cpbiset}
\begin{lem}\label{lem:gob_correspondence}\pushQED{\qed}
  Let $(Z,f,i)$ be a topological correspondence from $Y$ to $X$, for path
  connected spaces $X,Y$. Let $(U_{\alpha})$, $(V_\beta)$, and $(W_\gamma)$
  be finite $1$-dimensional covers of $X$, $Y$, and $Z$ compatible with $f$
  and $i$ as above. Then the graph of bisets ${}_{\mathfrak Y}\mathfrak
  B_{\mathfrak X}$ of $(f,i)$ with respect to the above data is as follows:
  \begin{itemize}
  \item the graphs of groups $\mathfrak X$ and $\mathfrak Y$ are
    constructed as in Definition~\ref{defn:gog_1dimcovers} using the
    covers $(U_\alpha)$ and $(V_\beta)$ of $X,Y$ respectively. Choices
    of paths $\ell_e$, $m_e$ were made for edges $e$ in $\mathfrak X$,
    $\mathfrak Y$ respectively;
  \item the underlying graph of $\mathfrak B$ is similarly constructed
    using the cover $(W_\gamma)$ of $Z$. Note that neither $Z$ nor the
    underlying graph of $\mathfrak B$ need be connected;
  \item for every vertex $z\in \mathfrak B$ the biset
    ${}_{G_{\lambda(z)}}B_z {}_{G_{\rho(z)}}$ is $B(f\restrict{W_z},
    i\restrict{W_z})$;
  \item for every edge $e\in \mathfrak B$ representing the embedding
    $W_{z'}\varsubsetneqq W_z$ the biset $B_e$ is $B_{z'}$, and if $e$
    is oriented so that $e^-=z'$ then the congruences $()^\pm$ are the
    maps $()^-=\one\colon B_e\to B_{z'}$ and $()^+\colon B_e\to B_z$
    given by
    $(\gamma^{-1},\delta)\mapsto(m_{\lambda(e)}^{-1}\#\gamma^{-1},\delta\#\ell_{\rho(e)})$
    in the description of $B_e$ as
    $B(i\restrict{W_{e^-}})^\vee\otimes B(f\restrict{W_{e^-}})$, see
    Lemma~\ref{lm:BisMorph}.\qedhere
  \end{itemize}
\end{lem}

\begin{thm}[Van Kampen's theorem for correspondences]\label{thm:vankampenbis}
  Let $(f,i)$ be a topological correspondence from a path connected
  space $Y$ to a path connected space $X$, and let
  $ {}_{\mathfrak Y}\mathfrak B_{\mathfrak X}$ be the graph of bisets
  subject to compatible finite $1$-dimensional covers of spaces in
  question.
  
  Then for every $v\in\mathfrak Y$ and $u\in\mathfrak X$ we have
  \[B(f,i,\dagger_v,*_u)\cong \pi_1(\mathfrak B, v,u),\]
  where $\dagger_v$ and $*_u$ are basepoints. 
\end{thm}
\begin{proof}
  Recall from \S\ref{ss:vk} that for all vertices $y\in \mathfrak Y$ there
  are chosen basepoints $\dagger_y\in V_y$ identifying the groups $G_y$
  with $\pi_1(V_y, \dagger_y)$ and for all edges $e\in\mathfrak Y$ there
  are curves $\gamma_e$ connecting $\dagger_{e^-}$ to $\dagger_{e^+}$
  satisfying $\gamma_y=\gamma^{-1}_{\overline y}$ and describing
  $()^-$-maps. Similarly, for objects in $\mathfrak X$ there are
  basepoints $*_x\in U_x$ and curves $\gamma_e\subset X$.

  Then every biset element $\overline b=(h_0,y_1,\dots,b,\dots x_n,g_n)\in
  \pi_1(\mathfrak B,\dagger, *)$ defines a certain element
  $\theta(\overline b)\in B((Z,f,i),\dagger_\dagger,*_*)$ that is the
  concatenation $h_0\#\gamma_{y_1}\#\dots\#b\dots\# \gamma_{x_n}g_n$. We
  get a biset morphism
  \[\theta\colon \pi_1(\mathfrak B,\dagger,*)\to B((Z,f,i),\dagger_\dagger,*_*).
  \]
  Clearly $\theta$ is surjective, because every element in
  $B((Z,f,i),\dagger_\dagger,*_*)$ can be presented as a concatenation
  $h_0\#\gamma_{y_1}\#\dots\#b\dots\# \gamma_{x_n}g_n$, see
  Remark~\ref{rem:OnFinitnOfDecomp}. To prove that $\theta$ is injective,
  we show that if $\theta(\overline b_1)$ is homotopic to $\theta(\overline
  b_2)$, then $\overline b_1=\overline b_2$.

  By the classical van Kampen argument (see~\cite{massey:at}*{Chapter~IV})
  a homotopy between $\theta(\overline b_1)$ and $\theta(\overline b_2)$
  can be expressed as a combination of the following operations: (1)
  homotopies within $U_i$ and in $V_j$; (2) replacement of $g_k$ or $h_k$
  by $\gamma_{z} g_k\gamma^{-1}_{z}$ or $\gamma_{z} h_k\gamma^{-1}_{z}$ for
  some $z$, and (3) replacement of $b$ by
  $\gamma^{-1}_{\lambda(z)}\#b\#\gamma_{\rho(z)}$ for some $z\in \mathfrak
  B$.  All the above operations fix the corresponding elements in
  $\pi_1(\mathfrak B,\dagger,*)$.
\end{proof}

\begin{lem}\label{lem:gfCongClass}
  Consider a path connected space $X$ with a $1$-dimensional cover
  $(X_v)_{v\in V}$ and the associated graph of groups $\gf$ as in
  Definition~\ref{defn:gog_1dimcovers}.

  Then, up to congruence, $\gf$ and the graph of bisets associated with a
  topological correspondence are independent of the choices of basepoints
  $\{*_v\}$ and connecting paths $\{\gamma_e\}$,
\end{lem}
\begin{proof}
  Let $\gf$, $\gf'$ be two graphs of groups associated with $X$ and
  with a fixed $1$-dimensional cover $(X_v)_{v\in V}$ but with
  different choices of basepoints and connecting paths. Consider a
  topological correspondence $(X,f,i)$ such that $f\colon X\selfmap$
  and $i\colon X\selfmap$ are the identity maps. We may assume that
  $\gf'$, resp $\gf$, is the graph of groups associated with the range
  of $f$, resp $i$. Let $_{\gf}\mathfrak B_{\gf'}$ be the graph of
  bisets associated with $(X,f,i)$ and the cover $(X_v)_{v\in
    V}$. Then $\mathfrak B$ is a biprincipal graph of bisets because
  $B(f|_{X_v}, i\mid_{X_v})$ is biprincipal for every $v$.

  The second claim follows from the first.
\end{proof}

\subsection{Products of correspondences}\label{ss:product_corr}
Correspondences, just as continuous maps, may be composed; the
operation is given by fiber products.

Let $(Z_1,f_1,i_1)$ and $(Z_2,f_2,i_2)$ be two correspondences such
that $f_1\colon Z_1\to X$ and $i_2\colon Z_2\to X$ have the same range $X$. Their
product is the correspondence $(Z,f,i)$ given by
\[Z=\{(z_1,z_2)\in Z_1\times Z_2\mid f_1(z_1)=i_2(z_2)\},\quad
f(z_1,z_2)=f_2(z_2),\quad i(z_1,z_2)=i_1(z_1).
\]
We have natural maps $\widetilde{i_2}\colon Z\to Z_1$ and
$\widetilde{f_1}\colon Z\to Z_2$ given respectively by
\[\widetilde{i_2}(z_1,z_2)=z_1,\qquad \widetilde{f_1}(z_1,z_2)=z_2.\]
It is easy to check that $f$ is a fibration, respectively a covering, if
both $f_1$ and $f_2$ are fibrations, respectively coverings.

The biset of a product of two correspondences is, in favourable cases,
the product of the corresponding bisets:
\begin{lem}\label{lem:BisetOfProd}
  Let $(Z,f,i)$ be the product of two correspondences $(Z_1,f_1,i_1)$ and
  $(Z_2,f_2,i_2)$, with $f_1\colon Z_1\to X$ and $i_2\colon Z_2\to X$ and
  $X$ path connected. Then there is a biset morphism
  \begin{equation}\label{eq:LemBisetOfProd}
    \begin{cases}
      B(f,i) &\to B(f_1,i_1)\otimes B(f_2,i_2)\\
      (\delta,p,\gamma) & \mapsto(\delta,\widetilde{i_2}(p),
      \varepsilon^{-1})\otimes (\varepsilon,\widetilde{f_1}(p),\gamma),
    \end{cases}
  \end{equation}
  for any choice of path $\varepsilon$ from the basepoint of $X$ to
  $f_1(\widetilde{i_2}(p))=i_2(\widetilde{f_1}(p))$.\qed
\end{lem}

\begin{exple}
  In general, the map in the above lemma need not be an isomorphism;
  for instance, the ranges of $i_2$ and $f_1$ need not intersect, in
  which case $Z=\emptyset$ so $B(f,i)=\emptyset$, while $B(f_1,i_1)$
  and $B(f_2,i_2)$ are non-empty so their product is not empty.

  For a less artificial example, consider the correspondence shown on
  Figure~\ref{Fg:BadCorrProduct}. Denote by $1$ the trivial group; we
  also view $1$ as a set consisting of one element.  Then
  $B(i_1,f_1)={}_1\Z_\Z$, $B(i_2,f_2)={}_\Z 1_\Z$; thus
  $B(i_1,f_1)\otimes B(i_2,f_2)={}_11_\Z $. On the other hand,
  $B(\widetilde{i_2}i_1,\widetilde{f_1}f_2)={}_1Z_\Z$.

  Sufficient conditions on the map~\eqref{eq:LemBisetOfProd} being an
  isomorphism are given by Lemma~\ref{lem:BisetOfGoodProd}, in analogy
  with Lemma~\ref{lem:tensor commutes with pi1}.
\end{exple}

\begin{figure}
\label{Fg:BadCorrProduct}
\[\begin{tikzpicture}[dot/.style={}]

\draw [rotate around={ 55:(0.5,1.1) }] (0.5,1.1) ellipse (0.7 and 0.3);
\draw  (-0.9,0.5) ellipse (0.99 and 0.2);
\draw  (-2.3,-1) ellipse (0.99 and 0.2);
\draw [rotate around={ 55:(1,-1) }] (1,-1) ellipse (0.7 and 0.3);

\draw[->] (0.6,0.7) -- node[left] {$f_2$} (0.8,-0.6);
\draw[->] (-0.7,0.2) -- (0.5,-1.4);

\draw[->] (-1,0.2) --  node[right] {$i_2$} (-1.9,-0.7);
\draw[->] (0.3,0.4) -- (-1.3,-1.0);

\node (a) at (-2.4,2) {\textbullet};
\node (b) at (-3.9,0.5) {\textbullet};
\node (c) at (-1.6,0.63) {\textbullet};
\node (d) at (-4.9,-1) {\textbullet};
\node (e) at (-3,-0.87) {\textbullet};

\draw[->] (a) -- node [above left] {$\widetilde{i_2}$} (b);
\draw[->] (a) -- node [above right] {$\widetilde{f_1}$} (c);
\draw[->] (b) -- node [above left] {$i_1$} (d);
\draw[->] (b) -- node [above right] {$f_1$} (e);

\end{tikzpicture}\]
\caption{An example with $B(f,i) \nsimeq B(f_1,i_1)\otimes
  B(f_2,i_2)$.}
\end{figure}
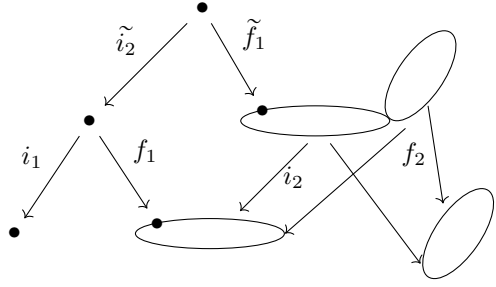

\begin{lem}\label{lem:BisetOfGoodProd} 
  If in Lemma~\ref{lem:BisetOfProd} at least one of the maps $i_2,f_1$
  is fibrant, then~\eqref{eq:LemBisetOfProd} is an isomorphism of
  bisets.
\end{lem}
\begin{proof}
  We construct an inverse to the map~\eqref{eq:LemBisetOfProd},
  keeping the notations of Lemma~\ref{lem:BisetOfProd}.  Suppose
  without loss of generality that $i_2$ is fibrant. Consider
  $b=(\delta_1,p_1,\gamma_1)\otimes(\delta_2,p_2,\gamma_2)\in
  B(f_1,i_1)\otimes B(f_2,i_2)$. Changing the basepoint in $X$ to
  $*=f_1(p_1)$ and denoting by $*$ the constant path in $X$ at $*$, we
  rewrite $b$ as
  $(\delta_1,p_1,*)\otimes(\gamma_1\#\delta_2,p_2,\gamma_2)$. Since
  $i_2$ is fibrant, the curve $\gamma_1\#\delta_2$ admits a lift
  $\varepsilon$ to $Z_2$ with $\varepsilon(1)=p_2$. Write
  $p'_2=\varepsilon(0)$. We then have $b=(\delta_1,p_1,*)\otimes
  (*,p'_2,(f_2\circ\varepsilon)\#\gamma_2)$, see
  Lemma~\ref{lmm:FibrCorr}. Since $f_1(p_1)=*=i_2(p'_2)$, there is a
  unique $p\in Z$ with $\widetilde{i_2}(p)=p_1$ and
  $\widetilde{f_1}(p)=p'_2$. Then
  $(\delta_1,p,(f_2\circ\varepsilon)\#\gamma_2)\in B(f,i)$ is a
  preimage of $b$ under~\eqref{eq:LemBisetOfProd}.
\end{proof}

\subsection{Fibrant maps and covers}
Suppose $f\colon Z\to X$ is a fibrant map between path connected spaces and
$(U_\alpha)$ is a cover of $X$ consisting of path connected sets
$U_\alpha$. The \emph{pullback} $(W_\gamma)\coloneqq f^*(U_\alpha)$ is the
cover of $Z$ consisting of all path connected components of
$f^{-1}(U_\alpha)$, for all $U_\alpha$ in the cover. It is easy to see
that if $(U_\alpha)$ is a finite cover and $f^{-1}(*)$ consists of
finitely many connected components for some or equivalently any $*\in
X$, then $(V_\alpha)$ is a finite cover. If, furthermore, $(U_\alpha)$
is $1$-dimensional, then so is $(W_\gamma)$. It is also immediate that
all $f\restrict{W_k}$ are fibrant maps.
\begin{prop}\label{prop:DecompFibrMaps}
  Suppose ${}_{\mathfrak Y}\mathfrak B_{\mathfrak X}$ is the graph of
  bisets of a topological correspondence $(Z,f,i)$ from $Y$ and $X$
  subject to finite $1$-dimensional covers $(U_\alpha)$, $(V_\beta)$,
  and $(W_\gamma)$ of $X$, $Y$, and $Z$ respectively.
  
  If $f\colon Z\to X$ is fibrant, respectively a covering, and
  $(W_\gamma)=f^*(U_\alpha)$, then ${}_{\mathfrak Y}\mathfrak B_{\mathfrak
    X}$ is a left-fibrant, respectively a left-free, graph of bisets.
\end{prop}

\noindent The proof of the above proposition is based on the following
property:
\begin{lem}\label{lem:LiftInFibrBis}
  Let
  \begin{itemize}
  \item $f\colon W\to U$ be a fibrant map between path connected spaces,
    with biset $B(f,\dagger, *)$;
  \item $U'\subset U$ be path connected with basepoint $*'$;
  \item $\{W'_{z}\}_{z\in I}$ be the set of connected components of $f^{-1}(U')$;
  \item $B(f\restrict{W'_z},\dagger'_z, *')$ be the biset of $f\colon W'_z\to U'$;
  \item $()^-\colon {}_{\pi_1(W'_z,\dagger'_z )}B(f\restrict{W'_z},\dagger'_z, *') _{\pi_1(U',*') }\to {}_{\pi_1(W,\dagger)}B(f,\dagger, *)_{\pi_1(U,*)}$ be congruences as in Lemma~\ref{lm:BisMorph}.
  \end{itemize}
  Then we have a $\pi_1(W,\dagger)$-$\pi_1(U',*)$ isomorphism 
  \begin{equation}\label{eq:lem:LiftInFibrBis}
    \bigsqcup_{z\in I} \pi_1(W,\dagger) \otimes_{\pi_1(W'_z,\dagger'_z )^-} B(f\restrict{W'_z}, \dagger'_z, *') \longrightarrow B(f\restrict{W}, \dagger, *), \qquad g\otimes b\mapsto g b^-
  \end{equation}  
  with right action of $\pi_1(U',*')$ on $B(f\restrict{W}, \dagger, *)$
  given via $()^-\colon\pi_1(U',*')\to\pi_1(U,*)$.
\end{lem}
\begin{proof}
  The statement is clearly stable under motion of $*$, thus we may assume
  that $*=*'$ and that $()^-\colon\pi_1(U', *')\to\pi_1(U, *)$ is given by
  taking each $\gamma\in \pi_1(U', *')$ modulo homotopy in $U$. By
  Lemma~\ref{lmm:FibrCorr}
  \[B(f,\dagger, *)=\{\beta\colon[0,1]\to W\mid\beta(0)=\dagger,f(\beta(1))=*\}/{\sim},
  \]
  with $\beta\sim \beta'$ if and only if there is a path
  $\varepsilon\colon[0,1]\to f^{-1}(*)$ connecting $\beta(1)$ to
  $\beta'(1)$ such that $\beta\#\varepsilon$ is homotopic to
  $\beta'$. Therefore, each $\beta\colon [0,1]\to W $ in $B(f)$ ends at a
  unique $W'_z$ independent on the choice representing $\beta$. It is now
  easy to see that the $\pi_1(W,\dagger)$-$\pi_1(U',*)$ subbiset of
  $B(f,\dagger, *)$ consisting of all $b\in B(f,\dagger, *)$ terminating at
  $W'_z$ is exactly $\pi_1(W,\dagger) \otimes_{\pi_1(W'_z,\dagger'_z )^-}
  B(f\restrict{W'_z}, \dagger'_z, *')$. This
  implies~\eqref{eq:lem:LiftInFibrBis}.
 \end{proof}

\begin{proof}[Proof of Proposition~\ref{prop:DecompFibrMaps}]
  Let $(Z_k)$ be the path connected components of $Z$. By definition,
  ${}_{\mathfrak Y}\mathfrak B_{\mathfrak X} = \bigsqcup_{k} {}_{\mathfrak
    Y}\mathfrak B(i\restrict{Z_k})^\vee_{\mathfrak Z_k} \bigotimes_{\mathfrak
    Z_k} {}_{\mathfrak Z_k}\mathfrak B(f\restrict{Z_k})_{\mathfrak
    X}$. Lemma~\ref{lem:LiftInFibrBis} implies that every ${}_{\mathfrak
    Z_k}\mathfrak B(f\restrict{Z_k})_{\mathfrak X}$ is left-fibrant:
  Equation~\eqref{eq:lem:LiftInFibrBis} is
  exactly~\eqref{eq:dfn:FibrationGrBis}.  Since the product of left-fibrant
  bisets is left-fibrant (Lemma~\ref{lem:ProdOfFibrGrBis}), ${}_{\mathfrak Y}\mathfrak B_{\mathfrak X}$ is a
  left-fibrant biset. This proves the part of
  Proposition~\ref{prop:DecompFibrMaps} concerning fibrant maps. The case
  of covering maps observe that all bisets in $\mathfrak B(f\restrict{Z_k})$ are left-free because $f\restrict{Z_k}$ are coverings while all bisets in $\mathfrak B(i\restrict{Z_k})^\vee$ are left-principal; hence all 
  bisets in $\mathfrak B_{\mathfrak X} = \bigsqcup_{k} {}_{\mathfrak
    Y}\mathfrak B(i\restrict{Z_k})^\vee_{\mathfrak Z_k} \bigotimes_{\mathfrak
    Z_k} {}_{\mathfrak Z_k}\mathfrak B(f\restrict{Z_k})$ are left-free. 
\end{proof}

\subsection{Partial self-coverings}
\label{ss:PartSelfCov}
We now turn to a restricted class of covering correspondences $(Z,f\colon
Z\to X, i\colon Z\to Y)$, in which the map $i$ is an inclusion, namely a
homeomorphism on its image. We then view $Z$ as a subset of $Y$, and write
the correspondence as a \emph{partial covering} $f\colon Y\dashrightarrow
X$. Here the dashed arrow means that the map is defined on the image of
$i$.

In this case, the definition of the biset of a correspondence
(see~\S\ref{ss:biset of correspondence}) can be simplified as follows:
\begin{lem}
  Let $f\colon X\dashrightarrow X$ be a partially defined self-covering. The
  biset $B(f)$ may then be constructed directly as follows:
  \[B(f)=\{\gamma\colon[0,1]\to X\mid\gamma(0)=*=f(\gamma(1))\}/{\sim}.
  \]
  The left action is by pre-composition by loops in $X$ at $*$, and
  the right action is by lifting loops through $f$:
  \[[\alpha]\cdot[\gamma]=[\alpha\#\gamma],\qquad[\gamma]\cdot[\alpha]=[\gamma\#\tilde\alpha]
  \]
  for the unique $f$-lift $\tilde\alpha$ of $\alpha$ that starts at
  $\gamma(1)$.
\end{lem}
\begin{proof}
  Simply remark that, in the notation of Lemma~\ref{lem:cpbiset}, the
  point $p$ is determined as $i^{-1}(\gamma(1))$, so may be removed
  from the definition.
\end{proof}

Let us consider now a self-correspondence
$(Z,f\colon Z\to X, i\colon Z\to X)$. It may arise as follows:
$f\colon X\selfmap$ is a branched covering, $Z$ is the subset of $X$
on which $f$ is a genuine covering, and $i$ is the inclusion.

Such correspondences may be iterated; however, their iterates are not
defined on the same $Z$ anymore. We set $X^{(1)}\coloneqq Z$ and, more
generally,
\[X^{(n)}=\{(x_1,\dots,x_n)\in Z^n \mid f(x_{i})=i(x_{i+1})\}.
\]

Define the maps $f^{(n)}(x_1,\dots,x_n)=f(x_n)$ and
$i^{(n)}(x_1,\dots,x_n)=i(x_1)$. Then the $n$-th power, in the sense
of \S\ref{ss:product_corr}, of $(f,i)$ is $(f^{(n)},i^{(n)})$.

If $f$ is fibrant, then so are all $f^{(n)}$, and naturally
(Lemma~\ref{lem:BisetOfGoodProd}), the biset of $(f^{(n)},i^{(n)})$ is
none other than $B(f,i)^{\otimes n}$.

If $i\colon X^{(1)}\to X$ is an inclusion, then so are all
$i^{(n)}\colon X^{(n)}\to X$.  Then $(f,i)$ is identified with a partially
defined map $f\colon X\dashrightarrow X$; and all
$(f^{n}\colon X^{(n)}\to X,i^{n}\colon X^{(n)}\to X)$ are identified with
$f^{(n)}\colon X\dashrightarrow X$. If $f$ is fibrant, then the bisets of
$(f^{(n)},i^{(n)})$ and of $f^{(n)}$ are naturally isomorphic.

Iteration may also be interpreted purely in the language of
homomorphisms, using Lemma~\ref{lem:biset=fibre}. Given a
$G$-$G$-biset $B$, write $G_0=G$, and find a group $G_1$ such that $B$
decomposes as $B_{\psi_1}^\vee\otimes_{G_1}B_{\phi_1}$, for
homomorphisms $\phi_1,\psi_1\colon G_1\to G_0$. Define then iteratively
$G_{n+1}$ as the fibre product of $G_n$ with $G_n$ over $G_{n-1}$:
\[\begin{tikzpicture}[description/.style={fill=white,inner sep=2pt}]
  \matrix (m) [matrix of math nodes, row sep=3em,
  column sep=2.5em, text height=1.5ex, text depth=0.25ex]
  { {G_{n+1}} & {G_n}\\
    {G_n} & {G_{n-1}}\\};
  \path[->,font=\scriptsize]
  (m-1-1) edge[dotted] node[above] {$\phi_{n+1}$} (m-1-2)
  (m-1-1) edge[dotted] node[left] {$\psi_{n+1}$} (m-2-1)
  (m-1-2) edge node[description] {$\psi_n$} (m-2-2)
  (m-2-1) edge node[description] {$\phi_n$} (m-2-2);
\end{tikzpicture}
\]
Then $B^{\otimes n}=(B_{\psi_n\cdots\psi_1})^\vee\otimes_{G_n}B_{\phi_n\cdots\phi_1}$.

\subsection{Generic maps}\label{ss:generic}
We can slightly relax the previous setting, in which $i\colon Z\to Y$
is an inclusion, to ``generic'' maps in the following sense.

\begin{defn}[Generic maps]
  A continuous map $f\colon Y\to X$ is \emph{generic} if there exists a
  continuous map $g\colon X\to Y$ such that $f\circ g$ is isotopic to the
  identity on $X$.
\end{defn}

Here is a typical example: Consider a topological space $X$ and an
injective path $\gamma\colon[0,1]\to X$, such that a neighbourhood of the
image of $\gamma$ is contractible in $X$. Then the inclusion
$f\colon X\setminus\{\gamma(0),\gamma(1)\}\to X\setminus\{\gamma(0)\}$ is a
generic map. Indeed contractibility of the image of $\gamma$ implies
the existence of a homeomorphism
$g\colon X\setminus\{\gamma(0)\}\to X\setminus\gamma([0,1])$.

\begin{lem}
  Let $f\colon Y\to X$ be a generic map, and let $\dagger\in Y$ be a
  basepoint. Write $f(\dagger)=*$. Then
  $f_*\colon\pi_1(Y,\dagger)\to\pi_1(X,*)$ is a split epimorphism.

  Therefore, the biset of $f$ is left-invertible and right-principal;
  namely, it is a $\pi_1(Y,\dagger)$-$\pi_1(X,*)$-biset $B$ such that
  $B^\vee\otimes_{\pi_1(Y,\dagger)}B\cong{}_{\pi_1(X,*)}\pi_1(X,*)_{\pi_1(X,*)}$.
\end{lem}
\begin{proof}
  Let $g$ be a homotopy left inverse of $f$. The first statement
  follows simply from $g_*f_*=id_X$. The second one follows from the
  first.
\end{proof}

\subsection{Conjugacy classes in bisets}\label{ss:cc in bisets}
Let ${}_GB_G$ be a biset. Its set of conjugacy classes is
\[B^{G}\coloneqq{}_GB_G/\{b=g b g^{-1} \mid b\in B, g\in G\}.\]

Consider a self-correspondence $f,i\colon Z\rightrightarrows X$. A homotopy
\emph{pseudo-fixed point}~\cite{ishii-smillie:shadowing} is the data
$(p,\gamma)$ such that $p \in Z$ and $\gamma\colon [0,1]\to X$ with
$\gamma(0)=f(p)$ and $\gamma(1)=i(p)$. In other words, $\gamma$ encodes a
homotopy difference between $f(p)$ and $i(p)$. If $\gamma$ is a constant
path, then $p$ is a \emph{fixed point of $(f,i)$}. Two homotopy
pseudo-fixed points $(p,\gamma)$ and $(q,\delta)$ are \emph{conjugate} if
there is a path $\ell\colon [0,1]\to Z$ with $\ell(0)=p$, $\ell(1)=q$ such
that $f(\ell)\#\beta \# i(\ell^{-1})$ is homotopic to $\gamma$.

The set of fixed points conjugate to a given fixed point $p\in Z$ is also
known as the \emph{Nielsen class} of $p$. The \emph{Nielsen number}
$N(f,i)$ is the set of fixed points of $(f,i)$ considered up to conjugacy.

Every $(\delta^{-1}, p, \gamma )\in B(f,i)$, in the notation of
Lemma~\ref{lem:cpbiset}, naturally defines a homotopy pseudo-fixed point
$(p,\gamma\#\delta)$. Conversely, if $X$ is path connected, then for every
homotopy pseudo-fixed point $(p,\gamma)$ we may choose a path $\ell\colon
[0,1]\to X$ with $\ell(0)=*$ and $\ell(1)=f(p)$ and construct an element
$(\ell^{-1},p,\gamma\#\ell)\in B(f,i)$ encoding $(p,\gamma)$.

\noindent The following proposition is immediate. 
\begin{prop}
  Two elements $(\delta^{-1}, p, \gamma), (\delta'^{-1}, p', \gamma')\in
  B(f,i)$ are conjugate as elements of the biset $B(f,i)$ if and only if
  the homotopy fixed points $(p,\gamma\#\delta),(p',\gamma'\#\delta')$ are
  conjugate.
\end{prop}

As a corollary, if $X$ is path connected, then the Nielsen number $N(f,i)$
is bounded by the cardinality of $B^{G}$.

In \cite{bartholdi-dudko:bc4} we will further investigate homotopy
pseudo-periodic orbits of a Thurston map.

\section{Dynamical systems}\label{ss:dynamics}
We turn now to applications of the previous sections. They will mainly
be to the theory of iterations of branched self-coverings of
surfaces. The main objective is an algorithmic understanding of these
maps up to isotopy, and will be developed in later articles. Here are
some more elementary byproducts.

Consider first a polynomial $p(z)\in\C[z]$. We recall some basic
definitions and properties; see~\cites{douady-h:edpc1,douady-h:edpc2}
for details.

\noindent Let us denote by $\widetilde P(p)\subset \C$ the forward orbit of
$p$'s critical points:
\[\widetilde P(p)\coloneqq\{p^n(z)\mid p'(z)=0,\,n\ge0\, z\in \C\}.\]
The \emph{post-critical set} $P(p)\coloneqq p(\widetilde P(p))$ is the
forward orbit of $p$'s critical values. The polynomial $p$ is
\emph{post-critically finite} if $P(p)$ is finite. The \emph{Julia set}
$J(p)$ of $p$ is the closure of the repelling periodic points of
$p$. Equivalently, $J(p)$ is the boundary of the filled-in Julia set
$K(p)$:
\[K(p)\coloneqq\{z\in \C\mid P(p)\text{ is bounded}\},\quad J_c=\partial K_c.
\]
The \emph{Fatou set} $F(p)$ is the complement of $J(p)$.

Let us assume that $p$ is post-critically finite. This assumption is
essential to obtain simple combinatorial descriptions of $p$. Then each
bounded connected component of $F(p)$ is a disk, and $p$ acts on this set
of disks. Furthermore, the boundary of each disk component is a circle in
$J(p)$.

Furthermore, every grand orbit of $p\colon \widetilde P(p)\to\widetilde
P(p)$ contains a unique periodic cycle. If a periodic cycle $C\subset
\widetilde P(p)$ contains a critical point of $p$, then $C$ lies entirely
in the Fatou set, and moreover each element $c\in C$ lies in a different
disk of $F(p)$. On the other hand, if a cycle $C$ contains no critical
point of $p$, then $C$ lies in $J(p)$. If no cycle of $P(p)$ contains a
critical point of $p$, then $F(p)$ has no bounded component and $J(p)$ is a
\emph{dendrite}.

Let $\HT$ be the smallest tree in $K(p)$ containing $\widetilde P(p)$ and
containing $P(p)$ in its vertex set such that $\HT$ intersects $F(p)$ along
radial arcs. Since $\widetilde P(p)$ is forward invariant we get a self-map
$p\colon\HT\righttoleftarrow$. For every pair of adjacent edges $e_1, e_2$
there is the angle $\sphericalangle(e_1,e_2)\in \Q/\Z$ uniquely specified
by the following conditions, see~\cite{poirier:trees}:
\begin{itemize}
\item $\sphericalangle(e_1,e_2)=0$ if and only if $e_1=e_2$; 
\item if $v$ is a common vertex of $e_1$ and $e_2$ and $\deg_v(p)$ is the
  local degree of $p$ at $v$, then
  \[\sphericalangle(p(e_1),p(e_2))=\deg_v(p)\sphericalangle(e_1,e_2);\]
\item for all edges $e_1, e_2,e_3$ adjacent to a common vertex we have
  \[\sphericalangle(e_1,e_3)=\sphericalangle(e_1,e_2)+\sphericalangle(e_2,e_3).\] 
\end{itemize}

\begin{defn}[Hubbard trees]
  The \emph{(angled) Hubbard tree} of a complex polynomial $p$ is the data
  consisting of $p\colon\HT\righttoleftarrow$, the values $\deg_p(v)\in \N$
  measuring the local degrees of $p$ at vertices $v\in p^{-1}(\HT)$, and the angle
  structure ``$\sphericalangle$'' on $\HT$ satisfying the above axioms.
\end{defn}

In~\cite{poirier:trees} post-critically finite polynomials are classified
in terms of their Hubbard trees.

\begin{rem}
  If $p$ has degree $2$, then the angled structure of
  $p\colon\HT\selfmap$ is uniquely determined by how $\HT$ is
  topologically embedded into the plane. However, in general, the
  planarity of $\HT$ is insufficient to recover $p$; one needs to
  endow $p \colon\HT\selfmap$ with extra information sufficient to
  recover the embedding of $p^{-1}(\HT)$ into the plane.
\end{rem}

It may be more convenient to consider a variant, the \emph{Hubbard
  complex}, which is a special case of topological automaton
(see~\cite{nekrashevych:combinatorialmodels}*{\S3}):

\begin{defn}[Hubbard complexes]
  Let $p$ be a complex polynomial. Its \emph{Hubbard complex} is the
  self-correspondence $p,i\colon H^1\rightrightarrows H^0$ defined as
  follows:
  \begin{itemize}
  \item $H^0$ is the smallest $1$-dimensional subcomplex of the plane
    containing $P(p)\cap J(p)$ and all $\partial D$ for $D\subseteq F(p)$ a
    Fatou component intersecting $P(p)$;
  \item $H^1$ is the smallest $1$-dimensional subcomplex of the plane
    containing $P(p)\cap J(p)$ and all $\partial D$ for $D\subseteq F(p)$ a
    Fatou component intersecting $p^{-1}(P(p))$;
  \item $p\colon H^1\to H^0$ is the restriction of $p$ to $H^1$ and is a
    covering map;
  \item $i\colon H^1\to H^0$ retracts $H^1$ into $H^0$.\qedhere
  \end{itemize}
\end{defn}
\noindent (As for the Hubbard tree, we assume that $H^1, H^0$ intersects
$F(p)$ along radial lines. Points in $H^{0}\cap P(p)$ and in $H^{1}\cap \widetilde P(p)$ are treated as orbifold points.)

\begin{figure}[h]
\begin{center}
  \begin{tikzpicture}
    \begin{scope}[xshift=0cm,yshift=-1cm]
  \node[anchor=south west,inner sep=0] (basilica) at (0,0)       {\includegraphics[width=0.5\textwidth]{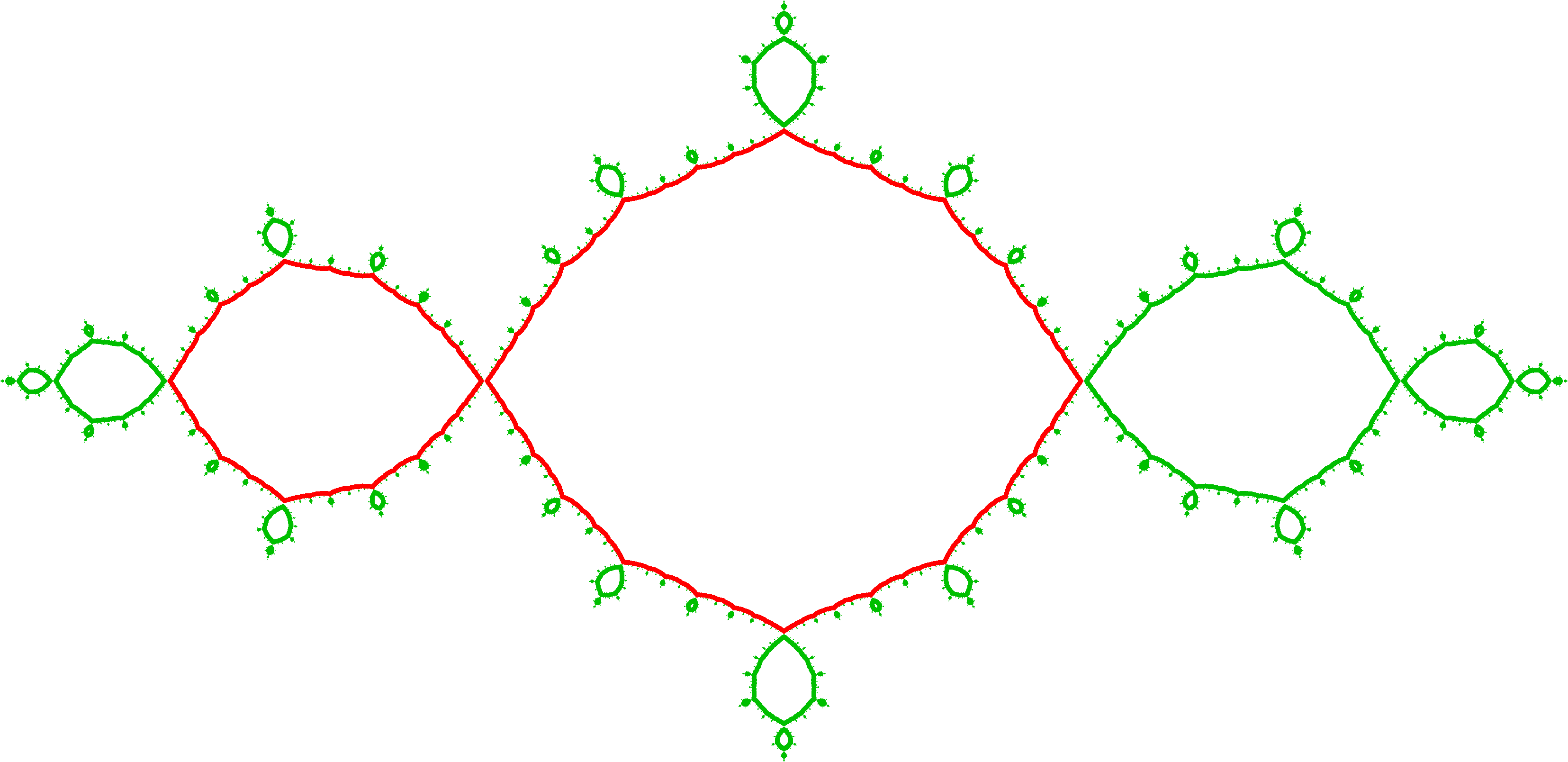}};
  \begin{scope}[x={(basilica.south east)},y={(basilica.north west)}]
    \draw[thin] (0.5,0.5) -- (0.2,0.5);
    \coordinate (alpha_basilica) at (0.31,0.5);
  \end{scope}
\end{scope}

\begin{scope}[xshift=0.5\textwidth,yshift=0cm]
  \node[anchor=south west,inner sep=0] (rabbit) at (0,0) {\includegraphics[width=0.45\textwidth]{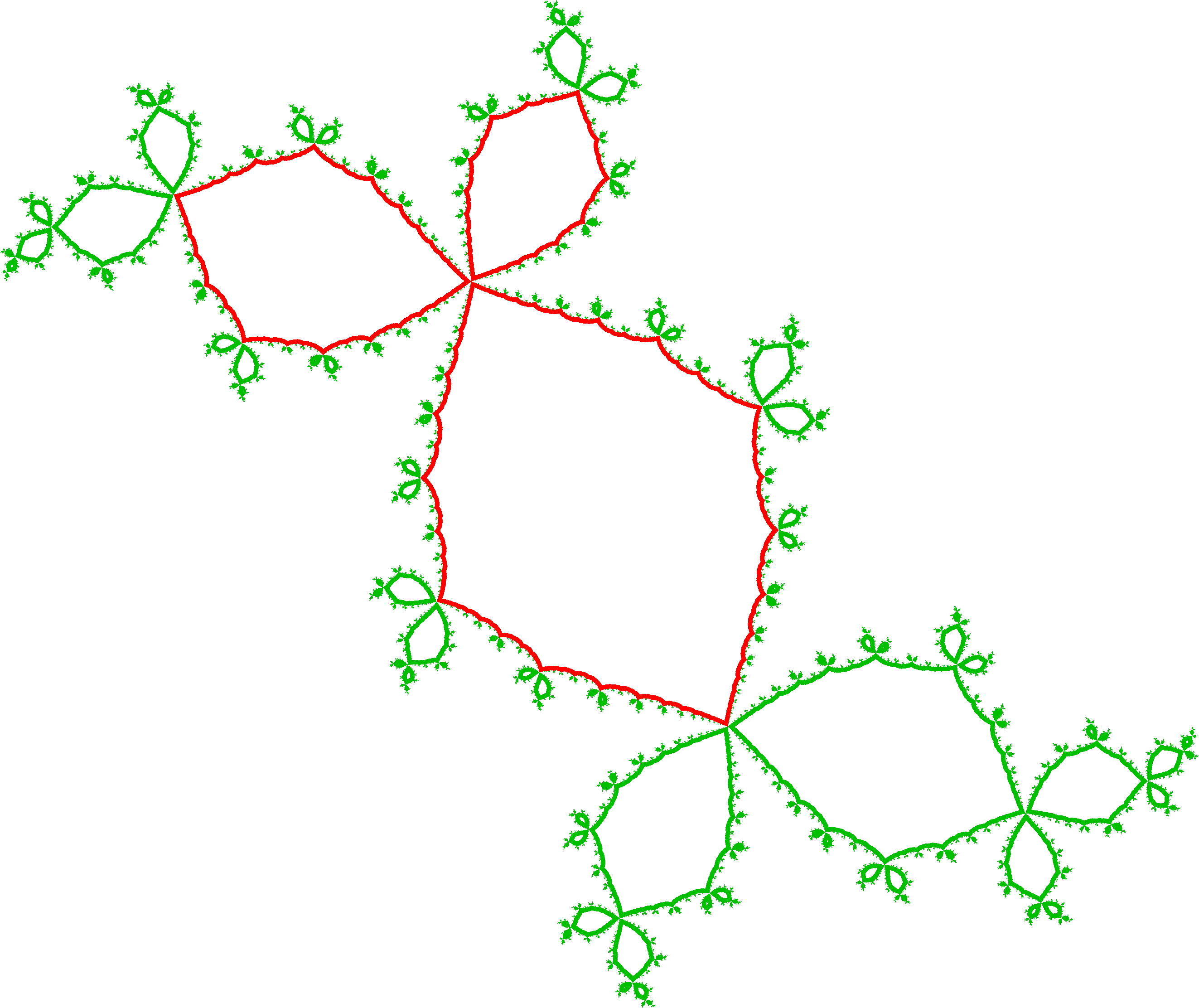}};
  \begin{scope}[x={(rabbit.south east)},y={(rabbit.north west)}]
    \coordinate (alpha_rabbit) at (0.3934,0.7204);
    \draw[thin] (0.5,0.5) -- (alpha_rabbit);
    \draw[thin] (0.452,0.843) -- (alpha_rabbit);
    \draw[thin] (0.244,0.759) -- (alpha_rabbit);
  \end{scope}
\end{scope}

\begin{scope}[xshift=15mm,yshift=2cm]
\node[anchor=south west,inner sep=0] (z2pi) at (0,0) {\includegraphics[width=0.35\textwidth]{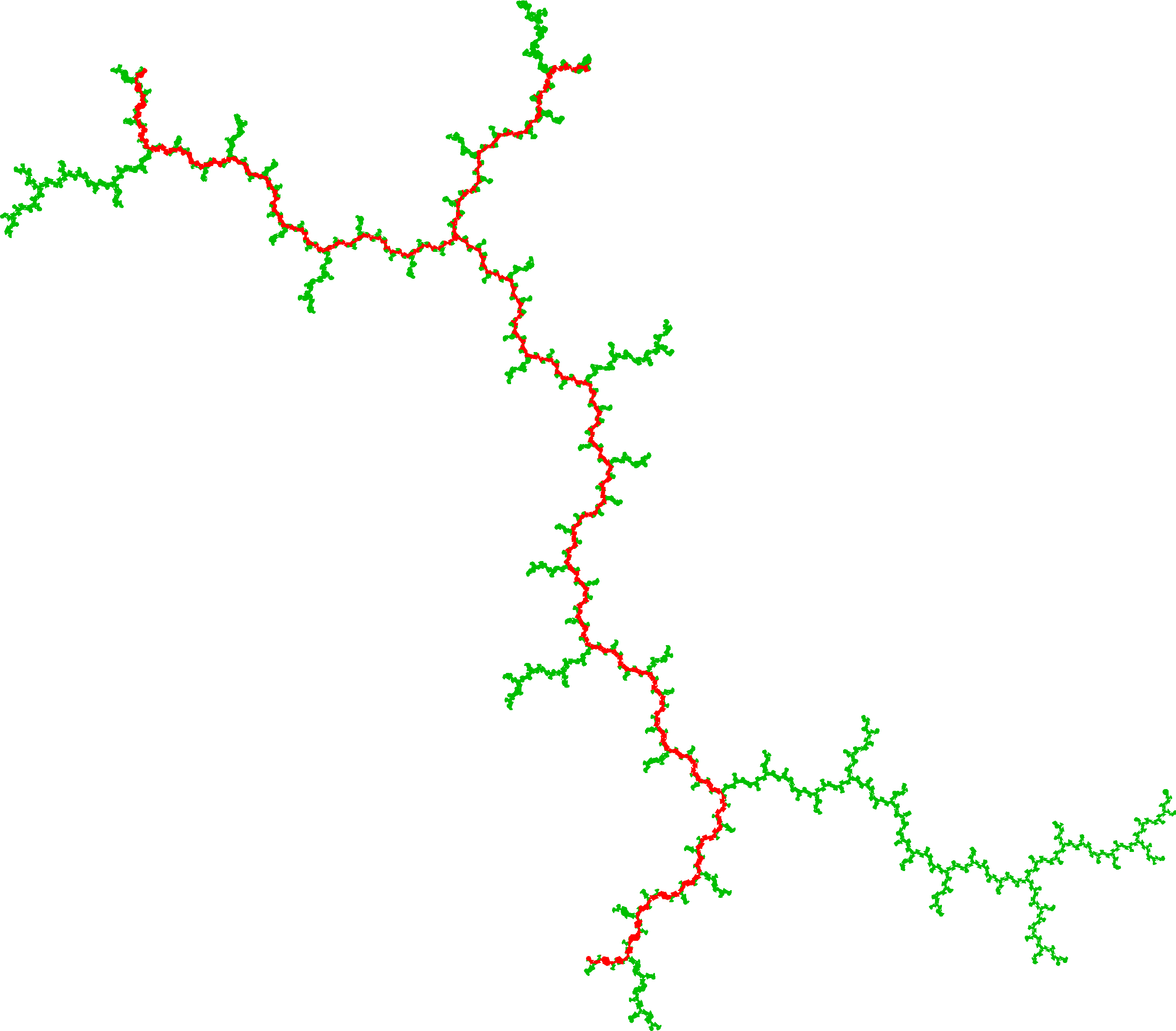}};
  \begin{scope}[x={(z2pi.south east)},y={(z2pi.north west)}]
    \coordinate (alpha_z2pi) at (0.387,0.232);
  \end{scope}
\end{scope}
\end{tikzpicture}
\end{center}
\caption{The Julia sets of the maps $z^2+i$, $z^2-1$ (the ``Basilica'') and $z^2+c$ for $(c^2+c)^2+c=0$ (the ``Rabbit''). The Hubbard complex is drawn in red, and (if it differs) the Hubbard tree is drawn in black.}
\end{figure}
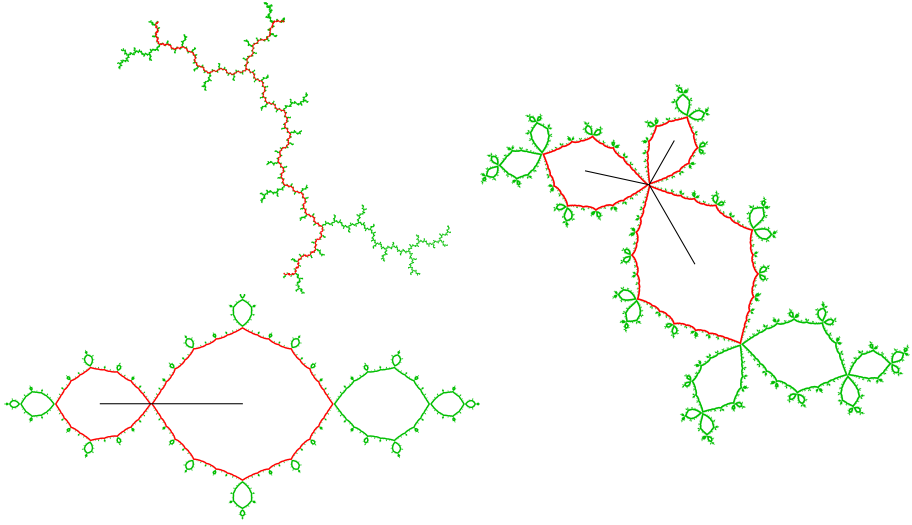

\subsection{Thurston maps}\label{ss:thurston}
We may consider the more general situation of a \emph{branched
  self-covering} of the sphere $S^2$, namely a map
$p\colon S^2\selfmap$ that is locally modelled, in complex charts, by
$z\mapsto z^n$ for some $n\in\N$. Those points $z\in S^2$ at which
$n\ge2$ are \emph{critical points}, and $P(p)$ is the forward orbit of
$p$'s critical values.  If furthermore there is a point
$\infty\in S^2$ with $p^{-1}(\infty)=\{\infty\}$, then $p$ is a
topological polynomial. If $P(p)$ is finite, then $p$ is called a
\emph{Thurston map}.

Unless $p$ expands a metric on $S^2$, there is no well-defined notion
of Julia set. There is, however, a convenient encoding of $p$ by a
biset. One sets $X=Y=S^2\setminus P(p)$ and
$Z=S^2\setminus p^{-1}(P(p))$, with maps $i\colon Z\to Y$ the
inclusion and $p\colon Z\to X$ the restriction of $p$. Thus
$p\colon S^2\selfmap$ is given by a covering correspondence
$p,i\colon Z\rightrightarrows X$.

Fix a basepoint $*\in X$, and write $G=\pi_1(X,*)$. Let $B(p)$ denote
the biset of the above correspondence associated with $p$; it is a
$G$-$G$-biset, left-free of degree $\deg(p)$.

Let $p_0,p_1$ be Thurston maps. They are called \emph{combinatorially
  equivalent} if there is an isotopy $(p_t)_{t\in[0,1]}$ from $p_0$ to
$p_1$ along which $P(p)$ moves continuously (and, in particular, has
constant cardinality). The biset of a Thurston map is a complete
invariant for combinatorial equivalence:
\begin{thm}[Kameyama~\cite{kameyama:thurston}, Nekrashevych~\cite{nekrashevych:ssg}*{Theorem~6.5.2}]
\label{thm:Kameyama}
  Let $p_0,p_1$ be Thurston maps. Then $p_0,p_1$ are Thurston
  equivalent if and only if the bisets $B(p_0),B(p_1)$ are conjugate
  by an isomorphism
  $\pi_1(S^2\setminus P(p_0),*_0)\to\pi_1(S^2\setminus P(p_1),*_1)$
  induced by a surface homeomorphism.
\end{thm}

By definition, $p$ behaves locally as $z\mapsto z^{\deg_z(p)}$ at a point
$z\in S^2$; set
\[\ord(v)=\lcm\{\deg_z(p^n)\mid n\ge0,z\in p^{-n}(v)\}.\]

Clearly, $\ord(z)>1$ if and only if $z\in P(p)$, and $\ord(v)=\infty$ if
and only if $v$ is in a periodic cycle containing a critical point. For
example, the degrees at $P(p)$ are all $\infty$ for the Basilica and the
Rabbit maps, and are $2$ at $i,-i,i-1$ for $z^2+i$.

This order function $\ord$ defines an \emph{orbispace} structure on
$S^2$: in our simplified context, a topological space with the extra
data of a non-trivial group attached at a discrete set of points. We
will not go into details of orbispaces, but simply note that, if $v$
is a point with group $G_v$ attached to it, then $v$ has canonical
neighbourhoods with fundamental group isomorphic to $G_v$. In our
situation, the group attached to $v\in P(p)$ is cyclic of order
$\ord(v)$. If $\ord(v)=\infty$, then the point $v$ may be treated as a
puncture rather than as a point with $\Z$ attached to it.

For each $z\in P(p)$, let $\gamma_z$ denote a small loop around $z$, and
identify $\gamma_z$ with a representative of a conjugacy class in
$\pi_1( S^2\setminus P(p),*)$. It follows that the fundamental group of the
orbispace defined by $\ord$ is given as follows:
\begin{equation}
\label{eq:G_p}
G_p = \pi_1( S^2\setminus P(p),*)/\langle \gamma_z^{\ord(z)}:z\in P(p)\rangle.
\end{equation}
For every $z\in S^2$ define $\ord^1(z)\coloneqq \ord(p(z))/\deg_z(p)$. Then
$p\colon (S^2,\ord^1)\to (S^2,\ord)$ is a covering between orbispaces while
$(S^2,\ord^1)\hookrightarrow (S^2,\ord)$ is an orbispace inclusion. This
defines the $G_p$-biset $B(p)$. Note that Theorem~\ref{thm:Kameyama} was
stated for bisets over the group $\pi_1(S^2\setminus P(p))$, not for bisets
over the orbispace fundamental group $G_p$. However, an analogue of
Theorem~\ref{thm:Kameyama} is also true in that context; this will be
proven in~\cite{bartholdi-dudko:bc2}.

Much structure in the space of Thurston maps can be obtained by comparing,
or deriving, maps from the simplest example $f(z)=z^d$. This map has
$J(f)=\{|z|=1\}$ and $P(f)=\{0,\infty\}$ so that $\pi_1(\C\setminus
P(f),*)=\langle t\rangle\cong\Z$. We call the corresponding biset the
\emph{regular cyclic biset of degree $d$}: as a left $\langle t\rangle$-set, it is
$\langle t\rangle\times\{1,\dots,d\}$; and the right action on the basis
$\{1,\dots,d\}$ is given by
\[1\cdot t=2,\quad\dots,\quad(d-1)\cdot t=d,\quad d\cdot t=t\cdot 1.\] It
may also be defined more economically as $B(z^d)=\{t^{j/d}:j\in\Z\}$, with
left and right actions given by $t^i\cdot t^{j/d}\cdot t^k=t^{i+j/d+k/d}$.

Since topological polynomials of degree $d$ behave as $z^d$ in a
neighbourhood of $\infty$, their bisets contain a copy of the
regular degree-$d$ cyclic biset. More generally, if $p$ has a fixed point in a
neighbourhood of which it acts as $z\mapsto z^n$, then $p$ contains a
regular degree-$n$ cyclic subbiset.

The graph of bisets decompositions that we shall consider essentially
attempt to describe bisets of Thurston maps in terms of cyclic bisets.

\subsection{(Graphs of) bisets from Hubbard trees}\label{ss:hubbardtrees}
Let us consider the (angled) Hubbard tree $p\colon \HT\righttoleftarrow$ of
a complex polynomial $p$. We will now apply Van Kampen's theorem to
$p\colon \HT\righttoleftarrow$ to decompose $B(p)$ as a graph of bisets
${}_\gf\mathfrak T_\gf$ as in Example~\ref{exm:Basil}.

Let $\HT^1$ be the preimage of $\HT$ under $p\colon\C\selfmap$. We
note that $\HT^1$ is easily reconstructable from the data
$p\colon\HT\selfmap$ and $\sphericalangle$ and $\deg_v(p)$ for all
$v\in \HT$. The angled structure of $\HT$ lifts via
$p\colon \HT^1\to\HT$ to an angled structure on $\HT^1$. We denote by
$\iota\colon\HT^1\to\HT$ the natural retraction of $\HT^1$ into its
subtree. We write $\lambda(z)=v$ if an object $z\in \HT^1$ retracts
into a vertex $v\in \HT$ and we write $\lambda(z)=e$ if an object
$z\in \HT^1$ retracts into a subset of an edge $e\in \HT$ but
$\iota(z)$ is not a vertex of $\HT$. This defines a graph morphism
$\lambda\colon \HT^1\to\HT$ between undirected graphs,
see\S\ref{ss:examplesgfgps} such that the images of adjacent objects
are adjacent or equal objects.

Let $X$ be the space obtained from $\HT$ by blowing up each vertex $v\in
\HT$ into a closed unit disc $D_v$ with $\partial D_v\cong \R/\Z$ and by
blowing down each edge $e$, say adjacent to vertices $v$ and $w$, into a
point $D_e$ with $\{D_e\}=D_v\cap D_w$ such that for every pair of edges
$e_1,e_2$ adjacent to a common vertex $v$ we have
\[\sphericalangle(e_1,e_2)=D_{e_1}-D_{e_2}\in \R/\Z \cong \partial D_v .\]
For every vertex $v\in \HT$ put an orbifold point of order $\ord(v)$ at the
center of $D_v$, so that the fundamental group of $D_v$ is a cyclic group
of order $\ord(v)$.

Similarly, let $X^1$ be the space obtained from $\HT^1$ by blowing up each
vertex $v\in \HT^1$ into a closed disc $D_v$ with orbifold point of order
$\ord^1(v)$ at the center of $D_v$ and by blowing down each edge $e$ into a
point $D_e$ satisfying same properties as above.

Then $p\colon\HT^1\to\HT$ naturally induces a covering map $p\colon X^1\to
X$ specified so that all maps between unit discs are of the form $z\to
z^d$. Furthermore, we have a natural retraction $\iota\colon X^1\to X$
satisfying $\iota(D_z)\subset D_{\lambda(z)}$ for every object $z\in
\HT^1$. Applying Van Kampen's theorem Theorem~\ref{thm:vankampenbis} to the
covering pair $(X^1,p, \iota)$ subject to covers $\{D_z\}_{z\in \HT^1},
\{D_z\}_{z\in \HT}, p, \lambda\colon \HT^1\to\HT$ we get the graph of
bisets ${}_\gf{\mathfrak T}_{\gf}$.

This graph of bisets can in fact also directly be described out of the
Hubbard tree data. We present below an algorithm that computes
${}_\gf{\mathfrak T}_{\gf}$.

We say that a vertex $v\in \HT^1$ is \emph{essential} if it is the image of
a vertex under the embedding $\HT\hookrightarrow \HT^1$. By definition,
every vertex $v\in \HT$ has a unique essential preimage under $\lambda$. We
say a vertex $v\in \HT^1$ is \emph{critical} if $\ord(v)>1$. Observe that
if $v$ is a critical but non-essential, then $\lambda(v)$ is an edge.

On the level of graphs, $\gf \overset\lambda\leftarrow \mathfrak
T\overset\rho\rightarrow\mathfrak Y$ is the barycentric subdivision of
$\HT\overset\lambda\leftarrow \HT^1\overset p\rightarrow \HT$. For
convenience let us write $\ord(e)=1$ for every edge $e\in \HT$. For every
object $z\in \HT$ set
\[G_z\coloneqq \Z/\ord(z).
\]
In particular, $G_z$ is a non-trivial group if and only if $z$ is an
essential vertex. This constructs the graph of groups $\gf$. The
fundamental group of $\gf$ is isomorphic to a free product of $G_v$ over
all essential vertices $v\in \HT$. For every $z\in \HT^1$ set
\begin{equation}
\label{eq:HT:VertBis}B_z \coloneqq \big(\tfrac{1}{\deg_z(p)}\Z \big)/\ord(\lambda(v))\end{equation}
as a set with $G_{\lambda(z)}$-$G_{p(z)}$ actions given by
\begin{equation}\label{eq:HT:VertBis:Act}
  m\cdot b\cdot n=\begin{cases} m+b+\frac{n}{\deg_z(p)} & \text{if $z$ is an essential or a critical vertex},\\ m+b &\text{otherwise}. 
\end{cases}
\end{equation}

It remains to specify a congruence from $B_e$ into $B_v$ for every edge
$e\in \HT^1$ adjacent to $v\in \HT^1$. Suppose first that $v$ is a
non-essential non-critical vertex. Then $\lambda(e)=\lambda(v)$ and we
define $B_e\to B_v$ to be the natural bijection coming
from~\eqref{eq:HT:VertBis}; this is a well defined congruence between
bisets because the right actions are trivial.

Let $e\in \HT^1$ be an edge adjacent to either an essential or a critical
vertex $v\in \HT^1$. Then the congruence $B_e\to B_v$ is given by
\[b\to b+\lfloor D_e\rfloor
\]
with $D_e$ treated as an element of $\R/\Z\cong \partial D_v$ and $\lfloor
D_e\rfloor$ means ``round $D_e$ down to an element in $B_v$''.

\begin{thm}\label{thm:GrBsOutOfHT}
  The fundamental group of $\gf$ is isomorphic to $G_p$~\eqref{eq:G_p}. The
  fundamental biset of ${}_\gf{\mathfrak T}_{\gf}$ constructed as above is
  isomorphic to $B(p)$. All bisets $B_z, z\in \mathfrak T$ are cyclic.
\end{thm}
\begin{proof}
  By construction, the covering pair $((\C,\ord^1), p,\one)$ is
  homotopic to the covering pair $(X^1,p,\iota)$ along a homotopy path of
  covering pairs. Therefore, by Theorem~\ref{thm:vankampen} the fundamental
  group of $\gf$ is isomorphic to $G_p$; and by
  Theorem~\ref{thm:vankampenbis} the fundamental biset of
  $\pi_1({}_\gf{\mathfrak T}_{\gf})$ is isomorphic to $B(p)$. Let us verify
  that the above algorithm computes ${}_\gf{\mathfrak T}_{\gf}$.

  Let us assume that $0\in \R/\Z\cong \partial D_v$ is the basepoint of
  $\pi_1(D_v)$ for every vertex $v\in \HT$. For every edge $e$ adjacent to
  $v$ set $[0,D_e]\subset \R/\Z\cong \partial D_v$ to be the path
  connecting the basepoint of $D_v$ to the basepoint of $D_e$ (recall that
  $D_e$ is a singleton).

  For every vertex $v\in \HT^1$ there is an $a_v\in \R/\Z$ such that the
  map $p\colon\partial D_v\to\partial D_{p(v)}$ is given by $x\to
  \deg_v(p) x+ a_v$ in the $\R/\Z$-coordinates. Moreover, for every
  essential vertex $v\in \HT^1$ we may choose coordinates $\partial
  D_{v}\cong \R/\Z$ such that the map $\iota\colon\partial D_v\to\partial
  D_{\lambda(v)}$ is the identity in $\R/\Z$-coordinates.

  Clearly, $G_z$ is isomorphic to $\Z/\ord(z)$ for every object $z\in \HT$.

  For every essential vertex $v\in \HT^1$ the biset $B_v$ is computed
  by~\eqref{eq:Bfi:lmm:FibrCorr} as
  \[B_v=\{b\colon[0,1]\to\R/\Z\cong \partial D_{\lambda(v)}\mid b(0)=0 ,
  \deg_v(p)b(1)+a_v=0\}/{\sim}
  \]
  with $G_{\lambda(v)}$ actions are given by pre-concatenation and by
  post-concatenation via
  lifting. Writing
  \[\frac{t}{\deg_v(p)}\coloneqq\left[0,\frac{t-a_v}{\deg_v(p)}\right]\subset
  \R/\Z\cong \partial D_{\lambda (v)}
  \]
  we see that $B_v$ takes the form given in~\eqref{eq:HT:VertBis}
  and~\eqref{eq:HT:VertBis:Act}. The case of non-essential vertices is
  immediate because the right action is trivial.

  If $e\in \HT^1$ is an edge adjacent to a non-essential vertex $v\in
  \HT^1$, then the congruence $B_e\to B_v$
  respects~\eqref{eq:HT:VertBis}. Suppose $e\in \HT^1$ is an edge adjacent
  to an essential vertex $v\in \HT^1$. If $\lambda(e)$ is an edge, then
  $B_e\cong \{D_e\}$ embeds into $B_v$ as $ [0,D_{e} ]\subset \R/\Z
  \cong\partial D_{v}$ followed by the lift of $[\deg_{v}(p) D_{e}+a_v,
  0]\subset \partial D_{p(v)}$ under $p\colon \partial D_v\to\partial
  D_{p(v)}$ starting at $D_{e}\in \partial D_v$. This gives an element
  $\lfloor D_e\rfloor \in B_v$ with $B_v$ viewed in the
  form~\eqref{eq:HT:VertBis}.

  If $\lambda(e)$ is a vertex, then $\lambda(e)=v$ and the $0$-element of
  $B_e$ is $[0,D_{e}]\subset \R/\Z \cong D_v$. Again, this element is
  mapped to $\lfloor D_e \rfloor$ via the congruence $B_e\to B_v$.
\end{proof}

\subsection{(Graphs of) bisets from Hubbard complexes}\label{ss:hubbardcompl}
Consider now a Hubbard complex $p, i\colon H^1\rightrightarrows H^0$. As in
the case of Hubbard trees~\S\ref{ss:hubbardtrees} we blow up each vertex
$v\in H^1\sqcup H^0$ into a closed disc $D_v$ with an orbifold point of
order $\ord(v)$ at the center of $D_v$, and we blow down each edge $e\in
H^1\sqcup H^0$ into a point $D_e$. Observe that all $D_v$ have finite
fundamental groups. Then $p, i\colon H^1\rightrightarrows H^0$ naturally
induces a correspondence $p, i\colon Y^1\rightrightarrows Y^0$ normalized
so that all maps between unit discs are of the form $z\to z^d$. As
in~\S\ref{ss:hubbardtrees} we specify a graph morphism $\lambda\colon
H^1\to H^0$ by $i(D_z)\subset D_{\lambda(z)}$ for every object $z\in
H^1$. Applying Van Kampen's Theorem~\ref{thm:vankampenbis} to the covering
pair $(Y^1,p, i)$ subject to covers $\{D_z\}_{z\in H^1}, \{D_z\}_{z\in
  H^0}$ with maps $p, \lambda\colon H^1\to H^0$ we get a graph of bisets
${}_{\mathfrak Y}{\mathfrak H}_{\mathfrak Y}$. We remark that
${}_{\mathfrak Y}{\mathfrak H}_{\mathfrak Y}$ could be explicitly computed
out of $p,i\colon H^1\to H^0$ in the same way as ${}_{\gf}{\mathfrak
  T}_{\gf}$ was computed out of the Hubbard tree \S\ref{ss:hubbardtrees}.

\begin{thm}\label{thm:GrBsOutOfHC}
  The fundamental group of $\mathfrak Y$ is isomorphic to
  $G_p$~\eqref{eq:G_p}. The fundamental biset of ${}_{\mathfrak
    Y}{\mathfrak H}_{\mathfrak Y}$ is isomorphic to $B(p)$.

  All groups $G_y$, $y\in \mathfrak Y$ and all bisets $B_z, z\in \mathfrak
  H$ are finite.
\end{thm}
\begin{proof}
  The proof is the same as of Theorem~\ref{thm:GrBsOutOfHT}. The claim
  about finiteness of $G_y$ and $B_z$ is straightforward.
\end{proof}

\subsection{(Graphs of) bisets from subdivision rules}\label{ss:SubdRul}
Let us now generalize the setup of~\S\ref{ss:hubbardtrees}. Suppose that
$p\colon S^2\righttoleftarrow$ is a \emph{topological} Thurston map. By a
\emph{subdivision rule} or a ``puzzle partition'' we mean graphs
$G^0\subset G^1\subset S^2$ such that
\begin{itemize}
\item $G^1= p^{-1} (G^0)$;
\item there is a retraction $\iota\colon G^1\to G^0$, with
  $\iota\restrict{G_0}=\one$ such that $G^1$ is homotopic in $S^2$ to
  $\iota(G^1)$ rel the post-critical set;
\item each connected component $S^2\setminus G^1$ contains at most one
  post-critical point.
\end{itemize}
The last condition guarantees that $p,\iota\colon G^1 \rightrightarrows
G^0$ captures all combinatorial information of $p\colon
S^2\righttoleftarrow$.

Assume, furthermore, that $\iota$ can be chosen in such a way that it
maps vertices and edges of $G^1$ into vertices and edges of $G^1$.  We
may construct a graph of bisets
${}_{\mathfrak Z}{\mathfrak G}_{\mathfrak Z}$ associated with
$p,\iota\colon G^1 \rightrightarrows G^0$. As
in~\S\ref{ss:hubbardtrees} we blow up each vertex $v\in G^1\sqcup G^0$
into a closed disc $D_v$ with orbifold point of order $\ord(v)$ at the
center of $D_v$ and we blow down each edge $e\in G^1\sqcup G^0$ into a
point $D_e$. Then $p, \iota \colon G^1\rightrightarrows G^0$ naturally
descents into a correspondence
$p, \iota \colon Z^1\rightrightarrows Z^0$ normalized so that all maps
between unit discs are of the form $z\to z^d$. As
in~\S\ref{ss:hubbardtrees} we specify a graph morphism
$\lambda\colon G^1\to G^0$ by $\iota(D_z)\subset D_{\lambda(z)}$ for
every object $z\in G^1$. Applying Van Kampen's
Theorem~\ref{thm:vankampenbis} to the covering pair $(Z^1,p, i)$
subject to covers $\{D_z\}_{z\in G^1}, \{D_z\}_{z\in G^0}$ with
$p, \lambda\colon G^1\to G^0$ we get the graph of bisets
${}_{\mathfrak Z}{\mathfrak G}_{\mathfrak Z}$.

\begin{thm}\label{thm:GrBsOutOfSR}
  There is a natural epimorphism $\phi\colon \pi_1(\mathfrak Z,*)\to
  G_p$~\eqref{eq:G_p} and there is a natural surjective map $\beta\colon
  \pi_1({}_{\mathfrak Z}{\mathfrak G}_{\mathfrak Z},*)\to B(p)$ such
  that
  \[(\phi, \beta)\colon \pi_1({}_{\mathfrak Z}{\mathfrak G}_{\mathfrak
    Z})\to {}_{G_p}B(p)_{G_p}
  \]
  is a semi-conjugacy respecting combinatorics (see~\S\ref{ss:combequiv}).
\end{thm}
\begin{proof}
  The space $Z^0$ has a natural embedding into $S^2$, unique up to homotopy
  rel the post-critical set, such that $Z^0$ separates post-critical
  points. Therefore, there is a natural epimorphism
  $\phi\colon\pi_1(\mathfrak G)\to G_p$.

  The embedding of $p,\iota\colon Z^1\rightrightarrows Z^0$ into $S^2$
  defines a semi-conjugacy $(\phi, \beta)\colon B(Z^1,p,\iota)\to
  B(p)$. Observe that if a loop $\gamma\in \pi_1(Z^0)$ is trivial in $G_p$,
  then all lifts of $\gamma$ via $p\colon Z^1\to Z^0$ are loops. Therefore,
  $\gamma$ has trivial monodromy action, so $(\phi, \beta)$ respects
  combinatorics.
\end{proof}

\subsection{Tuning and mating}
The \emph{tuning} operation takes as input a topological polynomial
$p$, a periodic cycle $z_0,z_1,\dots,z_n=z_0$ for $p$, and $n$
topological polynomials $q_0,\dots,q_{n-1}$ with
$\deg(q_i)=\deg_{z_i}(p)$ for all $i=0,\dots,n-1$; and produces a
topological polynomial of degree $\deg(p)$.

For each $i$, we give ourselves a set $P_i$ containing the critical values
of $q_{i-1}$ and such that $q_i(P_i)\subset P_{i+1}$. In particular, $P_i$
contains the post-critical set of $q_{i-1}\circ\cdots\circ q_{i+1}\circ
q_i$.

For simplicity, let us assume that the cycle $z_0,z_1,\dots,z_n=z_0$ has no
iterated critical preimages outside of the cycle. This condition could be lifted,
but the construction would require slight modifications.

Up to isotopy, we may assume that around each $z_i$ there is a small closed
topological disc $F_i$ such that $p$ restricts to a covering map $p\colon
F_i\setminus \{z_i\}\to F_{i+1}\setminus \{z_{i+1}\}$ and such that there
are homeomorphisms (known as ``B\"ottcher co\"ordinates'') $\psi_i\colon
\operatorname{int} (F_i)\to\C$ so that $\psi_{i+1}\circ p\circ\psi_i^{-1}$
is the map $z\mapsto z^{\deg_{z_i}(p)}$. Also, up to isotopy, we may assume
that $q_i(z)=z^{\deg (q_i)} + o(z^{\deg (q_i)})$ so that the extension of
$q_i$ to the circle at infinity coincides with the extension of $z\to
z^{\deg_{z_i} (p)} $. We consider the following \emph{tuning} map
$t(p,\{z_i\},\{q_i\})$:
\[t(z)=\begin{cases}
  \psi_{i+1}^{-1}(q_i(\psi_i(z))) & \text{ if $z\in \operatorname{int}(F_i)$ for some $i$},\\
  p(z) & \text{ otherwise}.
\end{cases}\]

Note that if $q_i=z^{\deg_{z_i}(p)}$ for all $i$ then $t$ is isotopic to
$p$.

\begin{thm}\label{thm:tuning}
  Suppose that $t(p,\{z_i\},\{q_i\})$ is the tuning of a complex
  post-critically finite polynomial $p$ with polynomials
  $q_i\colon\C\setminus P_i\dashrightarrow \C\setminus P_{i+1}$ as
  above. Suppose also $\ord(z_i)=\infty$ for all $i$. Let ${}_\gf{\mathfrak
    T}_{\gf}$ denote the graph of bisets constructed out of the Hubbard
  tree of $p$, see Theorem~\ref{thm:GrBsOutOfHT}. Set
  $G_i=\pi_1(\C\setminus P_i,*_i)$ for a basepoint $*_i\gg0$, and let
  $\phi_i\colon G_{z_i}\to G_i$ be the monomorphism defined by identifying
  $\operatorname{int}(F_i)$ with $\C$ via $\psi_i$, so that the circle
  $\partial F_i$ corresponds to the circle at infinity in $\C\setminus
  P_i$. Let $\beta_i\colon B_{z_i}\to B(q_i)$ be the
  $(\phi_i,\phi_{i+1})$-congruence defined by viewing $p\colon\partial
  F_i\to\partial F_{i+1}$ as a map at infinity of $q_i\colon \C\setminus
  P_i \dashrightarrow \C\setminus P_{i+1}$.
  
  For all $i$, replace the cyclic group $G_{z_i}$ at $z_i\in\mathfrak X$ by
  $G_i\coloneqq\pi_1(\C\setminus P_i)$, and replace the regular cyclic biset
  $B_{z_i}$ at $z_i\in\mathfrak B$ by the $G_{i}$-$G_{i+1}$-biset
  $B(q_i)$. Modify accordingly all non-essential bisets: replace all $B_z$
  with $z\in \lambda^{-1}(z_i)\setminus \{z_i\}$ by $G_i \otimes_{G_{z_i}}
  B_z$ and declare for all $z\in p^{-1}(z_{i+1})\setminus \{z_i\}$ the
  right $G_{i+1}$-action on $B_z$ to be trivial. Congruences of edge bisets
  into $B(q_i)$ are given through the congruence $\beta_i\colon B_{z_i}\to
  B(q_i)$. This defines a new graph of bisets $\mathfrak B$ with same
  underlying graph as $\mathfrak T$.

  Then $B(t)$ and $\pi_1(\mathfrak B)$ are isomorphic.
\end{thm}
\begin{proof}
  Let $p,\iota \colon X^1\rightrightarrows X$ be the correspondence
  from~\S\ref{ss:hubbardtrees} associated with the Hubbard tree of $p$. Let
  $t,\iota\colon\widetilde X^1\rightrightarrows \widetilde X$ be the
  correspondence obtained by replacing each $p\colon \operatorname{int}
  (D_{z_i})\to\operatorname{int} (D_{z_{i+1}})$ with $q_i\colon
  \C\setminus P_i \dashrightarrow \C\setminus P_{i+1}$ appropriately glued
  with $\partial D_i$, $\partial D_{i+1}$. Since $t$ is isotopic to
  $t,\iota\colon\widetilde X^1\rightrightarrows \widetilde X$, the biset of
  $t$ is isomorphic to the biset of a covering pair $(\widetilde X^1,
  t,\iota)$ which by Van Kampen's Theorem~\ref{thm:vankampenbis} is
  $\pi_1(\mathfrak B)$.
\end{proof}

This operation may be performed with $p$ a Thurston map, not necessarily a
topological polynomial. As a special case, a topological polynomial has a
critical fixed point $\infty$. The \emph{formal mating} $p\sqcup q$ of two
topological polynomials $p,q$ of same degree is the tuning $t(p,\{\infty\},
q)$. In other words, $p\sqcup q$ is a topological map on a two dimensional
sphere where $p$ acts on the southern hemisphere while $ q$ acts on the
northern hemisphere.

Since the vertex $\infty$ does not belong to the Hubbard tree of $p$,
Theorem~\ref{thm:tuning} cannot apply. We do have a simple description of
the mating in terms of graphs of bisets, though:

\begin{thm}\label{thm:mating}
  Let $p,q$ be two topological polynomials of same degree $d$. Write
  $G_p=\pi_1(\C\setminus P(p),*)$ and $G_{ q}=\pi_1(\C\setminus P(
  q),\dagger)$ with $*,\dagger$ close to $\infty$. Consider the graph of
  groups $\mathfrak X$ with one edge $e$ and two vertices $p, q$. The
  groups at $p, q$ are $G_p,G_{ q}$ respectively; the group at $e$ is $\Z$,
  included in $G_p$ and $G_{ q}$ as the loop around infinity. Consider the
  graph of bisets $\mathfrak B$ with one edge and two vertices. The graph
  maps $\lambda,\rho$ are the identity, the biset $B_e$ on the edge is regular cyclic of degree $d$, and the bisets at the vertices are $B(p),B( q)$
  respectively. The congruences $ B_e\hookrightarrow B(p)$ and $
  B_{\overline e}\hookrightarrow B({ q})$ are defined by viewing $z\to
  z^d\colon S^1\righttoleftarrow$ as a map on the circle at infinity of
  $p\colon \C\setminus P(p)\dashrightarrow \C\setminus P(p) $ and $q\colon
  \C\setminus P(q)\dashrightarrow \C\setminus P(q)$.
  
  Then $\pi_1(\mathfrak B)$ is isomorphic to $B(p\sqcup q)$.
\end{thm}
\begin{proof}
  Applying Theorem~\ref{thm:vankampenbis} to $p\sqcup q$ subject to the
  cover consisting of the closed southern hemisphere, the closed northern
  hemisphere and the equator, we obtain a graph of bisets isomorphic to the
  barycentric subdivision of $\mathfrak B$.
\end{proof}

\subsection{Laminations}\label{ss:laminations}
Consider a complex polynomial $p$ of degree $d$. The Fatou component
around $\infty$ admits a B\"ottcher parameterization
$\phi\colon F_\infty\to\{|z|<1\}$ such that
$\phi(p(z))=\phi(z)^d$. Since $J(p)=\partial F_\infty$, every element
of $J(p)$ may be described as $\lim_{r\to1}\phi^{-1}(re^{i\theta})$
for some (non-unique) angle $\theta$; this gives a surjective map
$c_p\colon \{|z|=1\}\to J(p)$ encoding the Julia set. In fact, the
biset of $p$ may be read from the base-$d$ expansion of the
\emph{kernel} of $c_p$, namely the equivalence relation
$\Xi_p=\{(z_1,z_2)\mid c_p(z_1)=c_p(z_2)\}$.

The equivalence relation may be presented by a \emph{lamination} of
the disk $\{|z|\le1\}$: the disjoint collection of subsets of the
closed disk, called \emph{leaves}, that are convex hulls (in the
hyperbolic metric, say) of equivalence classes of $\Xi_p$.

\def\myarc#1#2{\path[name path=a] ({#1}:1) -- ($({#1}:1)!10cm!270:(0,0)$);
  \path[name path=b] ({#2}:1) -- ($({#2}:1)!10cm!90:(0,0)$);
  \draw[name intersections={of=a and b,by=t}] (t) let \p1 = ($(t)-({#1}:1)$) in
  circle ({veclen(\x1,\y1)});
}

\begin{center}
  \begin{tikzpicture}[scale=3]
    \draw (0,0) circle (1cm);
    \clip (0,0) circle (1cm);
    \begin{scope}[very thick]\myarc{120}{240}\end{scope}
    \myarc{300}{60}
    \foreach\j/\k in {1/360,2/720,4/1440,8/2880,16/5760} {
      \foreach\i in {150,330,...,\k} {\myarc{\i/\j}{(\i+60)/\j}};
    }
  \end{tikzpicture}
\end{center}

The Julia set $J(p)$ is the quotient of the circle by the relation
$\Xi_p$, so the filled-in Julia set is the quotient of the disk
obtained by contracting all leaves of the lamination to points. The
dynamics on the circle is the ``multiply the angle by $d$'' map. The
van Kampen theorem~\ref{thm:vankampen} may therefore be applied to the
covering consisting of the boundary circle and the leaves. In fact, if
$p$ is post-critically finite it suffices to consider a finite
collection of leaves. We consider, as a simple illustration, the
Basilica map $z^2-1$ obtained from the circle by pinching the
above lamination.

The space under consideration consists of a circle $C$ and an arc $A$
connecting the points $x=\exp(2i\pi/3)$ and $y=\exp(4i\pi/3)$, see
Figure~\ref{Fig:GenLamBas}. We fix basepoints $*=1$ on $C$ and $\dagger=x$
on $A$.
\begin{figure}
\label{Fig:GenLamBas}
\begin{center}
  \begin{tikzpicture}
    \begin{scope}[scale=2]
      \node[above] at (60:1) {$x_1$};
      \node[above] at (120:1) {$y_1$};
      \node[below] at (240:1) {$x_2$};
      \node[below] at (300:1) {$y_2$};
      \node[right] at (1,0) {$*_1$};
      \node[left] at (-1,0) {$*_2$};
      \node[right] at (-0.27,0) {$\dagger_1$};
      \node[left] at (0.27,0) {$\dagger_2$};
      \draw[very thick] (1,0) arc (0:180:1);
      \draw (0,0) circle (1cm);
      \clip (0,0) circle (1cm);
      \myarc{120}{240}
      \myarc{300}{60}
    \end{scope}
    \draw[thick,->] (2,1) -- node[above] {$f$} (4,1);
    \draw[thick,->] (2,-1) -- node[below] {$i$} (4,-1);
    \begin{scope}[xshift=6cm,scale=2]
      \node[above] at (120:1) {$x$};
      \node[below] at (240:1) {$y$};
      \node[right] at (1,0) {$*$};
      \node[right] at (-0.27,0) {$\dagger$};
      \draw (0,0) circle (1cm);
      \clip (0,0) circle (1cm);
      \myarc{120}{240}
    \end{scope}
  \end{tikzpicture}
\end{center}
\caption{The correspondence generating the lamination of $z^2-1$. Note that the two preimages of $x$ are $x_1,x_2=\pm\sqrt x$, and
similarly for $y$.}
\end{figure}
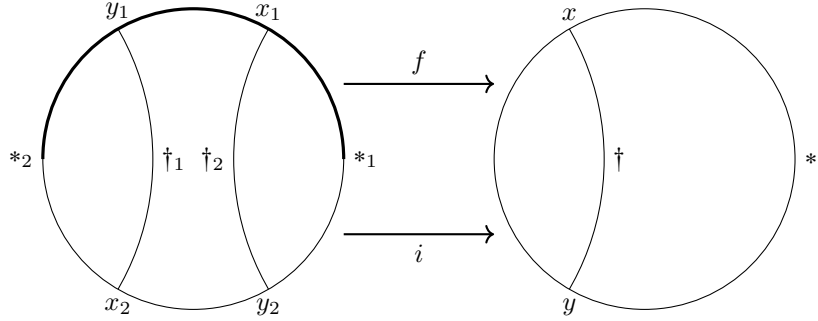
 
Let us denote by $S$, $T$, $U$ the bisets of $f,i\colon C\rightrightarrows
C$, $f,i\colon A\rightrightarrows A$, and $f\colon-A\to A$, $i\colon-A\to
C$ respectively.  We thus have the following graph of bisets $\mathfrak B$:
\begin{center}
  \begin{tikzpicture}
    \begin{scope}
      \node (A1) at (0,0) {$T$};
      \node (C) at (2,0) {$S$};
      \node (A2) at (4,0) {$U$};
      \draw (C) edge [bend right] node[above] {$y_1$} (A1)
               edge [bend left] node[below] {$x_2$} (A1)
               edge [bend right] node[below] {$y_2$} (A2)
               edge [bend left] node[above] {$x_1$} (A2);
    \end{scope}
    \begin{scope}[xshift=6cm]
      \node (C) at (3,0) {$\pi_1(C,*)$};
      \node (A) at (0,0) {$\pi_1(A,\dagger)$};
      \draw (C) edge [bend right] node[above] {$x$} (A)
                edge [bend left] node[below] {$y$} (A);
    \end{scope}
  \end{tikzpicture}
\end{center}
in which $S={}_{2\Z}\Z_\Z$ is the regular cyclic biset encoding the doubling map on
$C$ and the other bisets $T,U$ are trivial. The maps $\rho,\lambda$ are
given by
\[\rho(T)=\rho(U)=\lambda(T)=\pi_1(A,\dagger),\qquad\rho(S)=\lambda(S)=\lambda(U)=\pi_1(C,*).\]
The structure of $\mathfrak B$ will have been completely given when we
fix bases of $S,T,U$ and describe the edge bisets $x_1,y_1,x_2,y_2$
as maps from $T$ or $U$ into $S$. Let us write
$H=\pi_1(C,*)=\langle t\rangle$ and $1=\pi_1(A,\dagger)$. Then
\[{}_HS=H\times\{1,2\},\quad{}_1T=1\times\{3\},\quad{}_H U=H\times\{4\}.\]
Recall that elements of a biset $B(f,i,*'',*')$ are written, in their
most general form~\eqref{eq:lem:cpbiset}, as $(\alpha,p,\beta)$ for
paths $\alpha,\beta$ with
$\alpha(0)=*'',\alpha(1)=i(p),\beta(0)=f(p),\beta(1)=*'$. We fix paths
$p,x,\ell,y$ starting at $*=1$, turning counterclockwise on $C$, and
ending respectively at $\exp(2i\pi/6),x,*_2,y$. We also write
$\epsilon$ for any constant path. In this notation, the bases
$1,2,3,4$ are respectively

\[1=(\epsilon,*_1,\epsilon),\quad 2=(\ell,*_2,\epsilon),\quad
3=(\epsilon,\dagger_1,\epsilon),\quad 4=(p,\dagger_2,\epsilon).\]
We may now view $y_1,x_2$ as maps $T\to S$ and $x_1,y_2$ as maps
$U\to S$. They are given by
\begin{xalignat*}{2}
  y_1(3)&=(x,\dagger_1,y^{-1})=1, & x_1(4)&=(p,\dagger_2,x^{-1})=1,\\
  x_2(3)&=(y,\dagger_1,x^{-1})=2, & y_2(4)&=(p,\dagger_2,y^{-1})=t^{-1}2.
\end{xalignat*}

We are finally ready to compute the fundamental biset $B$ of
$\mathfrak B$. It is a $G$-$G$-biset for the group
$G=H*\langle y x^{-1}\rangle=\langle t,y x^{-1}\rangle$. We may keep the
basis $\{1,2\}$ of $S$, so that ${}_GB=G\times\{1,2\}$, and we have,
just as in $S$,
\[1\cdot t=2,\quad 2\cdot t=t\cdot 1.\]
We now compute
\begin{gather*}
  1\cdot y x^{-1}=y_1(3)y x^{-1}=x3x^{-1}=x y^{-1}x_2(3)=x y^{-1}\cdot 2,\\
  2\cdot y x^{-1}=t y_2(4)y x^{-1}=t4x^{-1}=t x_1(4)=t\cdot 1.
\end{gather*}
In this manner, we formally recover the presentation of the Basilica
biset from the graph of bisets $\mathfrak B$, using the relations
$y_1(3)y=x3$ etc.\ appearing in Definition~\ref{defn:p1biset}.

There is another sort of mating, obtained directly from the Julia
sets. The \emph{geometric mating} of two polynomials $p,q$ of same
degree is the map obtained from
$z\mapsto z^2\colon(\C\cup\{\infty\})\selfmap$ through the quotient of
$\{|z|\le 1\}$ by the lamination of $p$ and through the quotient of
$\{|z|\ge 1\}$ by the lamination of $q$ embedded into the outer disc
via the map $z\to\frac1z$. In case both $p$ and $q$ are
post-critically finite quadratic polynomials that are not in the
conjugate limbs of the Mandelbrot set, the formal and geometric
matings of $p$ and $q$ are isotopic. See~\cite{cheritat:shishikura}
for a discussion on the various types of mating.

\begin{thm}
  Let $p$ be a post-critically finite quadratic polynomial that lies
  outside of the Basilica limb.  Write $H=\pi_1(\C\setminus P(p),*)$. Then
  the biset of $(z^2-1)\sqcup p$ is the fundamental biset of the following
  graph of bisets:
\begin{center}
  \begin{tikzpicture}
    \begin{scope}
      \node (A1) at (0,0) {$T$};
      \node (C) at (2,0) {$B(p)$};
      \node (A2) at (4,0) {$U$};
      \draw (C) edge [bend right] node[above] {$y_1$} (A1)
               edge [bend left] node[below] {$x_2$} (A1)
               edge [bend right] node[below] {$y_2$} (A2)
               edge [bend left] node[above] {$x_1$} (A2);
    \end{scope}
    \begin{scope}[xshift=6cm]
      \node (C) at (3,0) {$H$};
      \node (A) at (0,0) {$1$};
      \draw (C) edge [bend right] node[above] {$x$} (A)
                edge [bend left] node[below] {$y$} (A);
    \end{scope}
  \end{tikzpicture}
\end{center}
with the same maps $\rho,\lambda$ and inclusions as above.
\end{thm}
\begin{proof}
  The map $(z^2-1)\sqcup p$ is homotopic to the map shown on
  Figure~\ref{Fig:GenLamBas} with $\{|z|\ge 1\}$ replaced by the filled-in
  Julia set of $p$. The statement becomes a corollary of
  Theorem~\ref{thm:vankampenbis}.
\end{proof}

\end{document}